\documentclass[english]{article}
\usepackage[T1]{fontenc}
\usepackage[latin9]{inputenc}
\usepackage{verbatim}
\usepackage{amsmath}
\usepackage{amsthm}
\usepackage{amssymb}
\PassOptionsToPackage{normalem}{ulem}
\usepackage{ulem}

\makeatletter
\numberwithin{equation}{section}
\numberwithin{figure}{section}
\newcommand{\lyxaddress}[1]{
\par {\raggedright #1
\vspace{1.4em}
\noindent\par}
}
\theoremstyle{plain}
\newtheorem{thm}{\protect\theoremname}[section]
  \theoremstyle{plain}
  \newtheorem{lem}[thm]{\protect\lemmaname}
  \theoremstyle{remark}
  \newtheorem{rem}[thm]{\protect\remarkname}
  \theoremstyle{plain}
  \newtheorem{prop}[thm]{\protect\propositionname}
  \theoremstyle{definition}
  \newtheorem{defn}[thm]{\protect\definitionname}
  \theoremstyle{plain}
  \newtheorem{fact}[thm]{\protect\factname}
  \theoremstyle{plain}
  \newtheorem{cor}[thm]{\protect\corollaryname}

\usepackage{a4wide}

\usepackage{bbm}
\usepackage{enumerate}

\allowdisplaybreaks

\newcommand{\R}{\mathbb{R}}
\newcommand{\p}{\mathbb{P}}
\newcommand{\E}{\mathbb{E}}
\newcommand{\q}{\mathbb{Q}}
\newcommand{\var}{\mathrm{Var}}
\newcommand{\cov}{\mathrm{Cov}}

\newcommand{\N}{\mathbb{N}}

\newcommand{\proba}{{\mathcal P}}
\newcommand{\meas}{{\mathcal M}^{+}}
\newcommand{\1}{\mathbbm{1}}

\newcommand{\cV}{\mathcal{V}}
\newcommand{\cW}{\mathcal{W}}
\newcommand{\tdR}{\widetilde{R}^\Delta}
\newcommand{\preD}{\underset{\Delta}{\preceq}}
\newcommand{\preDc}{\underset{\Delta,c}{\preceq}}
\newcommand{\convn}{\underset{n\rightarrow +\infty}{\longrightarrow}}

\newtheorem{hypothesis}{Hypothesis}

\date{\today}

\AtBeginDocument{
  
}

\makeatother

\usepackage{babel}
  \providecommand{\corollaryname}{Corollary}
  \providecommand{\definitionname}{Definition}
  \providecommand{\factname}{Fact}
  \providecommand{\lemmaname}{Lemma}
  \providecommand{\propositionname}{Proposition}
  \providecommand{\remarkname}{Remark}
\providecommand{\theoremname}{Theorem}

\begin{document}

\title{Stability of the optimal filter in continuous time: beyond the Bene\v{s}
filter}

\author{Bui, Van Bien \& Rubenthaler, Sylvain}
\maketitle

\lyxaddress{Laboratoire J. A. Dieudonné, Université de Nice Sophia Antipolis,
Parc Valrose, 06108 Nice cedex 2, FRANCE (\texttt{bienmaths@gmail.com},
\texttt{rubentha@unice.fr})}
\begin{abstract}
We are interested in the optimal filter in a continuous time setting.
We want to show that the optimal filter is stable with respect to
its initial condition. We reduce the problem to a discrete time setting
and apply truncation techniques coming from \cite{oudjane-rubenthaler-2005}.
Due to the continuous time setting, we need a new technique to solve
the problem. In the end, we show that the forgetting rate is at least
a power of the time $t$. The results can be re-used to prove the
stability in time of a numerical approximation of the optimal filter. 
\end{abstract}
filtering, signal detection, inference from stochastic processes.

\section{Introduction\label{sec:Introduction}}

\subsection{Exposition of the problem}

We are given a probability space $(\Omega,\mathcal{F},\p)$. We are
interested in the processes $(X_{t})_{t\geq0}$ and $(Y_{t})_{t\geq0}$,
solutions of the following SDE's in $\R$
\[
X_{t}=X_{0}+\int_{0}^{t}f(X_{s})ds+V_{t}\,,
\]
\[
Y_{t}=\int_{0}^{t}hX_{s}ds+W_{t}\,\text{(}h\neq0\text{)}\,,\,
\]
~where $V$, $W$ are two independent standard Brownian motions,
$X_{0}$ is a random variable in $\R$, of law $\pi_{0}$. We set
$(\mathcal{F}_{t})_{t\geq0}$ to be the filtration associated to $(V_{t},W_{t})$.
For $t\geq0$, we call optimal filter at time $t$ the law of $X_{t}$
knowing $(Y_{s})_{0\leq s\leq t}$, and we denote this law by $\pi_{t}$.
Let $\tau>1$, this parameter will be adjusted later. For any $t>0$,
we set $Q_{t}$ to be the transition kernel of the Markov chain $(X_{kt})_{k\geq0}$.
We set $Q=Q_{\tau}$. 

Before going further, we mention that we believe the results we have
could be transposed to processes in $\R^{d}$ for any $d$. This would
result in a more technical paper and we chose not to pursue this goal. 

\begin{hypothesis} \label{hyp:de-base} We suppose that $f$ is $C^{1}$
and that $\Vert f\Vert_{\infty}$, $\Vert f'\Vert_{\infty}$ are bounded
by a constant $M$. \end{hypothesis}

\begin{hypothesis}\label{hyp:sur-pi_0}We suppose that there exists
positive constants $v_{1}$, $v_{2}$ such that, for all nonnegative
$\Delta$, 
\[
\pi_{0}([-\Delta,\Delta]^{\complement})\leq v_{1}e^{-v_{2}\Delta^{2}}\,.
\]
 (This means that the tail of $\pi_{0}$ is sub-Gaussian.)

\end{hypothesis}

We are interested in the stability of $(\pi_{t})_{t\geq0}$ with respect
to its initial condition. As explained below in Equation (\ref{eq:kallianpur-striebel-01}),
for all $t$, $\pi_{t}$ can be written as a functional of $(Y_{s})_{0\leq s\leq t}$
and $\pi_{0}$. Suppose now we plug a probability $\pi_{0}'$ instead
of $\pi_{0}$ into this functional, we then obtain what is called
a ``wrongly initialized filter'' $\pi'_{t}$ (Equation (\ref{eq:kallianpur-striebel-02})).
It is natural to wonder whether $\pi_{t}-\pi'_{t}$ goes to zero when
$t$ goes to infinity in any sense, and at which rate it does so.
We would then say that the filter $(\pi_{t})$ is stable with respect
to its initial condition. We give such a result with a rate of convergence,
in Theorem \ref{thm:stabilite}, under Hypothesis \ref{hyp:de-base}
and \ref{hyp:sur-pi_0}. This question has been answered for more
general processes $(X_{t})$ and $(Y_{t})$ evolving in continuous
time. Here is a brief review and the existing results. We stress the
differences with our setting (see \cite{crisan-rozovskii-2011} for
other references).
\begin{itemize}
\item There is a proof of stability with respect to the initial condition
in \cite{atar-zeitouni-1997}. In this paper, the process $(X_{t})$
has to stay in a compact space. The rate is exponential. 
\item There are stability results in \cite{atar-1998}, \cite{kleptsyna-veretennikov-2008}.
There, the process $(X_{t})$ has to satisfy some ergodicity conditions.
The rate is exponential. 
\item In \cite{baxendale-chigansky-liptser-2004,chigansky-van-handel-2007},
the process $(X_{t})$ has to take values in a finite state space. 
\item The article \cite{stannat-2005} is the closest to our results. There
are some assumptions on the drift coefficient $f(\dots)$ (see Section
1.3 of \cite{stannat-2005} , $f$ would have to be of the form $\varphi'/\varphi$,
$\varphi\in\mathcal{C}^{2}(\R)$), which are different from ours.
The rate is exponential.
\end{itemize}
In the case of an exponential rate, the filter would be called ``exponentially
stable'' (with respect to its initial condition). The widespread
idea is that exponential stability induces that a numerical approximation
of the optimal filter would not deteriorate in time. Such an approximation
is usually based on a time-recursive computation and it is believed
that exponential stability will prevent an accumulation of errors.
In order to use a stability result in a proof concerning a numerical
scheme, one might like the distance between $\pi_{t}$ and $\pi'_{t}$
to be expressed in term of the distance between $\pi_{0}$ and $\pi_{0}'$,
and there is no such result, at least when the time is continuous.
We do not prove such a result in this paper. We explain below that
there is another way to prove that a numerical approximation of the
optimal filter does not deteriorate in time.

Again, our aim in this paper is to show stability in such a way that
the results can be used in a proof that a numerical scheme remains
good uniformly in time. \uline{This has not been done in the literature
we cited above}. We follow \cite{oudjane-rubenthaler-2005} by introducing
a ``robust filter'' restricted to compact spaces. We show that this
filter remains close to the optimal filter uniformly in time and this
is enough to prove the stability of the optimal filter with respect
to its initial condition. As in \cite{oudjane-rubenthaler-2005},
we do not show that the optimal filter is exponentially stable, nor
can we write the dependency in $\pi_{0}$, $\pi_{0}'$ in the stability
result. However, in a future work, we will use the stability properties
of the robust filter to show that there exists a numerical approximation
that remains uniformly good in time.

In the case where  $f$ satisfies 
\begin{equation}
f'(x)+f(x)^{2}+h^{2}x^{2}=P(x)\,,\label{eq:Benes}
\end{equation}
where $P(x)$ is a second-order polynomial with positive leading-order
coefficient, then $\pi_{t}$ is called the Bene\v{s} filter (see
\cite{benes-1981}, \cite{bain-crisan-2009}) and there exists an
explicit formula for the density of $\pi_{t}$, for all $t$. The
study of the Bene\v{s} filter is developed in \cite{ocone-1999}.
Under Hypothesis \ref{hyp:de-base}, we have (\ref{eq:Benes}) if
and only if $f(x)=f(0)(1+xf(0))^{-1}$. What we present here is a
case in the neighborhood of the Bene\v{s} filter.

We present the main results and our strategy in Section \ref{subsec:Main-results-and}
below. Before this, we have to go through some definitions.

For all $t>0$, the law of $Y_{0:t}$ under $\p$ and conditionally
on $X_{0}$, $X_{t}$ has a density denoted by~$y_{0:t}\mapsto\psi_{t}(y_{0:t},X_{0},X_{\tau})$
with respect to the Wiener measure (see Lemma \ref{lem:encadrement-potentiel}
below). The Kallianpur-Striebel formula (see \cite{bain-crisan-2009},
p. 57) gives us the following result (see the proof in Section \ref{subsec:Proofs-of-Section-1}).
\begin{lem}
\label{lem:kallianpur-striebel}For all $t>0$ and all bounded continuous
$\varphi$,
\begin{equation}
\pi_{t}(\varphi)=\frac{\int_{\R}\varphi(y)Q_{t}(x,dy)\psi_{t}(Y_{0:t},x,y)\pi_{0}(dx)}{\int_{\R}Q_{t}(x,dy)\psi_{t}(Y_{0:t},x,y)\pi_{0}(dx)}\,.\label{eq:kallianpur-striebel-01}
\end{equation}
\end{lem}

For any probability law $\pi_{0}'$ on $\R$, we define the wrongly
initialized filter (with initial condition $\pi_{0}'$) by, for any
$t>0$,
\begin{equation}
\pi_{t}'(\varphi)=\frac{\int_{\R}\varphi(y)Q_{t}(x,dy)\psi_{t}(Y_{0:t},x,y)\pi_{0}'(dx)}{\int_{\R}Q_{t}(x,dy)\psi_{t}(Y_{0:t},x,y)\pi_{0}'(dx)}\,.\label{eq:kallianpur-striebel-02}
\end{equation}

\subsection{Main results and outline of the proofs\label{subsec:Main-results-and}}

All the results are given under Hypothesis \ref{hyp:de-base} and
\ref{hyp:sur-pi_0}. The first main result is the following ($\Vert\dots\Vert$
is the total variation norm when applied to measures).
\begin{thm}
\label{thm:stabilite} There exists $\nu_{0}>0$ such that, if $\pi_{0}$
and $\pi_{0}'$ are comparable,
\[
\E(\Vert\pi_{t}-\pi'_{t}\Vert)=O(t^{-\nu_{0}})\,,
\]
when $t\rightarrow+\infty$ (see Section \ref{subsec:Notations} for
the definition of comparable).
\end{thm}

\begin{rem}
We have here a polynomial stability whereas filters in discrete time
are usually exponentially stable. We believe this result is not sharp
(mainly because the proof is very convoluted).
\end{rem}

In order to prove this Theorem, we first reduce the problem to a problem
in discrete time by looking at the sequences $(\pi_{0},\pi_{\tau},\pi_{2\tau},\dots)$,
$(\pi'_{0},\pi'_{\tau},\pi'_{2\tau},\dots)$. We then introduce a
truncated filter $(\pi_{k\tau}^{\Delta})_{k\geq0}$ and a wrongly
initialized truncated filter $((\pi')_{k\tau}^{\Delta})_{k\geq0}$,
where $\Delta$ is a truncation parameter. We do not get into more
details since to get a closed definition of the truncated filter,
one would have to look at Equations (\ref{eq:def-truncated-filter}),
(\ref{eq:def-R-Delta}), (\ref{eq:def-operateur-normalise}), (\ref{eq:def-D_k}),
(\ref{eq:def-compact-suite}), (\ref{eq:def-m_k}), (\ref{eq:def-xi_1}),
(\ref{eq:def-p}), (\ref{eq:def-Cz2}), (\ref{eq:A2}), (\ref{eq:def-C1}),
(\ref{eq:def-alpha-etc}), (\ref{eq:def-a-etc}), (\ref{eq:def-lambdas}),
(\ref{eq:def-theta}), (\ref{eq:def-sigma1}), (\ref{eq:def-sigma2})
and Lemma \ref{lem:calcul-var}. The truncated filter can be viewed
as a restriction of $(\pi_{k\tau})_{k\geq0}$ to some compacts. Equation
(\ref{eq:def-truncated-filter}) tells us that 
\[
\pi_{k\tau}^{\Delta}=\overline{R}_{k}^{\Delta}\overline{R}_{k-1}^{\Delta}\dots\overline{R}_{1}^{\Delta}(\pi_{0})\,\text{,}\,\text{for all }k\geq1\,,
\]
for some operators $\overline{R}_{\dots}^{\Delta}.$ We define
\[
(\pi')_{k\tau}^{\Delta}=\overline{R}_{k}^{\Delta}\overline{R}_{k-1}^{\Delta}\dots\overline{R}_{1}^{\Delta}(\pi_{0}')\,\text{,}\,\text{for all }k\geq1\,.
\]
We first prove that the operators $\overline{R}_{\dots}^{\Delta}$
are contracting (see Section \ref{sec:New-formula-for}, Equations
(\ref{eq:def-epsilon}), (\ref{eq:def-epsilon-prime}) for the definitions
of $\epsilon_{\dots}$, $\epsilon'_{\dots}$).
\begin{lem}
\label{lem:synthese-des-resultats}We suppose that $\tau$ and $h$
are chose such that Equation (\ref{eq:borne-theta}) holds. For $n\geq1$
and $k$ in $\{1,2,\dots,n-1\}$, for all $\mu$, $\mu'$ in $\mathcal{P}(\R)$,
\begin{multline*}
\left\Vert \overline{R}_{n}^{\Delta}\overline{R}_{n-1}^{\Delta}\dots\overline{R}_{k+1}^{\Delta}(\mu)-\overline{R}_{n}^{\Delta}\overline{R}_{n-1}^{\Delta}\dots\overline{R}_{k+1}^{\Delta}(\mu')\right\Vert \\
\leq\prod_{i=1}^{\lfloor(n-k)/2\rfloor}(1-\epsilon_{k+2i-1}^{2}(\epsilon'_{k+2i})^{2})\times4\inf\left(1,\frac{\Vert\mu-\mu'\Vert}{(\epsilon'_{k+2})^{2}\epsilon_{k+1}^{4}}\right)\,.
\end{multline*}
\end{lem}

We then write a telescopic sum (Equation (\ref{eq:somme-telescopique-01}))
for the optimal filter
\[
\Vert\pi_{n\tau}-\pi_{n\tau}^{\Delta}\Vert\leq\Vert\pi_{n\tau}-\overline{R}_{n}^{\Delta}(\pi_{(n-1)\tau})\Vert+\sum_{1\leq k\leq n-1}\Vert\overline{R}_{n:k+1}^{\Delta}(\pi_{k\tau})-\overline{R}_{n:k+1}^{\Delta}(\overline{R}_{k}^{\Delta}(\pi_{(k-1)\tau}))\Vert
\]
where, for all $n\geq k$, $\overline{R}_{n:k+1}^{\Delta}=\overline{R}_{n}^{\Delta}\circ\overline{R}_{n-1}^{\Delta}\circ\dots\circ\overline{R}_{k+1}^{\Delta}$
is a composition of operators (and we write a similar sum for the
wrongly initialized filter). Due to the contractivity of the operators
$\overline{R}_{\dots}^{\Delta}$, we are then able to prove that the
truncation error is uniform in time (see Section \ref{subsec:Asymptotics}
for the definition of $\preceq$).
\begin{prop}
\label{prop:approx-par-filtre-robuste}There exist a function of $h$
, $\lambda_{1}'(h)$ and a constant $\tau_{\infty}$ such that, if
$\tau\geq\tau_{\infty}$, we have
\begin{equation}
\sup_{n\geq0}\log(\E(\Vert\pi_{n\tau}-\pi_{n\tau}^{\Delta}\Vert))\preDc-\Delta^{2}\times\lambda_{1}'(h)\,;\label{eq:approx-pi}
\end{equation}
and, if $\pi_{0}$ and $\pi_{0}'$ are comparable, 
\begin{equation}
\sup_{n\geq0}\log(\E(\Vert\pi'_{n\tau}-(\pi')_{n\tau}^{\Delta}\Vert))\preDc-\Delta^{2}\times\lambda_{1}'(h)\,.\label{eq:approx-pi-prime}
\end{equation}
 (See Section \ref{subsec:Notations} for a definition of comparable.)
\end{prop}

The proof of Theorem \ref{thm:stabilite} then start with the inequality
\[
\Vert\pi_{n\tau}-\pi'_{n\tau}\Vert\leq\Vert\pi_{n\tau}-\pi_{n\tau}^{\Delta}\Vert+\Vert\pi_{n\tau}^{\Delta}-(\pi')_{n\tau}^{\Delta}\Vert+\Vert(\pi')_{n\tau}^{\Delta}-\pi'_{n\tau}\Vert\,.
\]
We choose $\Delta$ as a function of $n$ to make sure all the terms
on the right go to zero when $n$ goes to infinity. After some technicalities,
we are able to get rid of the discrete time index $n\tau$ and write
a result for any $t\geq0$.

The second main result of the paper is that these results can be used
to show that a numerical approximation of the filter does not deteriorate
in time. Suppose we have a numerical approximation $\pi_{n\tau}^{N,\Delta}$
of $\pi_{n\tau}^{\Delta}$ for all $n\geq0$ (with $\pi_{0}^{N,\Delta}=\pi_{0}^{\Delta}$)
; $\Delta$ is again the truncation parameter and $N$ is the approximation
parameter (one can think of a number of particles in case of a particle
filter approximation). We write
\[
\Vert\pi_{n\tau}-\pi_{n\tau}^{N,\Delta}\Vert\leq\Vert\pi_{n\tau}-\pi_{n\tau}^{\Delta}\Vert+\Vert\pi_{n\tau}^{\Delta}-\pi_{n\tau}^{N,\Delta}\Vert\,,\,\text{for all }n\geq0\,.
\]
We already have the result that $\Vert\pi_{n\tau}-\pi_{n\tau}^{\Delta}\Vert$
is small when $\Delta$ goes to infinity (Proposition \ref{prop:approx-par-filtre-robuste}).
For the second term, we write a telescopic sum
\[
\Vert\pi_{n\tau}^{N,\Delta}-\pi_{n\tau}^{\Delta}\Vert\leq\Vert\pi_{n\tau}^{N,\Delta}-\overline{R}_{n}^{\Delta}(\pi_{(n-1)\tau}^{N,\Delta})\Vert+\sum_{1\leq k\leq n-1}\Vert\overline{R}_{n:k+1}^{\Delta}(\pi_{k\tau}^{N,\Delta})-\overline{R}_{n:k+1}^{\Delta}(\overline{R}_{k}^{\Delta}(\pi_{(k-1)\tau}^{N,\Delta}))\Vert\,.
\]
Using Lemma \ref{lem:synthese-des-resultats} again, we would be able
to show that $\Vert\pi_{n\tau}^{N,\Delta}-\pi_{n\tau}^{\Delta}\Vert$
remains small uniformly in $n$ (when $N$ is big). Choosing $\Delta$
as a function of $N$ \_ in the same way as in \cite{oudjane-rubenthaler-2005}
\_ would then allow to show that $\Vert\pi_{n\tau}-\pi_{n\tau}^{N,\Delta}\Vert$
remains small uniformly in the time $n$, when $N$ is big. 

The fact that stability results have the potential to be used to show
that a numerical approximation does not deteriorate in time is a novelty
compared to the other papers in the area.

Another novelty lies in the proof of Lemma \ref{lem:synthese-des-resultats}.
In Proposition \ref{prop:representation-avec-U}, we exhibit a representation
of the truncated filter $\overline{R}_{n:1}^{\Delta}(\pi_{0})$ as
the $n$-th point of a Feynman-Kac sequence based on a mixing Markov
kernel (see Section \ref{sec:Definitions-and-useful} for the definition
of the term Feynman-Kac sequence). This is not straightforward because
the observations are given as a process in continuous time. Whatever
transformation we adopt, the problem remains more difficult than in
discrete time. We will write more comments about the novelty of the
proof in Section \ref{sec:New-formula-for} (it requires more notations
to develop on the subject).

The outline of the paper is as follows. In Section \ref{sec:Computation-of},
we compute useful estimates concerning $\psi_{t}$. In Section 3,
we recall useful notions on filtering. In Section 4 and 5, we introduce
the truncated filter and its properties. At the beginning of Section
5, we elaborate on our strategy. In Section 6, we prove that the optimal
filter can be approximated by robust filters uniformly in time (Proposition
\ref{prop:approx-par-filtre-robuste}), and that the optimal filter
is stable with respect to its initial condition (Theorem \ref{thm:stabilite}).

\section{\label{sec:Computation-of}Computations around $\psi_{t}$}

\subsection{Estimation of the transition density and the likelihood}

We begin by bounding the transition density by above and below (see
the proof in Section \ref{subsec:Proofs-of-Section-1}). 
\begin{lem}
\label{lem:encadrement-transition}For all $x,y\in\R$, $Q(x,dy)$
has a density $Q(x,y)$ with respect to the Lebesgue measure and
\[
\frac{e^{-\frac{(y-x)^{2}}{2\tau}}}{\sqrt{2\pi\tau}}e^{-M|y-x|-\tau\left(\frac{M}{2}+\frac{M^{2}}{2}\right)}\leq Q(x,y)\leq\frac{\exp^{-\frac{(y-x)^{2}}{2\tau}}}{\sqrt{2\pi\tau}}e^{M|y-x|+\frac{M\tau}{2}}\,.
\]
\end{lem}

Following~\cite{bain-crisan-2009} (Chapter 6, Section 6.1), we define
a new probability $\widehat{\p}$ by (for all $t\geq0$)
\[
\widehat{V}_{t}=V_{t}+\int_{0}^{t}f(X_{s})ds\,,
\]
 
\[
\left.\frac{d\p}{d\widehat{\p}}\right|_{\mathcal{F}_{t}}=\widehat{Z}_{t}=\exp\left(\int_{0}^{t}f(X_{s})d\widehat{V}_{s}-\frac{1}{2}\int_{0}^{t}f(X_{s})^{2}ds+\int_{0}^{t}hX_{s}dY_{s}-\frac{1}{2}\int_{0}^{t}h^{2}X_{s}^{2}ds\right)
\]
 We define, for all $0\leq s\leq t$, 
\[
Y_{s:t}=(Y_{u})_{s\leq u\leq t}\,.
\]
 We set 
\begin{equation}
\widehat{\psi}_{t}(y_{0:t},x_{0},x_{1})=\E^{\widehat{\p}}\left(\left.\exp\left(\int_{0}^{t}hX_{s}dY_{s}-\frac{1}{2}\int_{0}^{t}(hX_{s})^{2}ds\right)\right|X_{0}=x_{0},X_{t}=x_{1},Y_{0:t}=y_{0:t}\right)\,.\label{eq:def-psi-chapeau}
\end{equation}
We set
\[
\psi=\psi_{\tau}\,,\,\widehat{\psi}=\widehat{\psi}_{\tau}\,.
\]
We have the following bounds for the likelihood (see the proof in
Section \ref{subsec:Proofs-of-Section-1}).
\begin{lem}
\label{lem:encadrement-potentiel}For all $t>0$, the law of $Y_{0:t}$
under $\p$ and conditionally on $X_{0}$, $X_{t}$ has a density
denoted by~$y_{0:t}\mapsto\psi_{t}(y_{0:t},X_{0},X_{\tau})$ with
respect to the Wiener measure. This density satisfies, for all~$x$,~$z$~$\in\R$
and any continuous trajectory $y_{0:t}$ 
\[
\widehat{\psi}_{t}(y_{0:t},x,z)e^{-2M|z-x|-\tau\left(M+\frac{M^{2}}{2}\right)}\leq\psi_{t}(y_{0:t},x,z)\leq e^{2M|z-x|+\tau\left(M+\frac{M^{2}}{2}\right)}\widehat{\psi}_{t}(y_{0:t},x,z)\,.
\]
\end{lem}

\subsection{Change of measure}

Under $\widehat{\p}$, $\widehat{V}$ is a standard Brownian motion.
So, using a standard representation of a Brownian bridge, we can rewrite
$\widehat{\psi}$ as
\begin{multline*}
\widehat{\psi}(y_{0:\tau},x,z)=\E\left(\exp\left(\int_{0}^{\tau}h\left(x\left(1-\frac{s}{\tau}\right)+z\frac{s}{\tau}+\left(B_{s}-\frac{s}{\tau}B_{\tau}\right)\right)dy_{s}\right.\right.\\
\left.\left.-\frac{h^{2}}{2}\int_{0}^{\tau}\left(x\left(1-\frac{s}{\tau}\right)+z\frac{s}{\tau}+\left(B_{s}-\frac{s}{\tau}B_{\tau}\right)\right)^{2}ds\right)\right)\,,
\end{multline*}
where $B$ is a standard Brownian motion (under $\p$). As we want
to compute the above integral, where $B$ is the only random variable
involved, we can suppose that $B$ is adapted to the filtration $\mathcal{F}$.
We have (using the change of variable $s'=s/\tau$ and the scaling
property of the Brownian motion) 
\begin{multline*}
\widehat{\psi}(y_{0:\tau},x,z)=\E\left(\exp\left(\int_{0}^{1}h\left(x(1-s'\right)+zs'+B_{\tau s'}-s'B_{\tau})dy_{\tau s'}\right.\right.\\
\left.\left.-\frac{h^{2}\tau}{2}\int_{0}^{1}\left(x(1-s')+zs'+B_{\tau s'}-s'B_{\tau}\right)^{2}ds'\right)\right)\\
=\E\left(\exp\left(\int_{0}^{1}h(x(1-s')+zs'+\sqrt{\tau}(B_{s'}-s'B_{1}))dy_{\tau s'}\right.\right.\\
\left.\left.-\frac{h^{2}\tau}{2}\int_{0}^{1}(x(1-s')+zs'+\sqrt{\tau}(B_{s'}-s'B_{1}))^{2}ds'\right)\right)
\end{multline*}
In the spirit of~\cite{mansuy-yor-2008} (Section 2.1), we define
a new probability $\mathbb{Q}$ by (for all $t$)
\[
\left.\frac{d\q}{d\p}\right|_{\mathcal{F}_{t}}=\exp\left(-\frac{h^{2}\tau^{2}}{2}\int_{0}^{1}B_{s}^{2}ds-h\tau\int_{0}^{1}B_{s}dB_{s}\right)\,.
\]
By Girsanov's theorem, under the probability $\q$, the process
\begin{equation}
\beta_{t}=B_{t}+\int_{0}^{t}h\tau B_{s}ds\,,\,\forall t\geq0\label{eq:O-U}
\end{equation}
is a Brownian motion. We get
\begin{multline}
\widehat{\psi}(y_{0:\tau},x,z)=\exp\left(\int_{0}^{1}h(x(1-s)+zs)dy_{\tau s}-\frac{h^{2}\tau}{2}\int_{0}^{1}(x(1-s)+zs)^{2}ds\right)\\
\times\E^{\q}\left(\exp\left(\int_{0}^{1}h\sqrt{\tau}(B_{s}-sB_{1})dy_{\tau s}-h^{2}\tau^{3/2}\int_{0}^{1}(x(1-s)+zs)(B_{s}-sB_{1})ds\right.\right.\\
\left.\left.-\frac{h^{2}\tau^{2}}{2}\int_{0}^{1}s^{2}B_{1}^{2}-2sB_{s}B_{1}ds+h\tau\int_{0}^{1}B_{s}dB_{s}\right)\right)\,.\label{eq:psi-chapeau-01}
\end{multline}
Using the integration by parts formula, we can rewrite the last expectation
as 
\begin{multline}
\E^{\q}\left(\exp\left(-h\sqrt{\tau}\int_{0}^{1}(y_{\tau s}-\int_{0}^{1}y_{\tau u}du)dB_{s}+h^{2}\tau^{3/2}\int_{0}^{1}\left(-x\frac{(1-s)^{2}}{2}+z\frac{s^{2}}{2}+\frac{x}{6}-\frac{z}{6}\right)dB_{s}\right.\right.\\
\left.\left.+h^{2}\tau^{2}B_{1}\left(\frac{B_{1}}{3}-\int_{0}^{1}\frac{s^{2}}{2}dB_{s}\right)+h\tau\left(\frac{B_{1}^{2}}{2}-\frac{1}{2}\right)\right)\right)=\\
\E^{\q}\left(\exp\left(-h\sqrt{\tau}\int_{0}^{1}(y_{\tau s}-\int_{0}^{1}y_{\tau u}du)dB_{s}+h^{2}\tau^{3/2}x\int_{0}^{1}sdB_{s}+\frac{h^{2}\tau^{3/2}}{2}(z-x)\int_{0}^{1}s^{2}dB_{s}\right.\right.\\
\left.\left.-h^{2}\tau^{3/2}\left(\frac{x}{3}+\frac{z}{6}\right)B_{1}+\left(\frac{h^{2}\tau^{2}}{3}+\frac{h\tau}{2}\right)B_{1}^{2}-\frac{h^{2}\tau^{2}}{2}B_{1}\int_{0}^{1}s^{2}dB_{s}-\frac{h\tau}{2}\right)\right)\,.\label{eq:psi-chapeau-02}
\end{multline}

\subsection{Covariances computation}

The last expectation contains an exponential of a polynomial of degree
$2$ of $4$ Gaussians:
\[
G_{1}=B_{1}\,,\,G_{2}=\int_{0}^{1}sdB_{s}\,,\,G_{3}=\int_{0}^{1}s^{2}dB_{s}\,,\,G_{4}=\int_{0}^{1}\left(y_{\tau s}-\int_{0}^{1}y_{\tau u}du\right)dB_{s}\,.
\]
So this expectation can be expressed as a function of the covariance
matrix of these Gaussians. We compute here the covariances, which
do not depend on $y_{0:\tau}$. We set 
\begin{equation}
\theta=h\tau\,.\label{eq:def-theta}
\end{equation}
\begin{lem}
\label{lem:calcul-var}We have:
\[
\var^{\q}(G_{1})=\frac{1-e^{-2\theta}}{2\theta}\,,
\]
\[
\var^{\q}(G_{2})=\left(1+\frac{1}{\theta}\right)^{2}\frac{(1-e^{-2\theta})}{2\theta}+\frac{1}{\theta^{2}}-\left(\frac{2}{\theta^{2}}+\frac{2}{\theta^{3}}\right)(1-e^{-\theta})\,,
\]
\begin{multline*}
\var^{\q}(G_{3})=\left(1+\frac{2}{\theta}+\frac{2}{\theta^{2}}\right)^{2}\frac{(1-e^{-2\theta})}{2\theta}+\left(\frac{2}{\theta}+\frac{2}{\theta^{2}}\right)^{3}\frac{\theta}{6}-\frac{8}{6\theta^{5}}-\frac{4}{\theta^{2}}\left(1+\frac{2}{\theta}+\frac{2}{\theta^{2}}\right)
\end{multline*}
\[
\cov^{\q}(G_{1},G_{2})=\left(\frac{1}{2\theta}+\frac{1}{2\theta^{2}}\right)(1-e^{-2\theta})+\frac{e^{-\theta}-1}{\theta^{2}}\,,
\]
\[
\cov^{\q}(G_{1},G_{3})=\left(\frac{1}{2\theta}+\frac{1}{\theta^{2}}+\frac{1}{\theta^{3}}\right)(1-e^{-2\theta})-\frac{2}{\theta^{2}}\,,
\]
\[
\cov^{\q}(G_{2},G_{3})=\left(1+\frac{1}{\theta}\right)\left(\frac{1}{2\theta}+\frac{1}{\theta^{2}}+\frac{1}{\theta^{3}}\right)(1-e^{-2\theta})-\left(\frac{1}{\theta^{2}}+\frac{2}{\theta^{3}}+\frac{2}{\theta^{4}}\right)(1-e^{-\theta})-\frac{1}{\theta^{2}}\,.
\]
\end{lem}

See the proof in Section \ref{subsec:Proofs-of-Section-psi}

Let $U_{1},U_{2},U_{3},U_{4}$ be i.i.d. of law $\mathcal{N}(0,1)$.
We can find $\alpha,\beta,\gamma,a,b,c,\lambda_{1},\lambda_{2},\lambda_{3},\lambda_{4}\in\R$
such that (under $\mathbb{Q}$) 
\[
\left(\begin{array}{c}
G_{1}\\
G_{3}\\
G_{2}\\
G_{4}
\end{array}\right)\overset{\mbox{law}}{=}\left(\begin{array}{cccc}
\alpha U_{1}\\
\beta U_{1} & +\gamma U_{2}\\
aU_{1} & +bU_{2} & +cU_{3}\\
\lambda_{1}U_{1} & +\lambda_{2}U_{2} & +\lambda_{3}U_{3} & +\lambda_{4}U_{4}
\end{array}\right)\,.
\]
(There is no mistake here, we do intend to look at the vector $(G_{1},G_{3},G_{2},G_{4})$.)
Indeed, we take 
\begin{equation}
\alpha=\sqrt{\var^{\q}(G_{1})}\,,\,\beta=\frac{\cov^{\q}(G_{1},G_{3})}{\alpha}\,,\,\gamma=\sqrt{\var^{\q}(G_{3})-\beta^{2}}\,,\label{eq:def-alpha-etc}
\end{equation}
\begin{equation}
a=\frac{\cov^{\q}(G_{1},G_{2})}{\alpha}\,,\,b=\frac{\cov^{\q}(G_{2},G_{3})-a\beta}{\gamma}\,,\,c=\sqrt{\var^{\q}(G_{2})-a^{2}-b^{2}}\,.\label{eq:def-a-etc}
\end{equation}
And we find $\lambda_{1},\dots,\lambda_{4}$ by solving 
\begin{equation}
\left\{ \begin{array}{cccccc}
\alpha\lambda_{1} &  &  &  & = & \cov^{\q}(G_{1},G_{4})\\
\beta\lambda_{1} & +\gamma\lambda_{2} &  &  & = & \cov^{\q}(G_{3},G_{4})\\
a\lambda_{1} & +b\lambda_{2} & +c\lambda_{3} &  & = & \cov^{\q}(G_{2},G_{4})\\
\lambda_{1}^{2} & +\lambda_{2}^{2} & +\lambda_{3}^{2} & +\lambda_{4}^{2} & = & \var^{\q}(G_{4})\,.
\end{array}\right.\label{eq:def-lambdas}
\end{equation}
We observe that $\alpha,\beta,\gamma,a,b,c$ can be written explicitly
in terms of the parameters of the problem.

\subsection{Integral computation}

The last part of (\ref{eq:psi-chapeau-02}) is equal to
\begin{multline}
\E^{\q}\left(\exp\left(\left(\frac{\theta^{2}}{3}+\frac{\theta}{2}\right)G_{1}^{2}-\frac{\theta^{2}}{2}G_{1}G_{3}-h^{2}\tau^{3/2}\left(\frac{x}{3}+\frac{z}{6}\right)G_{1}\right.\right.\\
\left.\left.\left.+h^{2}\tau^{3/2}xG_{2}+\frac{h^{2}\tau^{3/2}}{2}(z-x)G_{3}-h\sqrt{\tau}G_{4}-\frac{\theta}{2}\right)\right)\right)=\mbox{\ensuremath{}}\\
\int_{u_{1,}\dots,u_{4}\in\R}\exp\left(\left(\frac{\theta^{2}}{3}+\frac{\theta}{2}\right)\alpha^{2}u_{1}^{2}-\frac{\theta^{2}}{2}\alpha u_{1}(\beta u_{1}+\gamma u_{2})-h^{2}\tau^{3/2}\left(\frac{x}{3}+\frac{z}{6}\right)\alpha u_{1}\right.\\
\left.+h^{2}\tau^{3/2}x(au_{1}+bu_{2}+cu_{3})+\frac{h^{2}\tau^{3/2}}{2}(z-x)(\beta u_{1}+\gamma u_{2})-h\sqrt{\tau}(\lambda_{1}u_{1}+\dots+\lambda_{4}u_{4})-\frac{\theta}{2}\right)\\
\frac{\exp\left(-\frac{(u_{1}^{2}+\dots+u_{4}^{2})}{2}\right)}{(2\pi)^{2}}du_{1}\dots du_{4}=\\
\int_{u_{1,}\dots,u_{4}\in\R}\exp\left\{ -\frac{1}{2\sigma_{1}^{2}}\left[u_{1}-\sigma_{1}^{2}\left(-\frac{\theta^{2}\alpha\gamma}{2}u_{2}-h^{2}\tau^{3/2}\left(\frac{x}{3}+\frac{z}{6}\right)\alpha+h^{2}\tau^{3/2}xa\right.\right.\right.\\
\left.\left.+\frac{h^{2}\tau^{3/2}}{2}(z-x)\beta-h\sqrt{\tau}\lambda_{1}\right)\right]^{2}\\
+\frac{\sigma_{1}^{2}}{2}\left[-\frac{\theta^{2}\alpha\gamma}{2}u_{2}-h^{2}\tau^{3/2}\left(\frac{x}{3}+\frac{z}{6}\right)\alpha+h^{2}\tau^{3/2}xa+\frac{h^{2}\tau^{3/2}}{2}(z-x)\beta-h\sqrt{\tau}\lambda_{1}\right]^{2}\\
+h^{2}\tau^{3/2}x(bu_{2}+cu_{3})+\frac{h^{2}\tau^{3/2}}{2}(z-x)\gamma u_{2}-h\sqrt{\tau}(\lambda_{2}u_{2}+\dots+\lambda_{4}u_{4})-\frac{\theta}{2}\\
\left.-\frac{(u_{2}^{2}+\dots+u_{4}^{2})}{2}\right\} \frac{1}{(2\pi)^{2}}du_{1}\dots du_{4}\,,\label{eq:psi-chapeau-03}
\end{multline}
where
\begin{equation}
\sigma_{1}^{2}=\left(2\left(-\left(\frac{\theta^{2}}{3}+\frac{\theta}{2}\right)\alpha^{2}+\frac{\theta^{2}\alpha\beta}{2}+\frac{1}{2}\right)\right)^{-1}\,.\label{eq:def-sigma1}
\end{equation}
As the above expectation is finite then $\sigma_{1}^{2}$ is well
defined. We set 
\begin{equation}
m_{1}=\sigma_{1}^{2}\left(-h^{2}\tau^{3/2}\left(\frac{x}{3}+\frac{z}{6}\right)\alpha+h^{2}\tau^{3/2}xa+\frac{h^{2}\tau^{3/2}}{2}(z-x)\beta-h\sqrt{\tau}\lambda_{1}\right)\,.\label{eq:def-m1}
\end{equation}
The above expectation (\ref{eq:psi-chapeau-03}) is equal to:
\begin{multline}
\int_{u_{1,}\dots,u_{4}\in\R}\exp\left(-\frac{1}{2\sigma_{1}^{2}}\left[u_{1}+\sigma_{1}^{2}\frac{\theta^{2}\alpha\gamma}{2}u_{2}-m_{1}\right]^{2}\right.\\
+\left(\frac{\sigma_{1}^{2}\theta^{2}\alpha\gamma}{2}\right)^{2}u_{2}^{2}\frac{1}{2\sigma_{1}^{2}}+\frac{m_{1}^{2}}{2\sigma_{1}^{2}}-\frac{1}{2\sigma_{1}^{2}}\times2\left(\frac{\sigma_{1}^{2}\theta^{2}\alpha\gamma}{2}\right)m_{1}u_{2}\\
+h^{2}\tau^{3/2}x(bu_{2}+cu_{3})+\frac{h^{2}\tau^{3/2}}{2}(z-x)\gamma u_{2}\\
\left.-h\sqrt{\tau}(\lambda_{2}u_{2}+\lambda_{3}u_{3}+\lambda_{4}u_{4})-\frac{(u_{2}^{2}+u_{3}^{2}+u_{4}^{2})}{2}-\frac{\theta}{2}\right)\frac{1}{(2\pi)^{2}}du_{1}\dots du_{4}=\\
\int_{u_{1,}\dots,u_{4}\in\R}\exp\left(-\frac{1}{2\sigma_{1}^{2}}\left[u_{1}+\sigma_{1}^{2}\frac{\theta^{2}\alpha\gamma}{2}u_{2}-m_{1}\right]^{2}\right.\\
-\frac{1}{2\sigma_{2}^{2}}\left[u_{2}-\sigma_{2}^{2}\left(h^{2}\tau^{3/2}xb+\frac{h^{2}\tau^{3/2}}{2}(z-x)\gamma-h\sqrt{\tau}\lambda_{2}-\frac{\theta^{2}\alpha\gamma m_{1}}{2}\right)\right]^{2}\\
\left.+\frac{m_{2}^{2}}{2\sigma_{2}^{2}}+\frac{m_{1}^{2}}{2\sigma_{1}^{2}}+h^{2}\tau^{3/2}xcu_{3}-h\sqrt{\tau}(\lambda_{3}u_{3}+\lambda_{4}u_{4})-\frac{(u_{3}^{2}+u_{4}^{2})}{2}-\frac{\theta}{2}\right)\frac{1}{(2\pi)^{2}}\,,\label{eq:psi-chapeau-04}
\end{multline}
where 
\begin{equation}
\sigma_{2}^{2}=\left(2\left(-\frac{\sigma_{1}^{2}\theta^{4}\alpha^{2}\gamma^{2}}{8}+\frac{1}{2}\right)\right)^{-1}\,,\label{eq:def-sigma2}
\end{equation}
and 
\begin{equation}
m_{2}=\sigma_{2}^{2}\left(h^{2}\tau^{3/2}xb+\frac{h^{2}\tau^{3/2}}{2}(z-x)\gamma-h\sqrt{\tau}\lambda_{2}-\frac{\theta^{2}\alpha\gamma}{2}m_{1}\right)\,.\label{eq:def-m2}
\end{equation}
Then (\ref{eq:psi-chapeau-04}) is equal to:
\begin{multline}
\int_{u_{1,}\dots,u_{4}\in\R}\exp\left(-\frac{1}{2\sigma_{1}^{2}}\left[u_{1}+\sigma_{1}^{2}\frac{\theta^{2}\alpha\gamma}{2}u_{2}-m_{1}\right]^{2}-\frac{1}{2\sigma_{2}^{2}}\left[u_{2}-m_{2}\right]^{2}+\frac{m_{2}^{2}}{2\sigma_{2}^{2}}+\frac{m_{1}^{2}}{2\sigma_{1}^{2}}\right.\\
-\frac{1}{2}\left[u_{3}-h^{2}\tau^{3/2}cx+h\sqrt{\tau}\lambda_{3}\right]^{2}-\frac{1}{2}\left[u_{4}+h\sqrt{\tau}\lambda_{4}\right]^{2}\\
\left.+\frac{1}{2}(-h^{2}\tau^{3/2}cx+h\sqrt{\tau}\lambda_{3})^{2}+\frac{1}{2}(-h\sqrt{\tau}\lambda_{4})^{2}-\frac{\theta}{2}\right)\frac{1}{(2\pi)^{2}}du_{1}\dots du_{4}=\\
\sigma_{1}\sigma_{2}\exp\left(\frac{m_{1}^{2}}{2\sigma_{1}^{2}}+\frac{m_{2}^{2}}{2\sigma_{2}^{2}}+\frac{1}{2}(-h^{2}\tau^{3/2}cx+h\sqrt{\tau}\lambda_{3})^{2}+\frac{1}{2}(h\sqrt{\tau}\lambda_{4})^{2}-\frac{\theta}{2}\right)\label{eq:psi-chapeau-05}
\end{multline}

\subsection{\label{subsec:Asymptotics}Asymptotic $\tau\rightarrow+\infty$}

From (\ref{eq:psi-chapeau-01}), (\ref{eq:psi-chapeau-02}), (\ref{eq:psi-chapeau-03}),
(\ref{eq:psi-chapeau-04}), (\ref{eq:psi-chapeau-05}), we see that
$\widehat{\psi}(y_{0:\tau},x,z)\propto\exp(P(x,z))$ with $P$ a polynomial
of degree $2$ in $x$, $z$. Let us write $-A_{2}(\theta)$ for the
coefficient of $x^{2}$ in $P$, $-B_{2}(\theta)$ for the coefficient
of $z^{2}$ in $P$, $C_{1}(\theta)$ for the coefficient of $xz$
in $P$, $A_{1}(\theta)$ for the coefficient of $x$ in $P$, $B_{1}(\theta)$
for the coefficient of $z$ in $P$ and $C_{0}(\theta)$ for the ``constant''
coefficient. We will write $A_{1}^{y_{0:\tau}}(\theta)=A_{1}(y_{0:\tau},\theta)$
(or simply $A_{1}^{y_{0:\tau}}$), etc, when in want of stressing
the dependency in $y$. When there will be no ambiguity, we will drop
the $y_{.}$ superscript. The coefficients $A_{2}^{y_{0:\tau}}$,
$B_{2}^{y_{0:\tau}}$, $C_{1}^{y_{0:\tau}}$ do not depend on $y$
as it will be seen below. We have
\begin{equation}
\widehat{\psi}(y_{0:\tau},x,z)=\sigma_{1}\sigma_{2}\exp\left(-A_{2}x^{2}-B_{2}z^{2}+A_{1}^{y_{0:\tau}}x+B_{1}^{y_{0:\tau}}z+C_{1}^{y_{0:\tau}}xz+C_{0}^{y_{0:\tau}}\right)\,.\label{eq:pol-psi-chapeau}
\end{equation}
 We are interested in the limit $\tau\rightarrow+\infty$, with $h$
being fixed (or equivalently $\theta\rightarrow+\infty$ with $h$
being fixed).
\begin{lem}
\label{lem:asymptotics-of-A-B-C}We have
\[
A_{2}(\theta)\underset{\theta\rightarrow+\infty}{\longrightarrow}\frac{h}{2}\,,\,B_{2}(\theta)\underset{\theta\rightarrow+\infty}{\longrightarrow}\frac{h}{2}\,,\,C_{1}(\theta)=-\frac{3h}{2\theta}+o\left(\frac{1}{\theta}\right)\,.
\]
\end{lem}

The proof relies partly on a symbolic computation program (see Section
\ref{subsec:Proofs-of-Section-psi}).

Let us set, for all $s\leq t$,
\[
\mathcal{W}_{s,t}=\sup_{(s_{1},s_{2})\in[s,t]}|W_{s_{1}}-W_{s_{2}}|\,,\,\mathcal{V}_{s,t}=\sup_{(s_{1},s_{2})\in[s,t]}|V_{s_{1}}-V_{s_{2}}|\,.
\]
\begin{defn}
\label{def:preceq}Suppose we have functions $f_{1}$, $f_{2}$ going
from some set $F$ to $\R$. We write 
\[
f_{1}\preceq f_{2}
\]
if there exists a constant $B$ in $\R_{+}$, which does not depend
on the parameters of our problem, such that $f_{1}(z)\leq Bf_{2}(z)$,
for all $z$ in $F$. \\
In the particular case when we are dealing with functions of a parameter
$\Delta\in\R$ %
, we write 
\[
f_{1}\underset{\Delta}{\preceq}f_{2}\mbox{ or }f_{1}(\Delta)\preD f_{2}(\Delta)
\]
 if there exists a constant $B_{1}$ in $\R_{+}$, which does not
depend on the parameters of our problem, and a constant $\Delta_{0}$,
which may depend on the parameters of our problem, such that 
\[
\Delta\geq\Delta_{0}\Rightarrow f_{1}(\Delta)\leq B_{1}f_{2}(\Delta)\,.
\]
If, in addition, $\Delta_{0}$ depends continuously on the parameter
$\tau$, we write 
\[
f_{1}\underset{\Delta,c}{\preceq}f_{2}\,.
\]
\end{defn}

The notation $\underset{\Delta}{\preceq}$ is akin to the notation
$O(\dots)$. It has the advantage that one can single out which asymptotic
we are studying.

We state here (without proof) useful properties concerning the above
Definition. 
\begin{lem}
\label{lem:proprietes-preceq}Suppose we have functions $f$, $f_{1}$,
$f_{2}$, $h_{1}$, $h_{2}$.

\begin{enumerate}
\item If $f\leq f_{1}+f_{2}$ and $f_{1}\preceq f_{2}$ then $f\preceq f_{2}$.
\item If $f\leq f_{1}+f_{2}$ and $\log(f_{1})\preD h_{1}$ and $\log(f_{2})\preD h_{2}$,
with $h_{1}(\Delta),\,h_{2}(\Delta)\underset{\Delta\rightarrow\infty}{\longrightarrow}-\infty$,
then $\log(f)\preD\sup(h_{1},h_{2})$. 
\item \label{enu:lemme-technique-iii}If we have $f_{1}\preDc f_{2}$, say
for $\tau\geq\tau_{0}$ ($\tau_{0}>0$), then there exists a constant
$B_{1}$ and a continuous function $\Delta_{0}$ such that, for all
$\tau\geq\tau_{0}$ and $\Delta\geq\Delta_{0}(\tau)$, $f_{1}(\Delta)\leq B_{1}f_{2}(\Delta)$.
In particular, for any $\tau_{1}>\tau_{0}$, if $\tau\in[\tau_{0},\tau_{1}]$
and $\Delta\ge\sup_{t\in[\tau_{0},\tau_{1}]}\Delta_{0}(t)$ then $f_{1}(\Delta)\leq B_{1}f_{2}(\Delta)$.
\end{enumerate}
\end{lem}

We have the following ``variational'' bounds on $B_{1}$ (see the
proof in Section \ref{subsec:Proofs-of-Section-psi}).
\begin{lem}
\label{lem:variation-const}For all $k\in\N$, 
\[
\left|B_{1}(Y_{k\tau:(k+1)\tau},\theta)-B_{1}(Y_{(k+1)\tau:(k+2)\tau},\theta)\right|\preceq Mh\tau+h\mathcal{V}_{k\tau,(k+2)\tau}+(h+\frac{1}{\tau})\mathcal{W}_{k\tau,(k+2)\tau}\,,
\]
and
\[
|B_{1}(Y_{0:\tau},\theta)|\preceq Mh\tau+h\mathcal{V}_{0,2\tau}+(h+\frac{1}{\tau})\mathcal{W}_{0,2\tau}\,.
\]
\end{lem}

\section{Definitions and useful notions\label{sec:Definitions-and-useful} }

 We follow here the ideas of \cite{oudjane-rubenthaler-2005}.

\subsection{Notations\label{subsec:Notations}}

We state here notations and definitions that will be useful throughout
the paper.
\begin{itemize}
\item The set $\R$, $\R^{2}$ are endowed, respectively, with $\mathcal{B}(\R)$,
$\mathcal{B}(\R^{2})$, their Borel tribes.
\item The elements of $\R^{2}$ are treated like line vectors. If $(x,z)\in\R^{2}$
and $\delta>0$, $B((x,z),\epsilon)=\{(x',z')\in\R^{2}\,:\,\sqrt{(x-x')^{2}+(z-z')^{2}}<\epsilon\}$
(the ball of center $(x,z)$ and radius $\epsilon$). The superscript
$T$ denotes the transposition. For example, if $(x,z)\in\R^{2}$,
then $(x,z)^{T}$ is a column vector.
\item The set of probability distributions on a measurable space $(E,\mathcal{F})$
and the set of nonnegative measures on $(E,\mathcal{F})$ are denoted
by $\mathcal{P}(E)$ and $\mathcal{M}^{+}(E)$ respectively. We write
$\mathcal{C}(E)$ for the set of continuous function on a topological
space $E$ and $\mathcal{C}_{b}^{+}(E)$ for the set of bounded, continuous,
nonnegative functions on $E$.
\item When applied to measures, $\Vert\dots\Vert$ stands for the total
variation norm (for $\mu$, $\nu$ probabilities on a measurable space
$(F,\mathcal{F})$, $\Vert\mu-\nu\Vert=\sup_{A\in\mathcal{F}}|\mu(A)-\nu(A)|$).
\item For any nonnegative kernel $K$ on a measurable space $E$ and any
$\mu\in\mathcal{M}^{+}(E)$, we set
\[
K\mu(dv')=\int_{E}\mu(dv)K(v,dv')\,.
\]
\item If we have a sequence of nonnegative kernels $K_{1}$, $K_{2}$, \ldots{}
on some measured spaces $E_{1}$, $E_{2}$, \ldots{} (meaning that
for all $i\geq1$, $x\in E_{i-1}$, $K_{i}(x,.)$ is a nonnegative
measure on $E_{i}$, then for all $i<j$, we define the kernel 
\begin{multline*}
K_{i+1:j}(x_{i},dx_{j})=\int_{x_{i+1}\in E_{i+1}}\dots\int_{x_{j-1}\in E_{j-1}}K_{i+1}(x_{i},dx_{i+1})K_{i+2}(x_{i+1},dx_{i+2})\dots K_{j}(x_{j-1},dx_{j})\,.
\end{multline*}
\item For any measurable space $E$ and any nonzero $\mu\in\mathcal{M}^{+}(E)$,
we define the normalized nonnegative measure,
\[
\overline{\mu}=\frac{1}{\mu(E)}\mu\,.
\]
\item For any measurable space $E$ and any nonnegative kernel $K$ defined
on $E$, we define the normalized nonnegative nonlinear operator $\overline{K}$
on $\meas(E)$, taking values in $\proba(E)$, and defined by 
\begin{equation}
\overline{K}(\mu)=\frac{K\,\mu}{(K\,\mu)(E)}=\frac{K\,\overline{\mu}}{(K\,\overline{\mu})(E)}=\overline{K}(\overline{\mu})\,,\label{eq:def-operateur-normalise}
\end{equation}
 for any $\mu\in\meas(E)$ such that $K\,\mu(E)\neq0$, and defined
by $\overline{K}(\mu)=0$ otherwise. 
\item A kernel $K$ from a measurable space $E_{1}$ into another measurable
space $E_{2}$ is said to be $\epsilon$-mixing ($\epsilon\in(0,1)$)
if there exists $\lambda$ in $\mathcal{M}^{+}(E_{2})$ and $\epsilon_{1},\epsilon_{2}>0$
such that, for all $x_{1}$ in $E_{1}$,
\[
\epsilon_{1}\lambda(.)\leq K(x_{1},.)\leq\frac{1}{\epsilon_{2}}\lambda(.)\,\mbox{with }\epsilon_{1}\epsilon_{2}=\epsilon^{2}\,.
\]
This property implies that, for all $A$, $\mu$, $\overline{K}(\mu)(A)\geq\epsilon^{2}\overline{\lambda}(A)$.
If $\overline{K}$ is Markov, this last inequality implies that $\overline{K}$
is $(1-\epsilon^{2})$-contracting in total variation (see \cite{del-moral-guionnet-2001}
p. 161-162 for more details):
\[
\forall\mu,\nu\in\mathcal{P}(E),\,\Vert\overline{K}(\mu)-\overline{K}(\nu)\Vert\leq(1-\epsilon^{2})\Vert\mu-\nu\Vert\,.
\]
\item For any measurable space $E$ and any $\psi:E\rightarrow\R^{+}$ (measurable)
and $\mu\in\meas(E)$, we set 
\[
\left\langle \mu,\psi\right\rangle =\int_{E}\psi(x)\mu(dx)\,.
\]
If in addition, $\left\langle \mbox{\ensuremath{\mu},\ensuremath{\psi}}\right\rangle >0$,
we set 
\[
\psi\bullet\mu(dv)=\frac{1}{\left\langle \mu,\psi\right\rangle }\times\psi(v)\mu(dv)\,.
\]
\item For $\mu$ and $\mu'$ in $\mathcal{M}^{+}(E)$ ($(E,\mathcal{F})$
being a measurable space), we say that $\mu$ and $\mu'$ are comparable
if there exist positive constants $a$ and $b$ such that, for all
$A\in\mathcal{F}$, 
\[
a\mu'(A)\leq\mu(A)\leq b\mu'(A)\,.
\]
We then define the Hilbert metric between $\mu$ and $\mu'$ by
\[
h(\mu,\mu')=\log\left(\frac{\sup_{A\in\mathcal{F}:\mu'(A)>0}\frac{\mu(A)}{\mu'(A)}}{\inf_{A\in\mathcal{F}:\mu'(A)>0}\frac{\mu(A)}{\mu'(A)}}\right)\,.
\]
It is easily seen (see for instance \cite{oudjane-2000}, Chapter
2) that, for any nonnegative kernal $K$ and any $A$ in $\mathcal{F}$,
\begin{eqnarray}
h(K\mu,K\mu') & \leq & h(\mu,\mu')\,,\label{eq:hilbert-prop-01}\\
h(\overline{\mu},\overline{\mu'}) & \leq & h(\mu,\mu')\,,\label{eq:hilbert-prop-02}
\end{eqnarray}
\begin{equation}
\exp(-h(\mu,\mu'))\leq\frac{\mu(A)}{\mu'(A)}\leq\exp(h(\mu,\mu'))\mbox{, if }\mu'(A)>0\,.\label{eq:encadrement-rapport-mesures}
\end{equation}
In addition, we have the following relation with the total variation
norm:
\begin{equation}
\Vert\overline{\mu}-\overline{\mu'}\Vert\leq\frac{2}{\log(3)}h(\mu,\mu')\,.\label{eq:inegalite-tv-hilbert}
\end{equation}
\item We set $\widetilde{Q}$ to be the transition of the chain $(X_{k\tau},X_{(k+1)\tau})_{k\geq0}$.
\item We write $\propto$ between two quantities if they are equal up to
a multiplicative constant.
\item For $\psi\,:\,\R^{2}\rightarrow\R$, we write $\psi(0,.)$ for the
function such that, for all $x$ in $\R$, $\psi(0,.)(x)=\psi(0,x)$. 
\item We write $\lambda_{W}$ for the Wiener measure on $\mathcal{C}([0,\tau])$. 
\end{itemize}
We suppose here that the observation $(Y_{t})_{t\geq0}$ is fixed.
For $k\in\N^{*}$ and $x,z\in\R$, we define
\begin{equation}
\psi_{k}(x,z)=\psi(Y_{(k-1)\tau:k\tau},x,z)\label{eq:def-psi_k}
\end{equation}
(the density $\psi$ is defined in Lemma \ref{lem:encadrement-potentiel}).
For $x_{1}\in\R$, $x_{2}\in\R$ and $n\in\N^{*}$, we introduce the
nonnegative kernel
\begin{equation}
R_{n}(x_{1},dx_{2})=\psi_{n}(x_{1},x_{2})Q(x_{1},dx_{2})\,.\label{eq:def-R_n}
\end{equation}
Using the above notations, we now have, for all $n\in\N^{*}$, and
for all probability law $\pi_{0}'$ (with $(\pi'_{t})_{t\geq0}$ defined
in Equation (\ref{eq:kallianpur-striebel-02})) 
\[
\pi_{n\tau}=\overline{R}_{n}(\pi_{(n-1)\tau})\,,\,\pi'_{n\tau}=\overline{R}_{n}(\pi'_{(n-1)\tau})
\]
and for $0<m<n$, 
\begin{equation}
\pi_{n\tau}=\overline{R}_{n}\overline{R}_{n-1}\dots\overline{R}_{m}(\pi_{(m-1)\tau})\,,\,\pi'_{n\tau}=\overline{R}_{n}\overline{R}_{n-1}\dots\overline{R}_{m}(\pi'_{(m-1)\tau}).\label{eq:F-K-sequence}
\end{equation}

\subsection{\label{subsec:Representation-of-the}Representation of the optimal
filter as the law of a Markov chain}

Regardless of the notations of the other sections, we suppose we have
a Markov chain $(\mathfrak{X}_{n})_{n\geq0}$ taking values in measured
spaces $E_{0}$, $E_{1}$, \ldots{}, with nonnegative  kernels $\mathfrak{Q}_{1}$,
$\mathfrak{Q}_{2}$, \ldots{} (it might be a non-homogeneous Markov
chain). Suppose we have potentials $\Psi_{1}:\,E_{1}\rightarrow\R_{+}$,
$\Psi_{2}:\,E_{2}\rightarrow\R_{+}$, \ldots{} (measurable functions
with values in $\R_{+}$) and a law $\eta_{0}$ on $E_{0}$. We are
interested in the sequence of probability measures $(\eta_{k})_{k\geq1}$
, respectively on $E_{1}$, $E_{2}$, \ldots{}, defined by
\begin{equation}
\forall k\geq1\,,\,\forall\varphi\in\mathcal{C}_{b}^{+}(E_{k})\,,\,\eta_{k}(f)=\frac{\E_{\eta_{0}}(\varphi(\mathfrak{X}_{k})\prod_{1\leq i\leq k}\Psi_{i}(\mathfrak{X}_{i}))}{\E_{\eta_{0}}(\prod_{1\leq i\leq k}\Psi_{i}(\mathfrak{X}_{i}))}\,,\label{eq:def-F-K}
\end{equation}
where $\eta_{0}\in\mathcal{P}(E_{0})$ and the index $\eta_{0}$ means
we start with $\mathfrak{X}_{0}$ of law $\eta_{0}$ . We will say
that $(\eta_{k})_{k\geq0}$ is a Feynman-Kac sequence on $(E_{k})_{k\geq0}$
based on the transitions $(\mathfrak{Q}_{k})_{k\geq1}$, the potentials
$(\Psi_{k})_{k\geq1}$ and the initial law $\eta_{0}$. Suppose we
have another law $\eta_{0}'$, we then set
\[
\forall k\geq1\,,\,\forall\varphi\in\mathcal{C}_{b}^{+}(E_{k})\,,\,\eta_{k}'(f)=\frac{\E_{\eta_{0}'}(\varphi(\mathfrak{X}_{k})\prod_{1\leq i\leq k}\Psi_{i}(\mathfrak{X}_{i}))}{\E_{\eta_{0}'}(\prod_{1\leq i\leq k}\Psi_{i}(\mathfrak{X}_{i}))}\,.
\]
If the functions $\Psi_{k}$'s are likelihood associated to observations
of a Markov chain with transitions $\mathfrak{Q}_{1}$, $\mathfrak{Q}_{2}$,
\dots{} and initial law $\eta_{0}$, then the measures $\eta_{k}$'s
are optimal filters. We fix $n\geq1$. We would like to express $\eta_{n}$
as the marginal law of some Markov process. We will do so using ideas
from \cite{del-moral-guionnet-2001}. We set, for all $k\in\{1,\dots,n\}$,
\begin{equation}
\mathfrak{R}_{k}(x,dx')=\Psi_{k}(x')\mathfrak{Q}_{k}(x,dx')\,.\label{eq:def-Rfrak_k}
\end{equation}
 We suppose that, for all $k$, $\mathfrak{R}_{k}$ is $\epsilon_{k}$-mixing
(notice that $\mathfrak{Q}_{k}$ being $\epsilon_{k}$-mixing implies
that $\mathfrak{R}_{k}$ is $\epsilon_{k}$-mixing).%
{} By a simple recursion, we have, for all $n$, 
\begin{equation}
\overline{\mathfrak{R}_{1:n}}=\overline{\mathfrak{R}}_{n}\overline{\mathfrak{R}}_{n-1}\dots\overline{\mathfrak{R}}_{1}\,.\label{eq:recursive-representation}
\end{equation}
We set, for all $k\in\{0,1,\dots,n-1\}$,
\[
\Psi_{n|k}(x)=\int_{x_{k+1}\in E_{k+1}}\dots\int_{x_{n}\in E_{n}}\mathfrak{R}_{k+1}(x,dx_{k+1})\prod_{k+2\leq i\leq n}\mathfrak{R}_{i}(x_{i-1},dx_{i})\,.
\]
If $k=n$, we set $\Psi_{n|n}$ to be constant equal to $1$. For
$k\in\{1,2,\dots,n\}$, we set 
\[
\mathfrak{S}_{n|k}(x,dx')=\frac{\Psi_{n|k+1}(x')}{\Psi_{n|k}(x)}\mathfrak{R}_{k+1}(x,dx')\,.
\]
From \cite{del-moral-guionnet-2001}, we get the following result
(a simple proof can also be found in \cite{oudjane-rubenthaler-2005},
Proposition 3.1).
\begin{prop}
\label{prop:representation-FK-sequence}The operators $(\mathfrak{S}_{n|k})_{0\leq k\leq n-1}$
are Markov kernels. For all $k\in\{0,\dots,n-1\}$, $\mathfrak{S}_{n|k}$
is $\epsilon_{k+1}$-mixing. We have
\[
\eta_{n}=\mathfrak{S}_{n|n-1}\mathfrak{S}_{n|n-1}\dots\mathfrak{S}_{n|0}(\Psi_{n|0}\bullet\eta_{0})\,,
\]
\[
\eta_{n}'=\mathfrak{S}_{n|n-1}\mathfrak{S}_{n|n-1}\dots\mathfrak{S}_{n|0}(\Psi_{n|0}\bullet\eta_{0}')\,,
\]

and
\[
\Vert\eta_{n}-\eta_{n}'\Vert\leq\prod_{1\leq k\leq n}(1-\epsilon_{k}^{2})\times\Vert\Psi_{n|0}\bullet\eta_{0}-\Psi_{n|0}\bullet\eta_{0}'\Vert\,.
\]
\end{prop}

Following the computations of \cite{oudjane-rubenthaler-2005}, p.
434 (or \cite{oudjane-2000}, p.66), we have, for all measurable $\Psi\,:\,\R^{2}\rightarrow\R^{+}$,
\begin{equation}
\Vert\Psi\bullet\eta_{0}-\Psi\bullet\eta_{0}'\Vert\leq2\int_{x\in E_{0}}\frac{\Psi(x)}{\langle\eta_{0},\Psi\rangle}|\eta_{0}-\eta_{0}'|(dx)\leq2\inf\left(1,\frac{\Vert\Psi\Vert_{\infty}}{\langle\eta_{0},\Psi\rangle}\Vert\eta_{0}-\eta_{0}'\Vert\right)\,.\label{eq:maj-erreur-locale-01}
\end{equation}
For all $x$ in $E_{0}$, as $\mathfrak{R}_{1}$ is $\epsilon_{1}$-mixing,
\begin{eqnarray}
\frac{\Psi_{n|0}(x)}{\langle\eta_{0},\Psi_{n|0}\rangle} & = & \frac{\int_{z\in E_{1}}\mathfrak{R}_{2:n}(z,E_{n})\mathfrak{R}_{1}(x,dz)}{\int_{y\in E_{0}}\int_{z\in E_{1}}\mathfrak{R}_{2:n}(z,E_{n})\mathfrak{R}_{1}(y,dz)\eta_{0}(dy)}\nonumber \\
\mbox{(for some \ensuremath{\epsilon_{1}'}, \ensuremath{\epsilon_{1}''}, \ensuremath{\lambda_{1}}with \ensuremath{\epsilon_{1}'\epsilon_{1}''=\epsilon_{1}^{2}})} & \leq & \frac{\int_{z\in E_{1}}\mathfrak{R}_{2:n}(z,E_{n})\frac{1}{\epsilon'_{1}}\lambda_{1}(dz)}{\int_{z\in E_{1}}\mathfrak{R}_{2:n}(z,E_{n})\epsilon''_{1}\lambda_{1}(dz)}=\frac{1}{\epsilon_{1}^{2}}\,.\label{eq:maj-err-locale-02}
\end{eqnarray}

\section{\label{sec:Robust-approximation-of}Truncated filter}

We introduce in this section a filter built with truncated likelihoods.
We will call it truncated filter or robust filter, the use of the
adjective ``robust'' refers to the fact that it has stability properties
(it appears in the proof of Proposition \ref{prop:approx-par-filtre-robuste}
below).

\subsection{Integrals of the potential \label{subsec:Truncation}}

We are look here at $\widehat{\psi}(y_{0:\tau},x,z)$ for some $x$,
$z$ in $\R$ and a fixed observation $y_{0:\tau}$ between $0$ and
$\tau$. All what will be said is also true for observations between
$k\tau$ and $(k+1)\tau$ (for any $k$ in $\N$). From Equations
(\ref{eq:def-C1}), (\ref{eq:def-Cz}), (\ref{eq:def-Cx}), we see
that $A_{1}^{y_{0:\tau}}$, $B_{1}^{y_{0:\tau}}$ are polynomials
of degree $1$ in $\lambda_{1}$, \ldots{}, $\lambda_{3}$ and that
$C_{1}$ does not depend on $y_{0:\tau}$.  We fix $x$ and $z$ in
$\R$.%
{} Recall that, by Equation (\ref{eq:def-lambdas}), Lemmas \ref{lem:representation-O-U}
and \ref{lem:integral-O-U}, $\lambda_{1}$, $\lambda_{2}$, $\lambda_{3}$
are functions of $y_{0:\tau}$ and that they can be expressed as integrals
of deterministic functions against $dy_{\tau s}$ (this requires some
integrations by parts, see Lemma \ref{lem:lambda-integral}). %
{} Under the law $\widetilde{\p}$ (defined in Equation (\ref{eq:def-Ptilde})),
conditioned to $X_{0}=x$, $X_{\tau}=z$, we can write
\begin{equation}
X_{t}=\left(1-\frac{s}{\tau}\right)x+\frac{s}{\tau}z+\widetilde{B}_{s}-\frac{s}{\tau}\widetilde{B}_{\tau}\,,\label{eq:X-sous-Ptilde}
\end{equation}
where $(\widetilde{B}_{s})_{s\geq0}$ is a Brownian motion, independent
of $W$. And we can write, using integration by parts and Equation
(\ref{eq:def-Cx}) (see Lemma \ref{lem:A-as-integral}),
\begin{equation}
A_{1}^{Y_{0:\tau}}=\int_{0}^{\tau}(f_{1}(s)dW_{s}+f_{2}(s)dX_{s})\,,\label{eq:A-as-gaussian}
\end{equation}
for some deterministic functions $f_{1}$, $f_{2}$ (and the same
goes for $B_{1}^{Y_{0:\tau}}$, see Equation (\ref{eq:def-Cz})). 

We set 
\begin{equation}
p_{1,1}=\left(1-\frac{C_{1}^{2}}{4A_{2}B_{2}}\right)_{+}^{1/2}\,,\,p_{2,1}=-\frac{C_{1}}{2B_{2}}\,,\,p_{2,2}=1\,.\label{eq:def-p}
\end{equation}
\begin{fact}
\label{fact:We-now-fix}We now fix a parameter $\iota\in(1/2,1)$.
From now on, we suppose that the parameter $\tau$ is chosen such
that 
\begin{equation}
\theta\geq1\,,\,A_{2}\geq\frac{h}{4}\,,\,B_{2}\geq\frac{h}{4}\,,C_{1}\leq\frac{h}{8}\,,\,\frac{1}{1+p_{2,1}}-(2+\frac{4}{h})B_{2}p_{2,1}\theta^{1-\iota}>0\,,\,p_{1,1}>\frac{1}{2}\,,\,|p_{2,1}|\leq\frac{1}{2}\,.\label{eq:borne-theta}
\end{equation}
\end{fact}

This is possible because of Lemma \ref{lem:asymptotics-of-A-B-C}
and because this Lemma implies: $p_{2,1}=O(\theta^{-1})$, $p_{1,1}\underset{\theta\rightarrow+\infty}{\longrightarrow}1$.

Let us set
\begin{equation}
\kappa=\left[\begin{array}{cc}
A_{2} & -\frac{C_{1}}{2}\\
-\frac{C_{1}}{2} & B_{2}
\end{array}\right]\,.\label{eq:def-kappa}
\end{equation}
If we take
\begin{equation}
P=\left[\begin{array}{cc}
p_{1,1} & 0\\
p_{2,1} & p_{2,2}
\end{array}\right]=\left[\begin{array}{cc}
\left(1-\frac{C_{1}^{2}}{4A_{2}B_{2}}\right)^{1/2} & 0\\
-\frac{C_{1}}{2B_{2}} & 1
\end{array}\right]\,,\label{eq:def-P}
\end{equation}
then 
\begin{equation}
\kappa=P^{T}\left[\begin{array}{cc}
A_{2} & 0\\
0 & B_{2}
\end{array}\right]P\,.\label{eq:def-mu_1-mu_2}
\end{equation}
We have
\[
P^{-1}=\left[\begin{array}{cc}
\frac{1}{p_{1,1}} & 0\\
-\frac{p_{2,1}}{p_{1,1}} & 1
\end{array}\right]\,.
\]

First, we have to rule out the case where $A_{1}^{y_{0:\tau}}$ and
$B_{1}^{y_{0:\tau}}$ are colinear.
\begin{lem}
\label{lem:A-B-non-colineaire} The quantities $A_{1}^{y_{0:\tau}}$
and $B_{1}^{y_{0:\tau}}$ are not colinear (as functions of $y_{0:\tau}$).
\end{lem}

\begin{proof}
Suppose there exists $\lambda\in\R$ such that $B_{1}^{y_{0:\tau}}=\lambda A_{1}^{y_{0:\tau}}$
for $\lambda_{W}$-almost all $y_{0:\tau}$. %
{} We have, for all $\varphi$ in $\mathcal{C}_{b}^{+}(\R)$, using
Lemma \ref{lem:encadrement-potentiel} (remember Equations (\ref{eq:psi-chapeau-01}),
(\ref{eq:psi-chapeau-02}), (\ref{eq:psi-chapeau-03}), (\ref{eq:psi-chapeau-04}),
(\ref{eq:psi-chapeau-05}))
\begin{multline}
\int_{\mathcal{C}([0;\tau])}\varphi(A_{1}^{y_{0:\tau}})\psi(y_{0:\tau},x,z)\lambda_{W}(dy_{0:\tau})\\
\leq e^{2M|x-z|+\tau\left(M+\frac{M^{2}}{2}\right)}\int_{\mathcal{C}([0;\tau])}\varphi(A_{1}^{y_{0:\tau}})\widehat{\psi}(y_{0:\tau},x,z)\lambda_{W}(dy_{0:\tau})\\
=\sigma_{1}\sigma_{2}e^{2M|x-z|+\tau\left(M+\frac{M^{2}}{2}\right)}\int_{\R}\varphi(t)\exp(-A_{2}x^{2}-B_{2}z^{2}+C_{1}xz+tx+\lambda tz)\Psi'(t)dt\,,\label{eq:not-colinear-01}
\end{multline}
where 
\[
\Psi'(t)=\E^{\p}(\exp(C_{0}^{W_{0:\tau}})|A_{1}^{W_{0:\tau}}=t)\,.
\]
We know the last integral is finite (because $\int_{\mathcal{C}([0,1])}\psi(y_{0:\tau},x,z)=1$
and because of Lemma \ref{lem:encadrement-potentiel}). We introduce
$\Psi'_{1}$ such that
\[
\Psi'(t)=\exp(-\frac{1}{4}(t,\lambda t)\kappa^{-1}(t,\lambda t)^{T})\Psi'_{1}(t)\,,
\]
and 
\[
\forall(t_{1},t_{2})\in\R^{2},\,\mathcal{Q}(t_{1},t_{2})=\exp\left(-\frac{1}{4}(t_{1},t_{2})\kappa^{-1}(t_{1,}t_{2})^{T}\right)\,.
\]

We have, for all $t$, (remember that $\widetilde{\p}$ is defined
in Equation (\ref{eq:def-Ptilde}))
\[
e^{-M|X_{0}-X_{t}|-\frac{Mt}{2}-\frac{M^{2}t}{2}}\leq\left.\frac{d\p}{d\widetilde{\p}}\right|_{\mathcal{F}_{t}}\leq e^{M|X_{t}-X_{0}|+\frac{Mt}{2}}
\]
(this can be deduced from the computations in the proof of Lemma \ref{lem:encadrement-transition}).
Thus, we have, for all $\varphi$ (the first equality being a consequence
of Lemma \ref{lem:encadrement-potentiel})
\begin{eqnarray}
\int_{\mathcal{C}([0;1])}\varphi(A_{1}^{y_{0:\tau}})\psi(y_{0:\tau},x,z)\lambda_{W}(dy_{0:\tau}) & = & \E^{\p}(\varphi(A_{1}^{Y_{0:\tau}})|X_{0}=x,X_{\tau}=z)\nonumber \\
 & = & \frac{\E^{\widetilde{\p}}\left(\left.\varphi(A_{1}^{Y_{0:\tau}})\left.\frac{d\p}{d\widetilde{\p}}\right|_{\mathcal{F}_{\tau}}\right|X_{0}=x,X_{\tau}=z\right)}{\E^{\widetilde{\p}}\left(\left.\left.\frac{d\p}{d\widetilde{\p}}\right|_{\mathcal{F}_{\tau}}\right|X_{0}=x,X_{\tau}=z\right)}\nonumber \\
\mbox{( by Equation (\ref{eq:A-as-gaussian}))} & \geq & e^{-2M|x-z|-\tau\left(M+\frac{M^{2}}{2}\right)}\int_{\R}\varphi(t)\mathcal{Q}''_{x,z}(t)dt\label{eq:not-colinear-02}
\end{eqnarray}
 for some Gaussian density $\mathcal{Q}''_{x,z}$ (the density of
$A_{1}^{Y_{0:\tau}}$ knowing $X_{0}=x$, $X_{\tau}=z$, under $\widetilde{\p}$). 

From Equations (\ref{eq:not-colinear-01}), (\ref{eq:not-colinear-02}),
we deduce that, for $(x,z)$ fixed, we have for almost all $t$, 
\begin{equation}
e^{-2M|x-z|-\tau\left(M+\frac{M^{2}}{2}\right)}\mathcal{Q}''_{x,z}(t)\leq e^{2M|x-z|+\tau\left(M+\frac{M^{2}}{2}\right)}\sigma_{1}\sigma_{2}\mathcal{Q}((t,\lambda t)-2(x,z)\kappa)\Psi'_{1}(t)\,.\label{eq:ineq-dim-1-01}
\end{equation}
In the same way, for $(x,z)$ fixed, we have for almost all $t$,
\begin{equation}
e^{2M|x-z|+\tau\left(M+\frac{M^{2}}{2}\right)}\mathcal{Q}''_{x,z}(t)\geq e^{-2M|x-z|-\tau\left(M+\frac{M^{2}}{2}\right)}\sigma_{1}\sigma_{2}\mathcal{Q}((t,\lambda t)-2(x,z)\kappa)\Psi'_{1}(t)\,.\label{eq:indeq-dim-1-02}
\end{equation}
Looking at Equations (\ref{eq:X-sous-Ptilde}), (\ref{eq:A-as-gaussian}),
we can say that the density $\mathcal{Q}''_{x,z}(t)$ is of the form
\[
\mathcal{Q}''_{x,z}(t)=\frac{1}{\sqrt{2\pi\sigma_{0}^{2}}}\exp\left(-\frac{1}{2\sigma_{0}^{2}}(t-(a_{0}x+b_{0}z))^{2}\right)\,,
\]
with $\sigma_{0}$, $a_{0}$, $b_{0}$ independent of $(x,z)$%
. So, looking at the above inequalities in $(x,z)=(0,0)$, we see
there exists $\epsilon>0$ and a constant $C_{\epsilon}$, such that,
for almost all $t$ in $(-\epsilon,\epsilon)$, 
\begin{equation}
(C_{\epsilon}\sigma_{1}\sigma_{2})^{-1}e^{-\tau(2M+M^{2})}\leq\Psi_{1}'(t)\leq C_{\epsilon}(\sigma_{1}\sigma_{2})^{-1}e^{\tau(2M+M^{2})}\,.\label{eq:dim-1-encadrement-psi-1}
\end{equation}

{} For any $t$, the quantities $\log(\mathcal{Q}''_{x,z}(t))$, $\log(\mathcal{Q}((t,\lambda t)-2(x,z)\kappa))$
are polynomials in $x$, $z$, of degree less than $2$. Using the
above remarks and studying adequate sequences $(x_{n},z_{n})_{n\geq0}$
(for example, with $x_{n}\underset{n\rightarrow+\infty}{\longrightarrow}+\infty$,
$z_{n}$ remaining in a neighborhood of $0$), one can show that the
coefficients in $x^{2}$, $z^{2}$ and $xz$ in these two polynomials
the same. We then have
\[
\frac{a_{0}^{2}}{2\sigma_{0}^{2}}=A_{2}\,,\,\frac{b_{0}^{2}}{2\sigma_{0}^{2}}=B_{2}\,,\,\frac{a_{0}b_{0}}{\sigma_{0}^{2}}=C_{1}\,.
\]
By Equation (\ref{eq:borne-theta}), we have
\[
\frac{a_{0}b_{0}}{\sigma_{0}^{2}}=2\sqrt{A_{2}B_{2}}\geq\frac{h}{2}>\frac{h}{8}
\]
and $C_{1}\leq h/8$, which is not possible, hence the result.
\end{proof}
We can now write for any test function $\varphi$ in $\mathcal{C}_{b}^{+}([0,\tau])$
(remember Equation (\ref{eq:pol-psi-chapeau}))
\begin{multline}
\int_{\mathcal{\mathcal{C}}([0,\tau])}\varphi(A_{1}^{y_{0:\tau}},B_{1}^{y_{0:\tau}})\widehat{\psi}(y_{0:\tau},x,z)\lambda_{W}(dy_{0:\tau})=\\
\sigma_{1}\sigma_{2}\int_{\R^{k}}\varphi(t_{1},t_{2})\exp\left[-A_{2}x^{2}-B_{2}z^{2}+C_{1}xz+t_{1}x\right.\left.+t_{2}z\right]\times\Psi(t_{1},t_{2})dt_{1}dt_{2}\,,\label{eq:integrale-potentiel}
\end{multline}
where 
\[
\Psi(t_{1},t_{1})=\E^{\p}(\exp(C_{0}^{W_{0:\tau}})|A_{1}^{W_{0:\tau}}=t_{1},B_{1}^{W_{0:\tau}}=t_{2})\,.
\]
 We know the integral is finite (because $\int_{\mathcal{C}([0,1])}\psi(y_{0:\tau},x,z)\lambda_{W}(dy_{0:\tau})=1$
and because of Lemma \ref{lem:encadrement-potentiel}). Let us define
$\Psi_{1}$ by the formula 
\[
\Psi(t_{1},t_{2})=\exp\left(-\frac{1}{4}(t_{1},t_{2})\kappa^{-1}(t_{1},t_{2})^{T}\right)\Psi_{1}(t_{1},t_{2})\,.
\]
The next result tells us that, somehow, $\log(\Psi_{1}(t_{1},t_{2}))$
is negligible before $t_{1}^{2}+t_{2}^{2}$ (when $(t_{1},t_{2})\rightarrow+\infty$).
\begin{lem}
\label{lem:borne-Psi_1}There exists a constant $C'_{1}(h,\tau)$
(continuous in $(h,\tau)$) and $\epsilon>0$ such that for all $(x,z)$
and for almost all $(t_{1},t_{2})$ in $B(2(x,z)\kappa,\epsilon)$,
\begin{multline}
\frac{1}{C_{1}'(h,\tau)}\exp\left(-4M|x-z|-\tau\left(2M+M^{2}\right)\right)\leq\Psi_{1}(t_{1},t_{2})\\
\leq C_{1}'(h,\tau)\exp\left(4M|x-z|+\tau\left(2M+M^{2}\right)\right)\,.\label{eq:encadrement-Psi_1}
\end{multline}
\end{lem}

\begin{proof}
We fix $(x,z)$ in $\R^{2}$. Similarly as in Equation (\ref{eq:not-colinear-01}),
we get, using Lemma \ref{lem:encadrement-potentiel}, for all $\varphi\in\mathcal{C}_{b}^{+}(\R^{2})$,

\begin{multline}
\int_{\mathcal{C}([0,\tau])}\varphi(A_{1}^{y_{0:\tau}},B_{1}^{y_{0:\tau}})\psi(y_{0:\tau},x,z)\lambda_{W}(dy_{0:\tau})\\
\leq\int_{\mathcal{C}([0,\tau])}\varphi(A_{1}^{y_{0:\tau}},B_{1}^{y_{0:\tau}})e^{2M|x-z|+\tau\left(M+\frac{M^{2}}{2}\right)}\widehat{\psi}(y_{0:\tau},x,z)\lambda_{W}(dy_{0:\tau})\\
=\sigma_{1}\sigma_{2}e^{2M|x-z|+\tau\left(M+\frac{M^{2}}{2}\right)}\int_{\R^{2}}\varphi(t_{1},t_{2})e^{-A_{2}x^{2}-B_{2}z^{2}+C_{1}xz+t_{1}x+t_{2}z}\Psi(t_{1},t_{2})dt_{1}dt_{2}\\
=\sigma_{1}\sigma_{2}e^{2M|x-z|+\tau\left(M+\frac{M^{2}}{2}\right)}\int_{\R^{2}}\varphi(t_{1},t_{2})\exp\left(-\frac{1}{4}((t_{1},t_{2})-2(x,z)\kappa)\kappa^{-1}((t_{1},t_{2})^{T}-2\kappa(x,z)^{T})\right)\\
\times\Psi_{1}(t_{1},t_{2})dt_{1}dt_{2}\,.\label{eq:formule-psi_chapeau}
\end{multline}
Similarly as in Equation (\ref{eq:not-colinear-02}), we get, for
all $\varphi$, 
\[
\int_{\mathcal{C}([0,\tau])}\varphi(A_{1}^{y_{0:\tau}},B_{1}^{y_{0:\tau}})\psi(y_{0:\tau},x,z)\lambda_{W}(dy_{0:\tau})\leq e^{2M|z-x|+\tau\left(M+\frac{M^{2}}{2}\right)}\int_{\R^{2}}\varphi(t_{1},t_{2})\mathcal{Q}'_{x,z}(t_{1},t_{2})dt_{1}dt_{2}\,,
\]
\[
e^{-2M|z-x|-\tau\left(M+\frac{M^{2}}{2}\right)}\int_{\R^{2}}\varphi(t_{1},t_{2})\mathcal{Q}'_{x,z}(t_{1},t_{2})dt_{1}dt_{2}\leq\int_{\mathcal{C}([0,\tau])}\varphi(A_{1}^{y_{0:\tau}},B_{1}^{y_{0:\tau}})\psi(y_{0:\tau},x,z)\lambda_{W}(dy_{0:\tau})\,,
\]
for some Gaussian density $\mathcal{Q}'_{x,z}$ with covariance matrix
which does not depend on $x$, $z$ (see Equation (\ref{eq:A-as-gaussian})).
This is the density of $(A_{1}^{Y_{0:\tau}},B_{1}^{Y_{0:\tau}})$
knowing $X_{0}=x$ and $X_{\tau}=z$, under $\widetilde{\p}$. %
{} We then have, a.s. in $(t_{1},t_{2})$ (for the Lebesgue measure),
\begin{equation}
\mathcal{Q}'_{x,z}(t_{1},t_{2})\leq\sigma_{1}\sigma_{2}e^{4M|x-z|+\tau\left(2M+M^{2}\right)}\mathcal{Q}((t_{1},t_{2})-2(x,z)\kappa)\Psi_{1}(t_{1},t_{2})\,.\label{eq:maj-Q-prime}
\end{equation}
Using the lower bound in the inequality in Lemma \ref{lem:encadrement-potentiel},
we get in the same way, a.s. in $(t_{1},t_{2})$,
\begin{equation}
\mathcal{Q}'_{x,z}(t_{1},t_{2})\geq\sigma_{1}\sigma_{2}e^{-4M|x-z|-\tau\left(2M+M^{2}\right)}\mathcal{Q}((t_{1},t_{2})-2(x,z)\kappa)\Psi_{1}(t_{1},t_{2})\,.\label{eq:min-Q-prime}
\end{equation}
 We deduce from Equation (\ref{eq:min-Q-prime}), that there exists
$\epsilon_{1}>0$ such that, for all $(x,z)$ and for almost all $(t_{1},t_{2})$
in $B(2(x,z)\kappa,\epsilon_{1})$ 
\[
\Psi_{1}(t_{1},t_{2})\leq C'_{1}(\tau,h)e^{4M|x-z|+\tau\left(2M+M^{2}\right)}\,,
\]
for some function $C_{1}'(\tau,h)$ of the parameters $\tau$, $h$
(continuous in $(h,\tau)$). Looking at Equations (\ref{eq:X-sous-Ptilde}),
(\ref{eq:A-as-gaussian}), one can also see that $\mathcal{Q}'_{x,z}$
reaches its maximum at $(x,z)\kappa'$,  where $\kappa'$ is fixed
in $\mathcal{M}_{2,2}(\R)$ (the set of $2\times2$ matrices with
coefficients in $\R$). From Equation (\ref{eq:maj-Q-prime}), we
get that there exists $\epsilon_{2}>0$ such that, for all $x$, $z$
and for almost all $(t_{1},t_{2})$ in $B((x,z)\kappa',\epsilon_{2}')$,
\begin{multline*}
\mathcal{Q}((t_{1,}t_{2})-2(x,z)\kappa)\geq\frac{1}{2}\mathcal{Q}'_{x,z}((x,z)\kappa')(\sigma_{1}\sigma_{2})^{-1}e^{-4M|x-z|-\tau\left(4M+2M^{2}\right)}(C_{1}'(h,\tau))^{-1}\\
\times\exp\left(-4M\left|\frac{1}{2}(x,z)\kappa'\kappa^{-1}(1,-1)^{T}\right|\right)\,,
\end{multline*}
and so, by continuity, 
\begin{multline*}
\mathcal{Q}((x,z)\kappa'-2(x,z)\kappa)\geq\frac{1}{2}\mathcal{Q}'_{x,z}((x,z)\kappa')(\sigma_{1}\sigma_{2})^{-1}e^{-4M|x-z|-\tau\left(4M+2M^{2}\right)}(C_{1}'(h,\tau))^{-1}\\
\times\exp\left(-4M\left|\frac{1}{2}(x,z)\kappa'\kappa(1,-1)^{T}\right|\right)\,,
\end{multline*}
 If $\kappa'\neq2\kappa$, we can find a sequence $(x_{n},z_{n})$
such that $x_{n}^{2}+z_{n}^{2}\convn+\infty$ and 
\[
\log(\mathcal{Q}((x_{n},z_{n})\kappa'-2(x_{n},z_{n})\kappa))\preceq-(x_{n}^{2}+z_{n}^{2})\,,
\]
whereas
\begin{multline*}
\log\left(\mathcal{Q}'_{x_{n},z_{n}}((x_{n},z_{n})\kappa')(\sigma_{1}\sigma_{2})^{-1}e^{-4M|x_{n}-z_{n}|-\tau\left(4M+2M^{2}\right)}\exp\left(-4M\left|\frac{1}{2}(x,z)\kappa'\kappa^{-1}(1,-1)^{T}\right|\right)\right)\\
\succeq-|x_{n}|-|z_{n}|\,,
\end{multline*}
which is not possible. So $\kappa'=2\kappa$. 

So, we get from Equations (\ref{eq:maj-Q-prime}), (\ref{eq:min-Q-prime})
that there exists $\epsilon_{3}>0$ such that for all $x$, $z$,
and for almost all $(t_{1},t_{2})$ in $B(2(x,z)\kappa,\epsilon_{3})$,
\[
C_{1}'(h,\tau)e^{4M|x-z|+\tau(2M+M^{2})}\geq\Psi_{1}(t_{1},t_{2})\geq\frac{e^{-4M|x-z|-\tau(2M+M^{2})}}{C_{1}'(h,\tau)}
\]
(with, possibly, a new $C_{1}'(h,\tau)$).
\end{proof}
\begin{lem}
\label{lem:borne-integrale} If we have a set $\mathcal{A}=\{y_{0:\tau}\in\mathcal{C}([0;\tau])\,:\,(A_{1}^{y_{0:\tau}},B_{1}^{y_{0:\tau}})\in\mathcal{B}\}$
for some measurable subset $\mathcal{B}$ of $\R^{2}$, then
\begin{multline*}
\int_{\mathcal{A}}\widehat{\psi}(y_{0:\tau},x,z)\lambda_{W}(dy_{0:\tau})\leq\sigma_{1}\sigma_{2}C_{1}'(h,\tau)\\
\times\int_{(t_{2},t_{2})\in\mathcal{B}}\exp\left\{ -\frac{A_{2}^{-1}}{4}\left(\left|\frac{t_{1}}{p_{1,1}}-\frac{p_{2,1}t_{2}}{p_{1,1}}-2A_{2}p_{1,1}x\right|-4M\left|\frac{1}{p_{1,1}}+\frac{p_{2,1}}{p_{1,1}}\right|\right)_{+}^{2}\right.\\
\left.-\frac{B_{2}^{-1}}{4}\left(\left|t_{2}-2B_{2}(p_{2,1}x+z)\right|-4M\right)_{+}^{2}\right\} \\
\times\exp\left(4M^{2}(1,-1)\kappa^{-1}(1,-1)^{T}+\tau\left(2M+M^{2}\right)+4M|x-z|\right)dt_{1}dt_{2}\,,
\end{multline*}
and
\begin{multline*}
\int_{\mathcal{A}}\widehat{\psi}(y_{0:\tau},x,z)\lambda_{W}(dy_{0:\tau})\geq\sigma_{1}\sigma_{2}C_{1}'(h,\tau)^{-1}\\
\times\int_{(t_{2},t_{2})\in\mathcal{B}}\exp\left\{ -\frac{A_{2}^{-1}}{4}\left(\left|\frac{t_{1}}{p_{1,1}}-\frac{p_{2,1}t_{2}}{p_{1,1}}-2A_{2}p_{1,1}x\right|+4M\left|\frac{1}{p_{1,1}}+\frac{p_{2,1}}{p_{1,1}}\right|\right)^{2}\right.\\
\left.-\frac{B_{2}^{-1}}{4}\left(\left|t_{2}-2B_{2}(p_{2,1}x+z)\right|+4M\right)^{2}\right\} \\
\times\exp\left(4M^{2}(1,-1)\kappa^{-1}(1,-1)^{T}-\tau\left(2M+M^{2}\right)-4M|x-z|\right)dt_{1}dt_{2}\,.
\end{multline*}
\end{lem}

\begin{proof}
{} We have (computing as in Equation (\ref{eq:formule-psi_chapeau}))
\begin{multline*}
\int_{\mathcal{A}}\widehat{\psi}(y_{0:\tau},x,z)\lambda_{W}(dy_{0:\tau})\\
=\sigma_{1}\sigma_{2}\int_{\mathcal{B}}\exp\left(-(x,z)\kappa(x,z)^{T}+t_{1}x+t_{2}z\right)\exp\left(-\frac{1}{4}(t_{1},t_{2})\kappa{}^{-1}(t_{1},t_{2})^{T}\right)\Psi_{1}(t_{1},t_{2})dt_{1}dt_{2}\\
\mbox{(by Lemma \ref{lem:borne-Psi_1}, for a complete proof see Lemma \ref{lem:tech-03})}\\
\leq\sigma_{1}\sigma_{2}\int_{\mathcal{B}}\exp\left(-\frac{1}{4}[(t_{1},t_{2})^{T}-2\kappa(x,z)^{T}]^{T}\kappa^{-1}[(t_{1},t_{2})^{T}-2\kappa(x,z)^{T}]\right)\\
\times C_{1}'(h,\tau)\exp\left(2M\times|(1,-1)\kappa^{-1}(t_{1},t_{2})^{T}|+\tau\left(2M+M^{2}\right)\right)dt_{1}dt_{2}\,,
\end{multline*}
and we can bound by below by 
\begin{multline*}
\sigma_{1}\sigma_{2}\int_{\mathcal{B}}\exp\left(-\frac{1}{4}[(t_{1},t_{2})^{T}-2\kappa(x,z)^{T}]^{T}\kappa^{-1}[(t_{1},t_{2})^{T}-2\kappa(x,z)^{T}]\right)\\
\times\frac{1}{C_{1}'(h,\tau)}\exp\left(-2M\times|(1,-1)\kappa^{-1}(t_{1},t_{2})^{T}|-\tau\left(2M+M^{2}\right)\right)dt_{1}dt_{2}\,,
\end{multline*}
For $(t_{1},t_{2})\in\R^{2}$, we have, for any $\delta\in\{-1,1\}$
\begin{multline*}
\exp\left(-\frac{1}{4}[(t_{1},t_{2})^{T}-2\kappa(x,z)^{T}]^{T}\kappa^{-1}[(t_{1},t_{2})^{T}-2\kappa(x,z)^{T}]+2\delta M\times(1,-1)\kappa^{-1}(t_{1},t_{2})^{T}\right)\\
=\exp\left(-\frac{1}{4}[(t_{1},t_{2})^{T}-2\kappa(x,z)^{T}-4\delta M(1,-1)^{T}]^{T}\kappa^{-1}[(t_{1},t_{2})^{T}-2\kappa(x,z)^{T}-4\delta M(1,-1)^{T}]\right.\\
\left.+4M^{2}(1,-1)\kappa^{-1}(1,-1)^{T}+4\delta M(x,z)(1,-1)^{T}\right)\\
=\exp\left(-\frac{1}{4}[(P^{-1})^{T}(t_{1},t_{2})^{T}-2\left[\begin{array}{cc}
A_{2} & 0\\
0 & B_{2}
\end{array}\right]P(x,z)^{T}-4\delta M(P^{-1})^{T}(1,-1)^{T}]^{T}\right.\\
\times\left.\left[\begin{array}{cc}
A_{2}^{-1} & 0\\
0 & B_{2}^{-1}
\end{array}\right][(P^{-1})^{T}(t_{1},t_{2})^{T}-2\left[\begin{array}{cc}
A_{2} & 0\\
0 & B_{2}
\end{array}\right]P(x,z)^{T}-4\delta M(P^{-1})^{T}(1,-1)^{T}]\right)\\
\times\exp(4M^{2}(1,-1)\kappa^{-1}(1,-1)^{T}+4\delta M(x,z)(1,-1)^{T})\,.
\end{multline*}
From there, we get the result. 
\end{proof}

\subsection{Truncation\label{subsec:Truncation-1}}

In the following, the parameter $\Delta>0$ is to be understood as
a truncation level. For $k\geq0$ and $\Delta>0$, we set (for all
$b$)
\begin{equation}
C_{k+1}(\Delta,b)=\left\{ z\,:\,\left|2B_{2}(1+p_{2,1})z-b\right|\leq\Delta\right\} \label{eq:def-compact}
\end{equation}
(which indeed does not depend on $k$),
\begin{equation}
C_{k+1}(\Delta)=C_{k+1}(\Delta,B_{1}^{y_{k\tau:(k+1)\tau}})\label{eq:def-compact-suite}
\end{equation}
 and
\begin{equation}
m_{k+1}(b)=\frac{b}{2B_{2}(1+p_{2,1})}\,,\label{eq:def-m_k}
\end{equation}
(which indeed does not depend on $k$) and
\[
m_{k+1}=m_{k+1}(B_{1}^{Y_{k\tau:(k+1)\tau}})\,.
\]
We suppose that $m_{0}$ is a point in the support of $\pi_{0}$ (the
law of $X_{0}$) and we set 
\[
C_{0}(\Delta)=\left[m_{0}-\frac{\Delta}{2B_{2}(1+p_{2,1})}\,,\,m_{0}+\frac{\Delta}{2B_{2}(1+p_{2,1})}\right].
\]
Under Hypothesis \ref{hyp:sur-pi_0}, there exists constants $C_{0}$
and $C_{0}'$ such that
\[
\pi_{0}(C_{0}(\Delta)^{\complement})\leq C_{0}'e^{-C_{0}\Delta^{2}}\,,\,\forall\Delta>0\,.
\]

From Equation (\ref{eq:borne-theta}) and Lemma \ref{lem:variation-const},
we see that there exists a universal constant $C$ such that 
\begin{equation}
|m_{k}-m_{k-1}|\leq\begin{cases}
C(M\tau+\mathcal{V}_{(k-2)\tau,k\tau}+(1+\frac{1}{\theta})\cW_{(k-2)\tau,k\tau}) & \mbox{ if }k\geq2\,,\\
|m_{0}|+C(M\tau+\mathcal{V}_{0,2\tau}+(1+\frac{1}{\theta})\cW_{0,2\tau}) & \mbox{ if }k=1\,.
\end{cases}\label{eq:variation-m_k}
\end{equation}
We set 
\begin{equation}
d(\Delta)=\frac{\Delta}{1+p_{2,1}}-(2+\frac{16}{h})B_{2}p_{2,1}\theta^{1-\iota}\Delta-4M\,,\label{eq:def-d-Delta}
\end{equation}
(because of Equation (\ref{eq:borne-theta}), we have $d(\Delta)\underset{\Delta\rightarrow+\infty}{\longrightarrow}+\infty$)
and 
\begin{multline}
T(\Delta)=\frac{C\sqrt{\tau}\exp\left(-\frac{1}{2}\left(\frac{\theta^{1-\iota}\Delta}{12C\sqrt{2\tau}}\right)^{2}\right)}{\theta^{1-\iota}\Delta}\\
+\left(M\frac{(1+p_{2,1})}{p_{1,1}}+\sqrt{A_{2}}\right)C_{1}'(h,\tau)\sigma_{1}\sigma_{2}p_{1,1}\frac{B_{2}e^{26M\tau+\frac{7\tau M}{2}+40M^{2}/h}}{d(\Delta)}\exp\left(-\frac{1}{4B_{2}}d(\Delta)^{2}\right)+e^{-C_{0}\Delta^{2}}\,.\label{eq:def-T-Delta}
\end{multline}

We define, for all $k\geq1$, $x$ and $x'$ in $\R$ (recall that
$\psi_{k}$ is defined in Equation \ref{eq:def-psi_k}) 
\begin{equation}
\psi_{k}^{\Delta}(x,x')=\psi_{k}(x,x')\1_{C_{k}(\Delta)}(x')\,,\label{eq:def-psi-Delta}
\end{equation}
 
\begin{equation}
D_{k}=|m_{k}-m_{k-1}|\,,\label{eq:def-D_k}
\end{equation}
and for $D\geq0$, 
\begin{equation}
\xi_{1}(D,\Delta)=\frac{1}{\sqrt{2\pi\tau}}\exp\left(-\frac{\left(D+\frac{\Delta}{B_{2}(1+p_{2,1})}\right)^{2}}{2\tau}\right)\exp\left(-M\left(D+\frac{\Delta}{B_{2}(1+p_{2,1})}\right)-\left(\frac{\tau}{2}+\frac{\tau^{2}}{2}\right)M\right)\,,\label{eq:def-xi_1}
\end{equation}
\begin{equation}
\xi_{2}(D,\Delta)=\frac{1}{\sqrt{2\pi\tau}}\exp\left(-\frac{\left(\left(D-\frac{\Delta}{B_{2}(1+p_{2,1})}\right)_{+}\right)^{2}}{2\tau}\right)\exp\left(M\left(D+\frac{\Delta}{B_{2}(1+p_{2,1})}\right)+\frac{\tau}{2}M\right)\,,\label{eq:def-xi_2}
\end{equation}
and, 
\begin{equation}
R_{k}^{\Delta}(x,dx')=\begin{cases}
\psi_{k}^{\Delta}(x,x')Q(x,dx') & \mbox{ if }x\in C_{k-1}(\Delta)\,,\\
\psi_{k}^{\Delta}(x,x')\xi_{1}(D_{k},\Delta)dx' & \mbox{ if }x\notin C_{k-1}(\Delta)\,.
\end{cases}\label{eq:def-R-Delta}
\end{equation}
We define $(\pi_{n}^{\Delta})_{n\geq0}$ by the following
\begin{equation}
\begin{cases}
\pi_{0}^{\Delta}=\pi_{0}\\
\pi_{k\tau}^{\Delta}=\overline{R}_{k}^{\Delta}\overline{R}_{k-1}^{\Delta}\dots\overline{R}_{1}^{\Delta}(\pi_{0}) & \mbox{ for all }k\geq1\,.
\end{cases}\label{eq:def-truncated-filter}
\end{equation}
The next lemma tells us that the measures $\pi_{k}$ are concentrated
on the compacts $C_{k}(\Delta)$.
\begin{lem}
If
\begin{equation}
\theta^{1-\iota}\Delta>3|m_{0}|+3CM\tau\text{ and }d(\Delta)>0\,,\label{eq:hyp-Delta}
\end{equation}
then we have, for all $k\geq0$,
\[
\E(\pi_{k\tau}(C_{k}(\Delta)^{\complement})\preceq T(\Delta)\,.
\]
\end{lem}

\begin{proof}
We suppose $k\geq1$ (the proof is similar for $k=0$). For a measure
$\mu$ in $\mathcal{M}^{+}(\R)$, we define
\begin{equation}
\widetilde{Q}\mu(dx,dx')=\mu(dx)Q(x,dx')\,,\,\forall x,x'\in\R,\label{eq:def-extension-Qtilde}
\end{equation}
(recall $\widetilde{Q}$ has been defined as a Markov kernel on $\R^{2}$,
so the above is an extension of the definition of $\widetilde{Q}$).
We have
\begin{multline*}
\pi_{k}(C_{k}(\Delta)^{\complement})\times\1_{|m_{k}-m_{k-1}|\leq\Delta\theta^{1-\iota}}=\\
\1_{|m_{k}-m_{k-1}|\leq\Delta\theta^{1-\iota}}\int_{(x,x')\in\R^{2}}\frac{\psi_{k}(x,x')}{\langle\widetilde{Q}\pi_{k-1},\psi_{k}\rangle}\1_{C_{k}(\Delta)^{\complement}}(x)(\1_{C_{k-1}(2\Delta)}(x)+\1_{C_{k-1}(2\Delta)^{\complement}}(x))\widetilde{Q}\pi_{k-1}(dx,dx')\,,
\end{multline*}
{}  and (using the same computations as in \cite{legland-oudjane-2003},
proof of Proposition 5.3, \cite{oudjane-2000}, p. 66) 
\begin{multline}
\E\left(\left.\1_{|m_{k}-m_{k-1}|\leq\Delta\theta^{1-\iota}}\int_{\R^{2}}\frac{\psi_{k}(x,x')}{\langle\widetilde{Q}\pi_{k-1},\psi_{k}\rangle}\1_{C_{k}(\Delta)^{\complement}}(x')\1_{C_{k-1}(2\Delta)}(x)\widetilde{Q}\pi_{k-1}(dx,dx')\right|Y_{0:(k-1)\tau}\right)\\
=\int_{y\in\mathcal{C}([0,\tau])}\1_{|m_{k}(B_{1}^{y})-m_{k-1}|\leq\Delta\theta^{1-\iota}}\\
\times\left(\int_{\R^{2}}\frac{\psi(y,x,x')}{\langle\widetilde{Q}\pi_{k-1},\psi(y,.,.)\rangle}\1_{C_{k}(\Delta,B_{1}^{y})^{\complement}}(x')\1_{C_{k-1}(2\Delta)}(x)\widetilde{Q}\pi_{k-1}(dx,dx')\right)\\
\times\left(\int_{(u,u')\in\R^{2}}\widetilde{Q}\pi_{k-1}(u,du')\psi(y,u,u')\right)\lambda_{W}(dy)\\
\mbox{(by Fubini's theorem)}\\
=\int_{y\in\mathcal{C}([0,1])}\1_{|m_{k}(B_{1}^{y})-m_{k-1}|\leq\Delta\theta^{1-\iota}}\int_{\R^{2}}\psi(y,x,x')\1_{C_{k}(\Delta,B_{1}^{y})^{\complement}}(x')\1_{C_{k-1}(2\Delta)}(x)\widetilde{Q}\pi_{k-1}(dx,dx')\lambda_{W}(dy)\\
\mbox{(using Lemma \ref{lem:encadrement-potentiel} and Lemma \ref{lem:borne-integrale})}\\
\leq\sigma_{1}\sigma_{2}C_{1}'(h,\tau)\int_{(x,x')\in\R^{2}}\int_{(t_{1},t_{2})\in\R^{2}\,:\,|m_{k}(t_{2})-m_{k-1}|\leq\Delta\theta^{1-\iota}}e^{6M|x-x'|+\tau(3M+3M^{2}/2)+4M^{2}\left|(1,-1)\kappa^{-1}(1,-1)^{T}\right|}\\
\times e^{-\frac{1}{4A_{2}}\left(\left|\frac{t_{1}}{p_{1,1}}-\frac{p_{2,1}t_{2}}{p_{1,1}}-2A_{2}p_{1,1}x\right|-4M\left|\frac{1+p_{2,1}}{p_{1,1}}\right|\right)_{+}^{2}}\times e^{-\frac{1}{4B_{2}}\left(\left|t_{2}-2B_{2}(p_{2,1}x+x')\right|-4M\right)_{+}^{2}}\\
\times\1_{C_{k}(\Delta,t_{2})^{\complement}}(x')\1_{C_{k-1}(2\Delta)}(x)dt_{1}dt_{2}\widetilde{Q}\pi_{k-1}(dx,dx')\,.\label{eq:maj-compact-01}
\end{multline}
By a similar computation, we get
\begin{multline*}
\E\left(\left.\int_{\R^{2}}\frac{\psi_{k}(x,x')}{\langle\widetilde{Q}\pi_{k-1},\psi_{k}\rangle}\1_{C_{k}(\Delta)^{\complement}}(x)\1_{C_{k-1}(2\Delta)^{\complement}}(x)\widetilde{Q}\pi_{k-1}(dx,dx')\right|Y_{0:(k-1)\tau}\right)\\
=\int_{y\in\mathcal{C}([0,1])}\int_{\R^{2}}\psi(y,x,x')\1_{C_{k}(\Delta,B_{1}^{y})^{\complement}}(x')\1_{C_{k-1}(2\Delta)^{\complement}}(x)\widetilde{Q}\pi_{k-1}(dx,dx')\lambda_{W}(dy)\\
\leq\int_{\R^{2}}\1_{C_{k-1}(2\Delta)^{\complement}}(x)\widetilde{Q}\pi_{k-1}(dx,dx')\leq\pi_{k-1}(C_{k-1}(2\Delta)^{\complement})\,.
\end{multline*}
 For all $t_{2}$ such that $|m_{k}(t_{2})-m_{k-1}|\leq\theta^{1-\iota}\Delta$,
for $x\in C_{k-1}(2\Delta)$, $x'\in C_{k}(\Delta,t_{2})^{\complement}$,
\begin{multline}
\left|t_{2}-2B_{2}(p_{2,1}x+x')\right|\\
=\left|-2B_{2}x'+\frac{t_{2}}{p_{2,1}+1}+\frac{p_{2,1}}{p_{2,1}+1}t_{2}-2B_{2}p_{2,1}x\right|\\
\geq\frac{\Delta}{1+p_{2,1}}-2B_{2}p_{2,1}|m_{k}(t_{2})-m_{k-1}|-2B_{2}p_{2,1}|m_{k-1}-x|\\
\mbox{(using Equation (\ref{eq:borne-theta}))}\geq\frac{\Delta}{1+p_{2,1}}-(2+\frac{16}{h})B_{2}p_{2,1}\theta^{1-\iota}\Delta\,.\label{eq:star-01}
\end{multline}
 So
\begin{multline*}
\E(\pi_{k}(C_{k}(\Delta)^{\complement}))\leq\p(|m_{k}-m_{k-1}|\geq\theta^{1-\iota}\Delta)+\E(\pi_{k-1}(C_{k-1}(2\Delta)^{\complement}))\\
+\sigma_{1}\sigma_{2}C_{1}'(h,\tau)\int_{(x,x')\in\R^{2}}\int_{(t_{1,}t_{2})\in\R^{2}\,:\,\left|t_{2}-2B_{2}(p_{2,1}x+x')\right|>\frac{\Delta}{1+p_{2,1}}-(2+\frac{4}{h})\Delta\theta^{1-\iota}B_{2}p_{2,1}}\\
e^{6M|x-x'|+\tau(3M+3M^{2}/2)+4M^{2}(1,-1)\kappa^{-1}(1,-1)^{T}}\\
\times e^{-\frac{1}{4A_{2}}\left(\left|\frac{t_{1}}{p_{1,1}}-\frac{p_{2,1}t_{2}}{p_{1,1}}-2A_{2}p_{1,1}x\right|-4M\left|\frac{1+p_{2,1}}{p_{1,1}}\right|\right)_{+}^{2}}\\
\times e^{-\frac{1}{4B_{2}}\left(\left|t_{2}-2B_{2}(p_{2,1}x+x')\right|-4M\right)_{+}^{2}}dt_{1}dt_{2}\widetilde{Q}\pi_{k-1}(dx,dx')\,.
\end{multline*}

We have, for all $x\geq0$, 
\begin{eqnarray*}
\p(\mathcal{V}_{0,2\tau}\geq x) & = & \p(\sup_{s\in[0,2\tau]}V_{s}-\inf_{s\in[0,2\tau]}V_{s}\geq x)\\
 & \leq & \p(\sup_{s\in[0,2\tau]}V_{s}\geq x/2)+\p(-\inf_{s\in[0,2\tau]}V_{s}\geq x/2)\\
 & = & 4\p(V_{2\tau}\geq x/2)\\
 & = & 2\p(2|V_{2\tau}|\geq x)\,.
\end{eqnarray*}
And so, we can bound (for all $x$)
\begin{equation}
\p(\cV_{(k-2)_{+}\tau,k\tau}\geq x)\leq2\p(2|W_{2\tau}|\geq x)\,,\label{eq:maj-Vcal}
\end{equation}
\begin{equation}
\p(\cW_{(k-2)_{+}\tau,k\tau}\geq x)\leq2\p(2|W_{2\tau}|\geq x)\,.\label{eq:maj-Wcal}
\end{equation}
 So (with the constant $C$ defined in Equation (\ref{eq:variation-m_k})),
using the inequality
\begin{equation}
\forall z>0\,,\,\int_{z}^{+\infty}\frac{\exp\left(-\frac{t^{2}}{2\sigma^{2}}\right)}{\sqrt{2\pi\sigma^{2}}}dt\leq\frac{\sigma\exp\left(-\frac{z^{2}}{2\sigma^{2}}\right)}{z\sqrt{2\pi}}\,,\label{eq:queue-gaussienne}
\end{equation}
and using Equation (\ref{eq:variation-m_k}), as $\theta^{1-\iota}\Delta>3|m_{0}|+3CM\tau$,
we get
\begin{eqnarray}
\p(|m_{k}-m_{k-1}|\geq\theta^{1-\iota}\Delta) & \leq & 4\p\left(2\times|W_{2\tau}|\ge\frac{\theta^{1-\iota}\Delta}{6C}\right)\nonumber \\
 & = & 4\p\left(|W_{1}|\geq\frac{\theta^{1-\iota}\Delta}{12C\sqrt{2\tau}}\right)\nonumber \\
 & \leq & \frac{96C\sqrt{\tau}}{\theta^{1-\iota}\Delta\sqrt{\pi}}\exp\left(-\frac{1}{2}\left(\frac{\theta^{1-\iota}\Delta}{12C\sqrt{2\tau}}\right)^{2}\right)\,.\label{eq:maj-distance-compacts}
\end{eqnarray}

For all $x$, $x'$, we have
\begin{multline}
\int_{(t_{1,}t_{2})\in\R^{2}\,:\,\left|t_{2}-2B_{2}(p_{2,1}x+x')\right|>\frac{\Delta}{1+p_{2,1}}-(2+\frac{16}{h})\Delta\theta^{1-\iota}B_{2}p_{2,1}}\\
e^{-\frac{1}{4A_{2}}\left(\left|\frac{t_{1}}{p_{1,1}}-\frac{p_{2,1}t_{2}}{p_{1,1}}-2A_{2}p_{1,1}x\right|-4M\left|\frac{1+p_{2,1}}{p_{1,1}}\right|\right)_{+}^{2}}\\
\times e^{-\frac{1}{4B_{2}}\left(\left|t_{2}-2B_{2}(p_{2,1}x+x')\right|-4M\right)_{+}^{2}}dt_{1}dt_{2}\\
\mbox{(change of variables : }\left(\begin{array}{c}
t_{1}'\\
t_{2}'
\end{array}\right)=(P^{-1})^{T}\left(\begin{array}{c}
t_{1}\\
t_{2}
\end{array}\right))\\
=\int_{(t_{1}',t_{2}')\in\R^{2}\,:\,|t_{2}'-2B_{2}(p_{2,1}x+x')|>\frac{\Delta}{1+p_{2,1}}-(2+\frac{16}{h})\Delta\theta^{1-\iota}B_{2}p_{2,1}}e^{-\frac{1}{4A_{2}}\left(\left|t'_{1}-2A_{2}p_{1,1}x\right|-4M\left|\frac{1+p_{2,1}}{p_{1,1}}\right|\right)_{+}^{2}}\\
\times e^{-\frac{1}{4B_{2}}\left(\left|t'_{2}-2B_{2}(p_{2,1}x+x')\right|-4M\right)_{+}^{2}}p_{1,1}dt'_{1}dt'_{2}\\
\mbox{(by (\ref{eq:queue-gaussienne}))}\leq\left(8M\frac{|1+p_{2,1}|}{p_{1,1}}+2\sqrt{\pi A_{2}}\right)\\
\times\frac{4B_{2}}{d(\Delta)}\exp\left(-\frac{1}{4B_{2}}d(\Delta)^{2}\right)p_{1,1}\,.\label{eq:maj-compact-02}
\end{multline}

We have, by Lemma \ref{lem:encadrement-transition} and Equation (\ref{eq:borne-theta}),
\begin{multline}
\int_{(x,x')\in\R^{2}}e^{6M|x-x'|+\tau(3M+3M^{2}/2)+4M^{2}(1,-1)\kappa^{-1}(1,-1)^{T}}\widetilde{Q}\pi_{k-1}(dx,dx')\\
\leq\int_{(x,x')\in\R^{2}}\frac{1}{\sqrt{2\pi\tau}}\exp\left(-\frac{(x-x')^{2}}{2\tau}+7M|x-x'|+\tau\left(\frac{7}{2}M+\frac{3}{2}M^{2}\right)+40M^{2}\right)dx'\pi_{k-1}(dx)\\
\leq2\exp\left(\frac{49}{2}M^{2}\tau+\tau\left(\frac{7}{2}M+\frac{3}{2}M^{2}\right)+\frac{40M^{2}}{h}\right)\,.\label{eq:star-02}
\end{multline}

So we have
\begin{multline*}
\E(\pi_{k}(C_{k}(\Delta)^{\complement}))\leq\E(\pi_{k-1}(C_{k-1}(2\Delta)^{\complement}))+\frac{96C\sqrt{\tau}}{\theta^{1-\iota}\Delta\sqrt{\pi}}\exp\left(-\frac{1}{2}\left(\frac{\theta^{1-\iota}\Delta}{12C\sqrt{2\tau}}\right)^{2}\right)\\
+\left(8M\frac{|1+p_{2,1}|}{p_{1,1}}+2\sqrt{\pi A_{2}}\right)C_{1}'(h,\tau)\sigma_{1}\sigma_{2}p_{1,1}\frac{8B_{2}}{d(\Delta)}\\
\times\exp\left(-\frac{1}{4B_{2}}d(\Delta)^{2}+26M^{2}\tau+\frac{7}{2}\tau M+\frac{40M^{2}}{h}\right)\,.
\end{multline*}
Then, by recurrence, 
\begin{multline*}
\E(\pi_{k}(C_{k}(\Delta)^{\complement}))\leq\E(\pi_{0}(C_{0}(2^{k}\Delta)^{\complement}))+\sum_{i=0}^{k-1}\left[\frac{96C\sqrt{\tau}}{2^{i}\theta^{1-\iota}\Delta\sqrt{\pi}}\exp\left(-\frac{1}{2}\left(\frac{2^{i}\theta^{1-\iota}\Delta}{12C\sqrt{2\tau}}\right)^{2}\right)\right.\\
\left.+\left(8M\frac{|1+p_{2,1}|}{p_{1,1}}+2\sqrt{\pi A_{2}}\right)C_{1}'(h,\tau)\sigma_{1}\sigma_{2}p_{1,1}\frac{8B_{2}e^{26M^{2}\tau+\frac{7}{2}\tau M+40M^{2}/h}}{d(2^{i}\Delta)}\exp\left(-\frac{1}{4B_{2}}d(2^{i}\Delta)^{2}\right)\right]\,.
\end{multline*}
We have 
\begin{eqnarray*}
\frac{1}{d(\Delta)}+\frac{1}{d(2\Delta)}+\frac{1}{d(4\Delta)}+\dots & \leq & \frac{1}{d(\Delta)}+\frac{1}{d(\Delta)}+\frac{1}{d(\Delta)+2\Delta\left(\frac{1}{1+p_{2,1}}-(2+\frac{16}{h})B_{2}p_{2,1}\theta^{1-\iota}\right)}\\
 &  & +\frac{1}{d(\Delta)+4\Delta\left(\frac{1}{1+p_{2,1}}-(2+\frac{16}{h})B_{2}p_{2,1}\theta^{1-\iota}\right)}+\dots\\
 & \leq & \frac{2}{d(\Delta)}+\frac{1}{\Delta\left(\frac{1}{1+p_{2,1}}-(2+\frac{16}{h})B_{2}p_{2,1}\theta^{1-\iota}\right)}\leq\frac{3}{d(\Delta)}\,.
\end{eqnarray*}
So 
\begin{multline*}
\E(\pi_{k}(C_{k}(\Delta)^{\complement}))\leq\E(\pi_{0}(C_{0}(2^{k}\Delta)^{\complement}))+\frac{192C\sqrt{\tau}}{\theta^{1-\iota}\Delta\sqrt{\pi}}\exp\left(-\frac{1}{2}\left(\frac{\theta^{1-\iota}\Delta}{12C\sqrt{2\tau}}\right)^{2}\right)\\
+\left(8M\frac{|1+p_{2,1}|}{p_{1,1}}+2\sqrt{\pi A_{2}}\right)C_{1}'(h,\tau)\sigma_{1}\sigma_{2}p_{1,1}8B_{2}e^{26M^{2}\tau+\frac{7}{2}\tau M+40M^{2}/h}\\
\times\frac{3}{d(\Delta)}\exp\left(-\frac{1}{4B_{2}}d(\Delta)^{2}\right)\,.
\end{multline*}
And we can conclude because of Hypothesis \ref{hyp:sur-pi_0}.
\end{proof}
\begin{cor}
\label{cor:erreur-locale-pi-prime}We suppose that $\pi_{0}'\in\mathcal{P}(\R)$
is such that $\pi_{0}$ and $\pi_{0}'$ are comparable. We suppose
that $(\pi'_{t})_{t\geq0}$ is defined by Equation (\ref{eq:kallianpur-striebel-02}).
Under the assumption of the previous Lemma, we have, for all $k\geq0$,
\[
\E(\pi'_{k\tau}(C_{k}(\Delta)^{\complement})\preceq T(\Delta)e^{h(\pi_{0},\pi_{0}')}\,.
\]
\end{cor}

\begin{proof}
By Equations (\ref{eq:hilbert-prop-01}), (\ref{eq:hilbert-prop-02}),
(\ref{eq:F-K-sequence}), we have, for all $k$, 
\begin{equation}
h(\pi'_{k\tau},\pi_{k\tau})\leq h(\pi'_{0},\pi_{0})\,.\label{eq:hilbert-non-croissante}
\end{equation}
So, by Equation (\ref{eq:encadrement-rapport-mesures}), 
\begin{eqnarray*}
\E(\pi'_{k\tau}(C_{k}(\Delta)^{\complement}) & \leq & \E(e^{h(\pi_{0}',\pi_{0})}\pi_{k\tau}(C_{k}(\Delta)^{\complement}))\\
 & \preceq & T(\Delta)e^{h(\pi_{0},\pi_{0}')}\,.
\end{eqnarray*}
\end{proof}
The next result tells us that $\overline{R}_{k}$ and $\overline{R}_{k}^{\Delta}$
($k\geq1$) are close in some sense (recall that $\pi_{k\tau}'=\overline{R}_{k}(\pi'_{(k-1)\tau})$)
\begin{prop}
\label{prop:erreur-locale}We suppose that $\Delta$ satisfies the
assumption of the previous Lemma (Equation (\ref{eq:hyp-Delta})).
We suppose that $(\pi_{n\tau}')_{n\geq0}$ satisfies the assumptions
of the above Corollary. For all $k\geq1$, we have
\[
\E(\Vert\pi'_{k\tau}-\overline{R}_{k}^{\Delta}(\pi'_{(k-1)\tau})\Vert)\preceq T(\Delta)e^{2h(\pi_{0},\pi_{0}')}\,.
\]
\end{prop}

\begin{proof}
We define measures on $\R^{2}$ (remember Equation (\ref{eq:def-extension-Qtilde})):
\begin{multline*}
\mu=\widetilde{Q}\pi'_{(k-1)\tau}\\
=\widetilde{Q}(\1_{C_{k-1}(\Delta)^{\complement}}\pi'_{(k-1)\tau})(dx,dx')+(\1_{C_{k}(\Delta)}(x')+\1_{C_{k}(\Delta)^{\complement}}(x'))\widetilde{Q}(\1_{C_{k-1}(\Delta)}\pi'_{(k-1)\tau})(dx,dx')\,,
\end{multline*}
\begin{multline*}
\mu'(dx,dx')=\1_{C_{k}(\Delta)}(x')\widetilde{Q}(\1_{C_{k-1}(\Delta)}\pi'_{(k-1)\tau})(dx,dx')\\
+\pi'_{(k-1)\tau}(C_{k-1}(\Delta)^{\complement})\xi_{1}(D_{k},\Delta)\1_{C_{k}(\Delta)}(x')dxdx'\,,
\end{multline*}
where (by a slight abuse of notation)
\[
\widetilde{Q}(\1_{C_{k-1}(\Delta)}\pi'_{(k-1)\tau})(dx,dx')=\1_{C_{k-1}(\Delta)}(x)\pi'_{(k-1)\tau}(dx)Q(x,dx')\,,
\]
\[
\widetilde{Q}(\1_{C_{k-1}(\Delta)^{\complement}}\pi'_{(k-1)\tau})(dx,dx')=\1_{C_{k-1}(\Delta)^{\complement}}(x)\pi'_{(k-1)\tau}(dx)Q(x,dx')\,,
\]
By the definition of $R^{\Delta}$ (Equation (\ref{eq:def-R-Delta}))
and computing as in \cite{oudjane-rubenthaler-2005}, p. 433 (or as
in \cite{oudjane-2000}, p.66), we get
\begin{eqnarray*}
\Vert\pi'_{k\tau}-\overline{R}_{k}^{\Delta}(\pi'_{(k-1)\tau})\Vert_{\mbox{}} & = & \Vert\psi_{k}\bullet\mu-\psi_{k}\bullet\mu'\Vert\\
\mbox{(using Equation (\ref{eq:maj-erreur-locale-01}))} & \leq & 2\int_{\R^{2}}\frac{\psi_{k}(x,x')}{\langle\widetilde{Q}\pi'_{k-1},\psi_{k}\rangle}\times[\1_{C_{k}(\Delta)^{\complement}}(x')\widetilde{Q}(\1_{C_{k-1}(\Delta)}\pi'_{(k-1)\tau})(dx,dx')\\
 &  & \,\,+\widetilde{Q}(\1_{C_{k-1}(\Delta)^{c}}\pi'_{(k-1)\tau})(dx,dx')\\
 &  & \,\,+\pi'_{(k-1)\tau}(C_{k-1}(\Delta)^{\complement})\xi_{1}(D_{k},\Delta)\1_{C_{k}(\Delta)}(x')dxdx']\,.
\end{eqnarray*}
We have, by Lemma \ref{lem:encadrement-potentiel},
\begin{multline*}
\E(\1_{[0,\theta^{1-\iota}\Delta]}(|m_{k}-m_{k-1}|)\\
\times\int_{\R^{2}}\frac{\psi_{k}(x,x')}{\langle\widetilde{Q}\pi'_{k-1},\psi_{k}\rangle}\times\1_{C_{k}(\Delta)^{\complement}}(x')\widetilde{Q}(\1_{C_{k-1}(\Delta)}\pi'_{(k-1)\tau})(dx,dx')|Y_{0:(k-1)\tau})\\
=\int_{y_{0:\tau}\in\mathcal{C}([0,\tau])}\1_{[0,\theta^{1-\iota}\Delta]}(|m_{k}(B_{1}^{y_{0:\tau}})-m_{k-1}(B_{1}^{Y_{(k-2)\tau:(k-1)\tau}})|)\\
\times\int_{\R^{2}}\frac{\psi(y_{0:\tau},x,x')}{\langle\widetilde{Q}\pi'_{k-1},\psi(y_{0:\tau},.,.)\rangle}\1_{C_{k}(\Delta,B_{1}^{y_{0:\tau}})^{\complement}}(x')\widetilde{Q}(\1_{C_{k-1}(\Delta)}\pi'_{(k-1)\tau})(dx,dx')\\
\times\left(\int_{\R^{2}}\widetilde{Q}\pi_{(k-1)\tau}(du,du')\psi(y_{0:\tau},u,u')\right)\lambda_{W}(dy_{0:\tau})\\
\mbox{(using Equations (\ref{eq:encadrement-rapport-mesures}), (\ref{eq:hilbert-non-croissante}))}\\
\leq e^{2h(\pi_{0},\pi_{0}')}\int_{y_{0:\tau}\in\mathcal{C}([0,\tau])}\1_{[0,\theta^{1-\iota}\Delta]}(|m_{k}(B_{1}^{y_{0:\tau}})-m_{k-1}(B_{1}^{Y_{(k-2)\tau:(k-1)\tau}})|)\\
\times\int_{\R^{2}}\psi(y_{0:\tau},x,x')\1_{C_{k}(\Delta,B_{1}^{y_{0:\tau}})^{\complement}}(x')\1_{C_{k-1}(\Delta)}(x)\widetilde{Q}\pi{}_{k-1}(dx,dx')\lambda_{W}(dy_{0:\tau})\\
\mbox{(using (\ref{eq:maj-compact-01}), (\ref{eq:star-01}), (\ref{eq:maj-compact-02}), (\ref{eq:star-02}) and the fact that }C_{k-1}(\Delta)\subset C_{k-1}(2\Delta)\mbox{)}\\
\preceq T(\Delta)e^{2h(\pi_{0},\pi_{0}')}\,.
\end{multline*}
We have, in the same way, 
\begin{multline*}
\E\left(\left.\int_{\R^{2}}\frac{\psi_{k}(x,x')}{\langle\widetilde{Q}\pi'_{k-1},\psi_{k}\rangle}\times\widetilde{Q}(\1_{C_{k-1}(\Delta)^{\complement}}\pi'_{(k-1)\tau})(dx,dx')\right|Y_{0:(k-1)\tau}\right)\\
\leq e^{h(\pi_{0},\pi_{0}')}\int_{y_{0:\tau}\in\mathcal{C}([0,\tau])}\int_{\R^{2}}\psi(y_{0:\tau},x,x')\1_{C_{k-1}(\Delta)^{c}}(x)\widetilde{Q}\pi'_{(k-1)\tau}(dx,dx')\lambda_{W}(dy_{0:\tau})\\
=e^{h(\pi_{0},\pi_{0}')}\pi'_{(k-1)\tau}(C_{k-1}(\Delta)^{\complement})
\end{multline*}
and 
\begin{multline*}
\E\left(\left.\int_{\R^{2}}\frac{\psi_{k}(x,x')}{\langle\widetilde{Q}\pi'_{k-1},\psi_{k}\rangle}\times\pi'_{(k-1)\tau}(C_{k-1}(\Delta)^{\complement})\xi_{1}(D_{k},\Delta)\1_{C_{k}(\Delta,B_{1}^{y})}(x')dxdx'\right|Y_{0:(k-1)\tau}\right)\\
\leq e^{h(\pi_{0},\pi_{0}')}\int_{y_{0:\tau}\in\mathcal{C}([0,\tau])}\int_{\R^{2}}\psi(y_{0:\tau},x,x')\pi'_{(k-1)\tau}(C_{k-1}(\Delta)^{\complement})\xi_{1}(m_{k}(B_{1}^{y_{0:\tau}})-m_{k-1},\Delta)\\
\times\1_{C_{k}(\Delta,B_{1}^{y})}(x')dxdx'\lambda_{W}(dy_{0:\tau})\\
\leq e^{h(\pi_{0},\pi_{0}')}\pi'_{(k-1)\tau}(C_{k-1}(\Delta)^{\complement})\,.
\end{multline*}
So, using Equation (\ref{eq:maj-distance-compacts}) and Corollary
\ref{cor:erreur-locale-pi-prime}, we get the result.
\end{proof}

\section{New formula for the truncated filter \label{sec:New-formula-for}}

We have reduced the problem to a discrete-time problem. For all $n$,
$\pi_{n\tau}$ is the marginal of a Feynman-Kac sequence based on
the transition $\widetilde{Q}$ and the potentials $(\psi_{k})_{k\geq1}$
(see Equations (\ref{eq:def-psi_k}), (\ref{eq:def-R_n}), (\ref{eq:F-K-sequence}),
Section \ref{subsec:Notations} for the definition of $\widetilde{Q}$,
Section \ref{subsec:Representation-of-the} for the definition of
a Feynman-Kac sequence). As in \cite{oudjane-rubenthaler-2005}, we
restrict the state space to the compacts $(C_{k}(\Delta))_{k\geq0}$.
If $\widetilde{Q}$ restricted to compacts was mixing, then, due to
Proposition \ref{prop:representation-FK-sequence} (and to the remark
below Equation (\ref{eq:def-Rfrak_k})), $\pi_{n\tau}^{\Delta}$ could
be viewed as the law at time $n$ of a Markov chain with contracting
Markov kernels ; and so Lemma \ref{lem:synthese-des-resultats} would
be relatively easy to prove. By construction, $\widetilde{Q}$ restricted
to some compacts cannot be mixing. This is an effect of the fact that
the observations are continuous in time.

The purpose of this Section is to find another representation of the
sequence $(\pi_{n\tau})_{n\geq0}$ as a Feynman-Kac sequence, in such
a way that the underlying Markov operators would be mixing, when restricted
to compacts. Looking at Equation (\ref{eq:def-F-K}), we see that
a Feynman-Kac sequence is a result of the deformation of a measure
on trajectories (we weight the trajectories with potentials $(\psi_{k})_{k\geq1}$).
The main idea of the following is to incorporate the deformations
delicately (in two steps), in order to retain something of the mixing
property of the operator $Q$ (which is mixing when restricted to
compacts).

In this Section, we work with a fixed observation $(Y_{s})_{s\geq0}=(y_{s})_{s\geq0}$.

\subsection{Filter based on partial information}

We define, for all $x=(x_{1},x_{2})$, $x'=(x'_{1},x'_{2})$ in $\R^{2},$
$k$ in $\N^{*}$, $n$ in $\N^{*}$, $n\geq k$, (recall that $\psi_{k}^{\Delta}$
is defined in Equation (\ref{eq:def-psi-Delta}) and that $\xi_{1}$,
$\xi_{2}$ are defnied in (\ref{eq:def-xi_1}), (\ref{eq:def-xi_2}))
\begin{equation}
\mbox{for }k\geq2\,,\,\widetilde{R}_{k}^{\Delta}(x,dx')=\begin{cases}
\1_{C_{k-1}(\Delta)}(x'_{1})\psi_{k}^{\Delta}(x')\widetilde{Q}^{2}(x,dx') & \mbox{ if }x{}_{2}\in C_{k-2}(\Delta)\,,\\
\1_{C_{k-1}(\Delta)}(x'_{1})\psi_{k}^{\Delta}(x')\xi_{1}(D_{k-1},\Delta)dx' & \mbox{ otherwise }\,,
\end{cases}\label{eq:def-R-tilde-Delta}
\end{equation}
\begin{equation}
\psi_{2n|2k}^{\Delta}(x)=\begin{cases}
\widetilde{R}_{2n}^{\Delta}\widetilde{R}_{2n-2}^{\Delta}\dots\widetilde{R}_{2k+2}^{\Delta}(x,\R^{2}) & \mbox{ if }k\leq n-1\,,\\
1 & \mbox{if }k=n\,,
\end{cases}\label{eq:def-psi-Delta-1}
\end{equation}
(so $\psi_{2n|2k}^{\Delta}(x)$ does not depend on $x_{1}$),
\[
S_{2n|2k}^{\Delta}(x,dx')=\begin{cases}
\frac{\psi_{2n|2k+2}^{\Delta}(x')}{\psi_{2n|2k}^{\Delta}(x)}\widetilde{R}_{2k+2}^{\Delta}(x,dx') & \mbox{ if }k\leq n-1\,,\\
dx' & \mbox{ if }k=n-1\,.
\end{cases}
\]
These notations come from \cite{del-moral-guionnet-2001}. As $Q$
has a density with respect to the Lebesgue measure on $\R$, so has
$S_{2n|2k}^{\Delta}$ (with respect to the Lebesgue measure on $\R^{2}$).
We write $(x,x')\in E^{2}\mapsto S_{2n|2k}^{\Delta}(x,x')$ for this
density. We fix $n$ in $\N^{*}$ in the rest of this subsection and
in the following subsection. We define $S_{2n|2k}^{\Delta,(p)}$,
$\psi_{2n|2k}^{\Delta,(p)}$, for $0\leq k\leq n$, with the same
formulas used above to define $S_{2n|2k}^{\Delta}$, $\psi_{2n|2k}^{\Delta}$,
except we replace $\psi_{2n}^{\Delta}$ by $1$. %
For all $D>0$, we set 
\begin{equation}
\epsilon(D,\Delta)=\frac{\xi_{1}(D,\Delta)}{\xi_{2}(D,\Delta)}\,,\label{eq:def-epsilon}
\end{equation}
and, for all $k$, 
\[
\epsilon_{k}=\epsilon(D_{k},\Delta)\,.
\]
\begin{lem}
For $k\leq n-1$, $S_{2n|2k}^{\Delta}$ is a Markov operator and $S_{2n|2k}^{\Delta}$
is $(1-\epsilon_{2k+1}^{2})$-contracting for the total variation
norm, $S_{2n|2k}^{\Delta,(p)}$ is a Markov operator and $S_{2n|2k}^{\Delta,(p)}$
is $(1-\epsilon_{2k+1}^{2})$-contracting for the total variation
norm
\end{lem}

\begin{proof}
We write the proof only for the kernels $S_{\dots}^{\Delta}$, it
would be very similar for the kernels $S_{\dots}^{\Delta,(p)}$. %
{} By Proposition \ref{prop:representation-FK-sequence}, $S_{2n|2k}^{\Delta}$
is a Markov operator. We set, for all $k\geq1$, $x_{1}$, $x_{2}$
in $\R$, 
\[
\lambda_{k}(dx_{1},dx_{2})=\1_{C_{k-1}(\Delta)}(x_{1})\1_{C_{k}(\Delta)}(x_{2})\psi_{k}(x_{1},x_{2})Q(x_{1},x_{2})dx_{1}dx_{2}\,.
\]
By Lemma \ref{lem:encadrement-transition}, we have, for all $x_{1}$,
$x_{2}$, $z_{1}$, $z_{2}$ in $\R$, $k\geq2$ (we use here the
second line of Equation (\ref{eq:def-R-tilde-Delta})) 
\begin{equation}
\xi_{1}(D_{k-1},\Delta)\lambda_{k}(dz_{1},dz_{2})\leq\widetilde{R}_{k}^{\Delta}(x_{1},x_{2},dz_{1},dz_{2})\leq\xi_{2}(D_{k-1},\Delta)\lambda_{k}(dz_{1},dz_{2})\,.\label{eq:R-tilde-mixing}
\end{equation}
So $\widetilde{R}_{k}^{\Delta}$ is $\sqrt{\epsilon_{k-1}}$-mixing.
So, for all $x$ in $\R^{2}$, all $k$ such that $0\leq k\leq n-1$
(the convention being that, if $k=n-1$, $(\tdR_{2n}\dots\tdR_{2k+4})(y,dz)=\delta_{y}(dz)$)
\begin{eqnarray*}
\psi_{2n|2k}^{\Delta}(x) & = & \int_{\R^{2}\times\R^{2}}(\widetilde{R}_{2n}^{\Delta}\dots\widetilde{R}_{2k+4}^{\Delta})(y,dz)\widetilde{R}_{2k+2}^{\Delta}(x,dy)\\
 & \leq & \int_{\R^{2}\times\R^{2}}(\widetilde{R}_{2n}^{\Delta}\dots\widetilde{R}_{2k+4}^{\Delta})(y,dz)\xi_{2}(D_{2k+1},\Delta)\lambda_{2k+2}(dy)\,,
\end{eqnarray*}
and, for $x'$ in $\R^{2}$,
\[
\widetilde{R}_{2k+2}^{\Delta}(x,dx')\geq\xi_{1}(D_{2k+1},\Delta)\lambda_{2k+2}(dx')\,,
\]
so 
\begin{equation}
S_{2n|2k}^{\Delta}(x,dx')\geq\frac{\xi_{1}(D_{2k+1},\Delta)}{\xi_{2}(D_{2k+1},\Delta)}\times\frac{\widetilde{R}_{2k+4:2n}^{\Delta}(x',\R^{2})\lambda_{2k+2}(dx')}{\int_{\R^{2}}\widetilde{R}_{2k+4:2n}^{\Delta}(y,\R)\lambda_{2k+2}(dy)}\,.\label{eq:min-S-Delta}
\end{equation}
 In the same way as above, we can also obtain
\begin{equation}
S_{2n|2k}^{\Delta}(x,dx')\leq\frac{\xi_{2}(D_{2k+1},\Delta)}{\xi_{1}(D_{2k+1},\Delta)}\times\frac{\widetilde{R}_{2k+4:2n}(x',\R^{2})\lambda_{2k+2}(dx')}{\int_{\R^{2}}\widetilde{R}_{2k+4:2n}(y,\R^{2})\lambda_{2k+2}(dy)}\,.\label{eq:maj-S-Delta}
\end{equation}
This implies that $S_{2n|2k}^{\Delta}$ is $(1-\epsilon_{2k+1}^{2})$-contracting
for the total variation norm (see Subsection \ref{subsec:Notations}).
One can also use Proposition \ref{prop:representation-FK-sequence}
to prove this result. We did it this way because we will re-use Equations
(\ref{eq:min-S-Delta}), (\ref{eq:maj-S-Delta}).
\end{proof}
We set $Z_{0}$ to be of the form $Z_{0}=(0,Z_{0}^{(2)})$, with $Z_{0}^{(2)}$
a random variable. We set $(Z_{2k})_{0\leq k\leq n}$ to be a non-homogeneous
Markov chain with kernels $S_{2n|0}^{\Delta}$, $S_{2n|2}^{\Delta}$,
\ldots{}, $S_{2n|2n-2}^{\Delta}$ (for $k$ in $\{1,2,\dots,n\}$,
the law of $Z_{2k}$ knowing $Z_{2k-2}$ is $S_{2n|2k-2}^{\Delta}(Z_{2k-2},.)$).
For $Z_{2k}$ being a element of this chain, we denote by $Z_{2k}^{(1)}$
and $Z_{2k}^{(2)}$ its first and second component respectively. Recalling
Proposition \ref{prop:representation-FK-sequence} (or Proposition
3.1, p. 428 in \cite{oudjane-rubenthaler-2005}, or similar results
in \cite{del-moral-guionnet-2001}), if the law of $Z_{0}$ is chosen
properly, then $Z_{2n}^{(2)}$ has the same law as $X_{2n}$ knowing
$Y_{\tau:2\tau}$, \ldots{}, $Y_{(2n-1)\tau:2n\tau}$, hence the
title of this Subsection. %
\begin{rem}
\label{rem:valeurs-dans-compact}We have that, for all $k\geq1$,
$Z_{2k}^{(2)}$ takes values in $C_{2k}(\Delta)$ and $Z_{2k}^{(1)}$
takes values in $C_{2k-1}(\Delta).$
\end{rem}

We set $(Z_{2k}^{(p)})_{0\leq k\leq n}$ to be a non-homogeneous Markov
chain with $Z_{0}^{(p)}=Z_{0}$ and with kernel $S_{2n|0}^{\Delta,(p)}$,
$S_{2n|2}^{\Delta,(p)}$, \ldots{}, $S_{2n|2n-2}^{\Delta,(p)}$.

We set $U_{2k+1}=(Z_{2k}^{(2)},Z_{2k+2}^{(1)})$ for $k\in\{0,1,\dots,n-1\}$
and $U_{2n+1}^{(1)}=Z_{2n}^{(2)}$. We set $U_{2k+1}^{(p)}=(Z_{2k}^{(p)(2)},Z_{2k+2}^{(p)(1)})$
for $k\in\{0,1,2,\dots,n-1\}$.

\subsection{New Markov chain}
\begin{lem}
\label{lem:U-markov}The sequence $(U_{1},U_{3},\dots,U_{2n-1},U_{2n+1}^{(1)})$
is a non-homogeneous Markov chain. The sequence $(U_{1}^{(p)},U_{3}^{(p)},\dots,U_{2n-3}^{(p)},U_{2n-1}^{(p)})$
is a non-homogeneous Markov chain. 
\end{lem}

If $Z_{0}^{(2)}$ is of law $\mu$, then the law of $U_{1}$ is given
by, for all $(z,z')$ in $\R^{2}$, 
\begin{equation}
\p(U_{1}\in(dz,dz'))=\int_{x\in\R}S_{2n|0}^{\Delta}((0,z),(dz',dx))\mu(dz)\,.\label{eq:loi-de-U1}
\end{equation}

We write $S_{2n|2k+1}^{U}$ for the transition kernel between $U_{2k-1}$
and $U_{2k+1}$ (for $k=1,2,\dots,n-1$) and $S_{2n|2n+1}^{U}$ for
the transition between $U_{2n-1}$ and $U_{2n+1}^{(1)}$. We write
$S_{2n|2k+1}^{(p)U}$ for the transition kernel between $U_{2n|2k-1}^{(p)}$
and $U_{2n|2k+1}^{(p)}$ (for $k=1,2,\dots,n-1$)
\begin{proof}
We write the proof only for $(U_{1},U_{3},\dots,U_{2n-1},U_{2n+1}^{(1)})$,
it would be very similar for the sequence $(U_{1}^{(p)},U_{3}^{(p)},\dots,U_{2n-1}^{(p)})$.
Let $\varphi$ be a test function (in $\mathcal{C}_{b}^{+}(\R^{2})$).
For $k\in\{1,\dots,n-1\}$, we have (for $z_{0}^{(2)}$, $z_{2}^{(1)}$,
\ldots{} ,$z_{2k}^{(1)}$ in $\R$)
\begin{multline*}
\E(\varphi(U_{2k+1})|U_{1}=(z_{0}^{(2)},z_{2}^{(1)}),\dots,U_{2k-1}=(z_{2k-2}^{(2)},z_{2k}^{(1)}))=\\
\E(\varphi(Z_{2k}^{(2)},Z_{2k+2}^{(1)})|Z_{0}^{(2)}=z_{0}^{(2)},\dots,Z_{2k-2}^{(2)}=z_{2k-2}^{(2)},Z_{2k}^{(1)}=z_{2k}^{(1)})=\\
\E(\varphi(Z_{2k}^{(2)},Z_{2k+2}^{(1)})|Z_{2k-2}^{(2)}=z_{2k-2}^{(2)},Z_{2k}^{(1)}=z_{2k}^{(1)})\,,
\end{multline*}
as $S_{2n|2k-2}^{\Delta}(z_{2k-2}^{(1)},z_{2k-2}^{(2)},.,.)$ does
not depend on $z_{2k-2}^{(1)}$. So the quantity above is equal to,
for any $z\in C_{2k-1}(\Delta)$,
\begin{multline*}
\int_{\R^{2}}\varphi(z_{2k}^{(2)},z_{2k+2}^{(1)})\frac{S_{2n|2k-2}^{\Delta}((z,z_{2k-2}^{(2)}),(z_{2k}^{(1)},z_{2k}^{(2)}))}{\int_{\R}S_{2n|2k-2}^{\Delta}((z,z_{2k-2}^{(2)}),(z_{2k}^{(1)},z'))dz'}\\
\times\left(\int_{\R}S_{2n|2k}^{\Delta}((z_{2k}^{(1)},z_{2k}^{(2)}),(z_{2k+2}^{(1)},z_{2k+2}^{(2)}))dz_{2k+2}^{(2)}\right)dz_{2k}^{(2)}dz_{2k+2}^{(1)}\,.
\end{multline*}
A similar computation can be made for $\E(\varphi(U_{2n+1}^{(1)})|U_{1},\dots U_{2n-1})$. 
\end{proof}
We set, for all $k$, 
\begin{multline}
\epsilon'(D,\Delta)=\exp\left[-\left(\frac{B_{2}p_{2,1}^{2}}{2}+\frac{1}{4\tau}\right)\left(\frac{\Delta}{B_{2}(1+p_{2,1})}+D\right)^{2}\right.\\
\left.-\left(\Delta p_{2,1}+3M\right)\left(\frac{\Delta}{B_{2}(1+p_{2,1})}+D\right)-3\tau\left(\frac{M}{2}+\frac{M^{2}}{4}\right)\right]\,,\label{eq:def-epsilon-prime}
\end{multline}
\[
\epsilon'_{k}=\epsilon'(|m_{k}-m_{k-1}|,\Delta)\,.
\]
\begin{prop}
\label{lem:U-mixing}For any $k=1,2,\dots,n$, the Markov kernel $S_{2n|2k+1}^{U}$
is $(\epsilon_{2k-1}^{2}(\epsilon'_{2k})^{2})$-contracting. For any
$k=1,2,\dots,n-1$, the Markov kernel $S_{2n|2k+1}^{(p)U}$ is is
$(\epsilon_{2k-1}^{2}(\epsilon'_{2k})^{2})$-contracting.
\end{prop}

Before going into the proof of the above Proposition, we need the
following technical result. We are interested in the bounds appearing
in Lemma \ref{lem:borne-integrale}. We suppose that $t_{1}$, $t_{2}$,
$x$, $z$ in $\R$ are fixed. To simplify the computations, we introduce
the following notations:
\[
\left(\begin{array}{c}
t'_{1}\\
t'_{2}
\end{array}\right)=(P^{-1})^{T}\left(\begin{array}{c}
t_{1}\\
t_{2}
\end{array}\right)\,,
\]
\[
M_{1}=\frac{2M(|1+p_{2,1}|)}{p_{1,1}}\,.
\]
\begin{lem}
\label{lem:encadrement-potential-separation}Suppose that, for some
$D\geq0$, 
\[
|x-z|\leq D\,,
\]
\[
\left|2B_{2}(p_{1,2}+1)z-t_{2}'\right|\leq\Delta
\]
then
\begin{multline*}
\exp\left(-\frac{1}{4B_{2}}(t_{2}'-2B_{2}(p_{2,1}+1)z)^{2}-B_{2}p_{2,1}^{2}D^{2}-p_{2,1}D\Delta\right)\\
\leq\exp\left(-\frac{1}{4B_{2}}(t_{2}'-2B_{2}(p_{2,1}x+z))^{2}\right)\\
\leq\exp\left(-\frac{1}{4B_{2}}(t_{2}'-2B_{2}(p_{2,1}+1)z)^{2}+p_{2,1}D\Delta\right)\,.
\end{multline*}
\end{lem}

\begin{proof}
We have
\begin{multline*}
\exp\left(-\frac{1}{4B_{2}}(t_{2}'-2B_{2}(p_{2,1}x+z))^{2}\right)\\
=\exp\left(-\frac{1}{4B_{2}}(t_{2}'-2B_{2}(p_{2,1}+1)z)+2B_{2}p_{2,1}(z-x))^{2}\right)\\
\leq\exp\left(-\frac{1}{4B_{2}}(t_{2}'-2B_{2}(p_{2,1}+1)z)^{2}-B_{2}p_{2,1}^{2}(z-x)^{2}+|p_{2,1}(z-x)|\times|t_{2}'-2B_{2}(p_{2,1}+1)z|\right)\\
\leq\exp\left(-\frac{1}{4B_{2}}(t_{2}'-2B_{2}(p_{2,1}+1)z)^{2}+p_{2,1}D\Delta\right)
\end{multline*}
and
\begin{multline*}
\exp\left(-\frac{1}{4B_{2}}(t_{2}'-2B_{2}(p_{2,1}x+z))^{2}\right)\\
\geq\exp\left(-\frac{1}{4B_{2}}(t_{2}'-2B_{2}(p_{2,1}+1)z)^{2}-B_{2}p_{2,1}^{2}(z-x)^{2}-|p_{2,1}(z-x)|\times|t_{2}'-2B_{2}(p_{2,1}+1)z|\right)\\
\geq\exp\left(-\frac{1}{4B_{2}}(t_{2}'-2B_{2}(p_{2,1}+1)z)^{2}-B_{2}p_{2,1}^{2}D^{2}-p_{2,1}D\Delta\right)
\end{multline*}
\end{proof}

\begin{proof}[Proof of Proposition \ref{lem:U-mixing}]

We write the proof in the case $k\in\{1,2,\dots,n-1\}$ and for $S_{2n|2k+1}^{U}$
(the other cases being very similar). Let $\varphi$ be a test function
(in $\mathcal{C}_{b}^{+}(\R)$). By Remark \ref{rem:valeurs-dans-compact},
we have that $U_{2k-1}^{(2)}$ takes its values in $C_{2k-1}(\Delta)$.
We write, for any $z_{2k-2}^{(2)}\in\R$, $z_{2k}^{(1)}\in C_{2k-1}(\Delta)$,
$z\in\R$ (like in the proof of Lemma \ref{lem:U-markov})
\begin{multline}
\E(\varphi(U_{2k+1})|U_{2k-1}=(z_{2k-2}^{(2)},z_{2k}^{(1)}))=\\
\int_{\R^{2}}\varphi(z_{2k}^{(2)},z_{2k+2}^{(1)})\frac{S_{2n|2k-2}^{\Delta}((z,z_{2k-2}^{(2)}),(z_{2k}^{(1)},z_{2k}^{(2)}))}{\int_{\R}S_{2n|2k-2}^{\Delta}((z,z_{2k-2}^{(2)}),(z_{2k}^{(1)},z'))dz'}\\
\times\left(\int_{\R}S_{2n|2k}^{\Delta}((z_{2k}^{(1)},z_{2k}^{(2)}),(z_{2k+2}^{(1)},z_{2k+2}^{(2)})dz_{2k+2}^{(2)}\right)dz_{2k}^{(2)}dz_{2k+2}^{(1)}\geq\\
\mbox{(by Equations (\ref{eq:min-S-Delta}), (\ref{eq:maj-S-Delta})) }\\
\int_{\R^{2}}\varphi(z_{2k}^{(2)},z_{2k+2}^{(1)})\epsilon_{2k-1}^{2}\\
\times\frac{\widetilde{R}_{2n:2k+2}^{\Delta}((z_{2k}^{(1)},z_{2k}^{(2)}),\R^{2})\1_{C_{2k-1}(\Delta)}(z_{2k}^{(1)})\1_{C_{2k}(\Delta)}(z_{2k}^{(2)})\psi_{2k}(z_{2k}^{(1)},z_{2k}^{(2)})Q(z_{2k}^{(1)},z_{2k}^{(2)})}{\int_{\R}\tdR_{2n:2k+2}((z_{2k}^{(1)},z'),\R^{2})\1_{C_{2k-1}(\Delta)}(z_{2k}^{(1)})\1_{C_{2k}(\Delta)}(z')\psi_{2k}(z_{2k}^{(1)},z')Q(z_{2k}^{(1)},z')dz'}\\
\times\left(\int_{\R}S_{2n|2k}^{\Delta}((z_{2k}^{(1)},z_{2k}^{(2)}),(z_{2k+2}^{(1)},z_{2k+2}^{(2)})dz_{2k+2}^{(2)}\right)dz_{2k}^{(2)}dz_{2k+2}^{(1)}\label{eq:esp-cond-U-01}
\end{multline}
From Lemma \ref{lem:encadrement-potentiel} and using the same kind
of computation as in the proof of Lemma  \ref{lem:borne-integrale}
and Equation (\ref{eq:formule-psi_chapeau}), we get, for all $z_{2k}$
such that $z_{2k}^{(2)}\in C_{2k}(\Delta)$,
\begin{multline*}
\psi_{2k}(z_{2k}^{(1)},z_{2k}^{(2)})\geq\sigma_{1}\sigma_{2}e^{-2M|z_{2k}^{(1)}-z_{2k}^{(2)}|-\tau\left(M+\frac{M^{2}}{2}\right)}\\
\times\exp\left(-A_{2}(z_{2k}^{(1)})^{2}-B_{2}(z_{2k}^{(2)})^{2}+C_{1}z_{2k}^{(1)}z_{2k}^{(2)}+A_{1}^{y_{(2k-1)\tau:(2k)\tau}}z_{2k}^{(1)}+B_{1}^{y_{(2k-1)\tau:(2k)\tau}}z_{2k}^{(2)}+C_{0}^{y_{(2k-1)\tau:(2k)\tau}}\right)\\
=\sigma_{1}\sigma_{2}e^{-2M|z_{2k}^{(1)}-z_{2k}^{(2)}|-\tau\left(M+\frac{M^{2}}{2}\right)}\times\exp\left[-\frac{1}{4A_{2}}\left(\frac{A_{1}^{y_{(2k-1)\tau:(2k)\tau}}}{p_{1,1}}-\frac{p_{2,1}B_{1}^{y_{(2k-1)\tau:(2k)\tau}}}{p_{1,1}}-2A_{2}p_{1,1}z_{2k}^{(1)}\right)^{2}\right.\\
-\frac{1}{4B_{2}}\left(B_{1}^{y_{(2k-1)\tau:(2k)\tau}}-2B_{2}(p_{2,1}z_{2k}^{(1)}+z_{2k}^{(2)})\right)^{2}+C_{0}^{y_{(2k-1)\tau:(2k)\tau}}\\
\left.+\frac{1}{4}(A_{1}^{y_{(2k-1)\tau:(2k)\tau}},B_{1}^{y_{(2k-1)\tau:(2k)\tau}})\kappa^{-1}(A_{1}^{y_{(2k-1)\tau:(2k)\tau}},B_{1}^{y_{(2k-1)\tau:(2k)\tau}})^{T}\right]\\
\geq\mbox{(by Lemma \ref{lem:encadrement-potential-separation}) }\sigma_{1}\sigma_{2}e^{-2M|z_{2k}^{(1)}-z_{2k}^{(2)}|+C_{0}^{y_{(2k-1)\tau:(2k)\tau}}}\times e^{-\tau\left(M+\frac{M^{2}}{2}\right)}\\
\times\exp\left[-\frac{1}{4A_{2}}\left(\frac{A_{1}^{y_{(2k-1)\tau:(2k)\tau}}}{p_{1,1}}-\frac{p_{2,1}B_{1}^{y_{(2k-1)\tau:(2k)\tau}}}{p_{1,1}}-2A_{2}p_{1,1}z_{2k}^{(1)}\right)^{2}\right.\\
-\frac{1}{4B_{2}}\left(B_{1}^{y_{(2k-1)\tau:(2k)\tau}}-2B_{2}(p_{2,1}+1)z_{2k}^{(2)}\right)^{2}-B_{2}p_{2,1}^{2}(z_{2k}^{(1)}-z_{2k}^{(2)})^{2}-\Delta p_{2,1}|z_{2k}^{(1)}-z_{2k}^{(2)}|\\
\left.+\frac{1}{4}(A_{1}^{y_{(2k-1)\tau:(2k)\tau}},B_{1}^{y_{(2k-1)\tau:(2k)\tau}})\kappa^{-1}(A_{1}^{y_{(2k-1)\tau:(2k)\tau}},B_{1}^{y_{(2k-1)\tau:(2k)\tau}})^{T}\right]\,.
\end{multline*}
We set 
\[
\psi_{2k}^{(1)}(x)=\exp\left(-\frac{1}{4A_{2}}\left(\frac{A_{1}^{y_{(2k-1)\tau:(2k)\tau}}}{p_{1,1}}-\frac{p_{2,1}B_{1}^{y_{(2k-1)\tau:(2k2)\tau}}}{p_{1,1}}-2A_{2}p_{1,1}x\right)^{2}\right)\,,
\]
\[
\psi_{2k}^{(2)}(x)=\exp\left(-\frac{1}{4B_{2}}\left(B_{1}^{y_{(2k-1)\tau:(2k)\tau}}-2B_{2}(p_{2,1}+1)x\right)^{2}\right)
\]
In the same way as above:
\begin{multline*}
\psi_{2k}(z_{2k}^{(1)},z_{2k}^{(2)})\leq\sigma_{1}\sigma_{2}e^{2M|z_{2k}^{(1)}-z_{2k}^{(2)}|+\tau\left(M+\frac{M^{2}}{2}\right)+C_{0}^{y_{(2k-1)\tau:(2k)\tau}}}\psi_{2k}^{(1)}(z_{2k}^{(1)})\psi_{2k}^{(2)}(z_{2k}^{(2)})\\
\times\exp\left(\Delta p_{2,1}|z_{2k}^{(1)}-z_{2k}^{(2)}|+\frac{1}{4}(A_{1}^{y_{(2k-1)\tau:(2k)\tau}},B_{1}^{y_{(2k-1)\tau:(2k)\tau}})\kappa^{-1}(A_{1}^{y_{(2k-1)\tau:(2k)\tau}},B_{1}^{y_{(2k-1)\tau:(2k)\tau}})^{T}\right)\,.
\end{multline*}

From Lemma \ref{lem:encadrement-transition}, we get for all $z_{2k}^{(1)}\in C_{2k-1}(\Delta)$,
$z_{2k}^{(2)}\in C_{2k}(\Delta)$ ,
\begin{multline*}
\frac{1}{\sqrt{2\pi\tau}}\exp\left(-\frac{1}{2\tau}\left(\frac{\Delta}{B_{2}(1+p_{2,1})}+D_{2k}\right)^{2}-M\left(\frac{\Delta}{B_{2}(1+p_{2,1})}+D_{2k}\right)-\tau\left(\frac{M}{2}+\frac{M^{2}}{2}\right)\right)\\
\leq Q(z_{2k}^{(1)},z_{2k}^{(2)})\leq\frac{1}{\sqrt{2\pi\tau}}\exp\left(M\left(\frac{\Delta}{B_{2}(1+p_{2,1})}+D_{2k}\right)+\frac{M\tau}{2}\right)
\end{multline*}

Looking back at (\ref{eq:esp-cond-U-01}), we get%
\begin{multline*}
\E(\varphi(U_{2k+1})|U_{2k-1}=(z_{2k-2}^{(2)},z_{2k}^{(1)}))\geq\\
\int_{\R^{2}}\varphi(z_{2k}^{(2)},z_{2k+2}^{(1)})\epsilon_{2k-1}^{2}(\epsilon_{2k}')^{2}\frac{\widetilde{R}_{2n:2k+2}^{\Delta}((z_{2k}^{(1)},z_{2k}^{(2)}),\R^{2})\1_{C_{2k-1}(\Delta)}(z_{2k}^{(1)})\1_{C_{2k}(\Delta)}(z_{2k}^{(2)})\psi_{2k}^{(2)}(z_{2k}^{(2)})}{\int_{\R}\tdR_{2n:2k+2}((z_{2k}^{(1)},z'),\R^{2})\1_{C_{2k-1}(\Delta)}(z_{2k}^{(1)})\1_{C_{2k}(\Delta)}(z_{2k}^{(2)})\psi_{2k}^{(2)}(z')dz'}\\
\times\left(\int_{\R}S_{2n|2k}^{\Delta}((z_{2k}^{(1)},z_{2k}^{(2)}),(z_{2k+2}^{(1)},z_{2k+2}^{(2)})dz_{2k+2}^{(2)}\right)dz_{2k}^{(2)}dz_{2k+2}^{(1)}\,.
\end{multline*}
As $R_{2n:2k+2}((z_{2k}^{(1)},z'),.)$ and $S_{2n|2k}^{\Delta}((z_{2k}^{(1)},z'),.)$
do not depend on $z_{2k}^{(1)}$ for any  $z'$, we get that $S_{2n|2k+1}^{U}$
is $(1-\epsilon_{2k-1}^{2}(\epsilon'_{2k})^{2})$-contracting (remember
Section \ref{subsec:Notations}).

\end{proof}

\subsection{New representation }
\begin{prop}
\label{prop:representation-avec-U}Let $n\geq1$. If we suppose that
$Z_{0}^{(2)}$ is of law $\psi_{2n|0}^{\Delta}(0,.)\bullet\mu$, then,
for all test function $\varphi$ (in $\mathcal{C}_{b}^{+}(\R)$),
\begin{equation}
\frac{\E(\varphi(U_{2n+1}^{(1)})\prod_{1\leq i\leq n}\psi_{2i-1}^{\Delta}(U_{2i-1}))}{\E(\prod_{1\leq i\leq n}\psi_{2i-1}^{\Delta}(U_{2i-1}))}=\left(\overline{R}_{2n}^{\Delta}\overline{R}_{2n-1}^{\Delta}\dots\overline{R}_{1}^{\Delta}(\mu)\right)(\varphi)\,.\label{eq:representation-U}
\end{equation}
If we suppose that $Z_{0}^{(2)}$ is of law $\psi_{2n|0}^{\Delta,(p)}(0,.)\bullet\mu$,
then, for all test function $\varphi$ (in $\mathcal{C}_{b}^{+}(\R)$),
\begin{equation}
\frac{\E(\varphi(U_{2n-1}^{(p)(2)})\prod_{1\leq i\leq n}\psi_{2i-1}^{\Delta}(U_{2i-1}^{(p)}))}{\E(\prod_{1\leq i\leq n}\psi_{2i-1}^{\Delta}(U_{2i-1}^{(p)}))}=\left(\overline{R}_{2n-1}^{\Delta}\overline{R}_{2n-2}^{\Delta}\dots\overline{R}_{1}^{\Delta}(\mu)\right)(\varphi)\,.\label{eq:representation-U(p)}
\end{equation}
\end{prop}

\begin{rem}
\label{rem:sur-representation}Recall that we are working with a fixed
observation $(Y_{s})_{s\geq0}=(y_{s})_{s\geq0}$. The above Proposition
tells that, for all $n$, $\overline{R}_{n}^{\Delta}\overline{R}_{n-1}^{\Delta}\dots\overline{R}_{1}^{\Delta}(\mu)$
can be written as the $n$-th term of a Feynman-Kac sequence based
on mixing kernels (by Proposition \ref{lem:U-mixing}). We can apply
Proposition \ref{prop:representation-FK-sequence} to this Feynman-Kac
sequence. This representation and this result are also true for a
measure $\overline{R}_{n}^{\Delta}\overline{R}_{n-1}^{\Delta}\dots\overline{R}_{k}^{\Delta}(\eta)$
for any $k\leq n$, $\eta$ probability measure on $\R$. %
\end{rem}

\begin{proof}
We write the proof only for Equation (\ref{eq:representation-U}).
The computation leading to Equation (\ref{eq:representation-U(p)})
would be very similar. It would simplify nicely because we replace
$\psi_{2n}^{\Delta}$ by $1$ in the definition of the $S_{2n|\dots}^{\Delta(p)}$,
$\psi_{2n|\dots}^{\Delta(p)}$. 

We have, for any test function $\varphi$ (in $\mathcal{C}_{b}^{+}(\R)$),
\begin{multline*}
\E(\varphi(U_{2n+1}^{(1)})\prod_{1\leq i\leq n}\psi_{2i-1}^{\Delta}(U_{2i-1}))=\\
\int_{\R\times(\R^{2})^{n}}\varphi(z_{2n}^{(2)})\prod_{0\leq k\leq n-1}\left[S_{2n|2k}^{\Delta}(z_{2k},z_{2k+2})\psi_{2k+1}^{\Delta}(z_{2k}^{(2)},z_{2k+2}^{(1)})\right]\\
\delta_{0}(dz_{0}^{(1)})(\psi_{2n|0}^{\Delta}(0,.)\bullet\mu)(dz_{0}^{(2)})dz_{2}\dots dz_{2n}=\\
\int_{\R\times(\R^{2})^{n}}\varphi(z_{2n}^{(2)})\prod_{0\leq k\leq n-1}\left[\frac{\psi_{2n|2k+2}^{\Delta}(z_{2k+2})}{\psi_{2n|2k}^{\Delta}(z_{2k})}\widetilde{R}_{2k+2}^{\Delta}(z_{2k},dz_{2k+2})\psi_{2k+1}^{\Delta}(z_{2k}^{(2)},z_{2k+2}^{(1)})\right]\\
\times\delta_{0}(dz_{0}^{(1)})(\psi_{2n|0}^{\Delta}(0,.)\bullet\mu)(dz_{0}^{(2)})dz_{2}\dots dz_{2n}=\\
\int_{\R\times(\R^{2})^{n}}\varphi(z_{2n}^{(2)})\prod_{0\leq k\leq n-1}\left[\widetilde{R}_{2k+2}^{\Delta}(z_{2k},dz_{2k+2})\psi_{2k+1}^{\Delta}(z_{2k}^{(2)},z_{2k+2}^{(1)})\right]\\
\times\frac{1}{\mu(\psi_{2n|0}^{\Delta}(0,.))}\delta_{0}(dz_{0}^{(1)})\mu(dz_{0}^{(2)})dz_{2}\dots dz_{2n}=\\
\int_{\R\times(\R^{2})^{n}}\varphi(z_{2n}^{(2)})\prod_{0\leq k\leq n-1}\left[\psi_{2k+2}^{\Delta}(z_{2k+2}^{(1)},z_{2k+2}^{(2)})\psi_{2k+1}^{\Delta}(z_{2k}^{(2)},z_{2k+2}^{(1)})\widetilde{Q}^{2}(z_{2k},dz_{2k+2})\right]\\
\times\frac{1}{\mu(\psi_{2n|0}^{\Delta}(0,.))}\delta_{0}(dz_{0}^{(1)})\mu(dz_{0}^{(2)})dz_{2}\dots dz_{2n}\,,
\end{multline*}
which proves the desired result (recall Equation (\ref{eq:recursive-representation})).
\end{proof}

\section{Stability results\label{sec:Stability-results} }

In this section, the observations are non longer fixed.

\subsection{Stability of the truncated filter}

We show here that a product of coefficients $\tau_{.}$ decays geometrically
in expectation (see the Lemma below). These coefficients are the contraction
coefficients of the operators $S_{.}^{U}$, $S_{.}^{U,(p)}$, which
are related to the truncated filter through Proposition \ref{prop:representation-avec-U}.
This is why we say that the result below means the stability of the
truncated filter.

We set, for all $t$ in $\R$, $k\geq1$, 
\[
\tau(t,\Delta)=1-(\epsilon'(t,\Delta)\epsilon(t,\Delta))^{2}\,,
\]
\[
\tau_{k}=1-(\epsilon'_{k}\epsilon_{k-1})^{2}\,.
\]
 We set, for $L>0$, 
\[
\widetilde{\alpha}(L)=\frac{192C\sqrt{\tau}}{L\sqrt{\pi}}\exp\left(-\frac{1}{2}\left(\frac{L}{12C\sqrt{2\tau}}\right)^{2}\right)\,.
\]
We fix $L>0$ such that 
\begin{equation}
L>3|m_{0}|+3CM\tau\mbox{ and }\widetilde{\alpha}(L)\leq\frac{1}{4}\,.\label{eq:cond-L}
\end{equation}
We set 
\[
\rho=\frac{\tau(L,\Delta)+\sqrt{\tau(L,\Delta)^{2}+4\widetilde{\alpha}(L)(1-\tau(L,\Delta))}}{2}\,.
\]
\begin{lem}
\label{lem:contraction}For $0\leq k\leq n-1$, we have
\[
\E(\tau_{2n+1}\tau_{2n-1}\dots\tau_{2k+3}|\mathcal{F}_{0:(2k+1)\tau})\leq\left(1-\frac{\epsilon(L,\Delta)^{2}\epsilon'(L,\Delta)^{2}}{2}\right)^{\left\lceil \frac{(n-k-2)_{+}}{2}\right\rceil }\,,
\]
 
\[
\E(\tau_{2n}\tau_{2n-2}\dots\tau_{2k+2}|\mathcal{F}_{0:2k\tau})\leq\left(1-\frac{\epsilon(L,\Delta)^{2}\epsilon'(L,\Delta)^{2}}{2}\right)^{\left\lceil \frac{(n-k-2)_{+}}{2}\right\rceil }\,.
\]
\end{lem}

\begin{proof}
We only write the proof of the second Equation above (the proof of
the other equation is very similar). We take $L>0$ and we set
\begin{eqnarray*}
\theta_{2k} & = & \begin{cases}
\tau(L,\Delta) & \mbox{if }|m_{2k}-m_{2k-1}|<L\mbox{ and }|m_{2k-1}-m_{2k-2}|<L\\
1 & \mbox{otherwise.}
\end{cases}
\end{eqnarray*}
For all $k$, we have $\tau_{2k}\leq\theta_{2k}$. For any $k\geq1$,
$|m_{k}-m_{k-1}|$ is a function of $Y_{(k-2)_{+}\tau:k\tau}$. So,
for all $k$, $\theta_{2k}$ is a function of $Y_{(2k-3)_{+}\tau:2k\tau}$
We fix $k\geq0$ and we define, for $n\geq0$, 
\[
e_{2n|2k+2}=\begin{cases}
\E(\theta_{2n}\theta_{2n-2}\dots\theta_{2k+2}|\mathcal{F}_{2k\tau}) & \mbox{ if }k\leq n-1\,,\\
1 & \mbox{otherwise.}
\end{cases}
\]
We suppose now that $n\geq k+2$. We then have 
\[
e_{2n|2k+2}=\E(\E(\theta_{2n}\theta_{2n-2}|\mathcal{F}_{(2n-3)\tau})\theta_{2n-4}\dots\theta_{2k+2}|\mathcal{F}_{2k\tau})
\]
and
\begin{multline*}
\E(\theta_{2n}\theta_{2n-2}|\mathcal{F}_{(2n-3)\tau})\\
=\E(\theta_{2n-2}(1-\1_{[0,L)}(D_{2n})\1_{[0,L)}(D_{2n-1}))+\tau(L,\Delta)\theta_{2n-2}\1_{[0,L)}(D_{2n})\1_{[0,L)}(D_{2n-1})|\mathcal{F}_{(2n-3)\tau})\\
=\E(\theta_{2n-2}\tau(L,\Delta)+(1-\tau(L,\Delta))\theta_{2n-2}(1-\1_{[0,L)}(D_{2n})\1_{[0,L)}(D_{2n-1}))|\mathcal{F}_{(2n-3)\tau})\\
\leq\tau(L,\Delta)\E(\theta_{2n-2}|\mathcal{F}_{(2n-3)\tau})+(1-\tau(L,\Delta))[\p(|m_{2n}-m_{2n-1}|\geq L|\mathcal{F}_{(2n-3)\tau})\\
+\p(|m_{2n-1}-m_{2n-2}|\geq L|\mathcal{F}_{(2n-3)\tau})]\,.
\end{multline*}
Using Equation (\ref{eq:variation-m_k}), we get 
\begin{multline*}
\E(\theta_{2n}\theta_{2n-2}|\mathcal{F}_{(2n-3)\tau})\leq\tau(L,\Delta)\E(\theta_{2n-2}|\mathcal{F}_{(2n-3)\tau})\\
+(1-\tau(L,\Delta))\left(\p\left(C\mathcal{V}_{(2n-2)\tau,2n\tau}\geq\frac{L}{3}|\mathcal{F}_{(2n-3)\tau}\right)+\p\left(C\mathcal{W}_{(2n-2)\tau,2n\tau}(1+\frac{1}{\theta})\geq\frac{L}{3}|\mathcal{F}_{(2n-3)\tau}\right)\right)\\
+(1-\tau(L,\Delta))\left(\p\left(C\mathcal{V}_{(2n-3)\tau,(2n-1)\tau}\geq\frac{L}{3}|\mathcal{F}_{(2n-3)\tau}\right)+\p\left(C\mathcal{W}_{(2n-3)\tau,(2n-1)\tau}(1+\frac{1}{\theta})\geq\frac{L}{3}|\mathcal{F}_{(2n-3)\tau}\right)\right)\\
\leq\tau(L,\Delta)\E(\theta_{2n-2}|\mathcal{F}_{(2n-3)\tau})+4(1-\tau(L,\Delta))\p\left(C\mathcal{V}_{0,2\tau}\geq\frac{L}{6}\right)\\
\mbox{(like in Equations (\ref{eq:maj-Vcal}), (\ref{eq:maj-Wcal}))}\\
\leq\tau(L,\Delta)\E(\theta_{2n-2}|\mathcal{F}_{(2n-3)\tau})+8(1-\tau(L,\Delta))\p\left(2C|W_{2\tau}|\geq\frac{L}{6}\right)\\
\mbox{(using Equation (\ref{eq:queue-gaussienne}))}\\
\leq\tau(L,\Delta)\E(\theta_{2n-2}|\mathcal{F}_{(2n-3)\tau})+(1-\tau(L,\Delta))\widetilde{\alpha}(L)\,.
\end{multline*}
The constant $\rho$ is the positive root of the polynomial $X^{2}-\tau(L,\Delta)X-(1-\tau(L,\Delta))\widetilde{\alpha}(L)$.
So we have
\[
1>\rho=\tau(L,\Delta)+\frac{1}{\rho}(1-\tau(L,\Delta))\widetilde{\alpha}(L)\geq\tau(L,\Delta)+(1-\tau(L,\Delta))\widetilde{\alpha}(L)\,.
\]
So, we have
\[
e_{2n|2k+2}\leq\tau(L,\Delta)e_{2n-2|2k+2}+(1-\tau(L,\Delta))\widetilde{\alpha}(L)e_{2n-4|2k+2}\leq\rho\times\sup(e_{2n-2|2k+2},e_{2n-4|2k+2})\,.
\]
Suppose now that $k$ is  fixed. We have 
\[
e_{2k+2|2k+2}\leq1\,,\,e_{2k+4|2k+2}\leq1\,.
\]
So, by recurrence, 
\[
e_{2n|2k+2}\leq\rho^{\left\lceil \frac{(n-k-2)_{+}}{2}\right\rceil }\,.
\]
As $\widetilde{\alpha}(L)\leq1/4$, we have
\[
\rho\leq\frac{1}{2}(\tau(L,\Delta)+\sqrt{\tau(L,\Delta)^{2}+1-\tau(L,\Delta)})\leq\frac{\tau(L,\Delta)+1}{2}=1-\frac{(\epsilon(L,\Delta)^{2}\epsilon'(L,\Delta))^{2}}{2}\,.
\]
\end{proof}
\begin{proof}[Proof of Lemma \ref{lem:synthese-des-resultats}]We
write the proof in the case where $n$ and $k$ are even. If $k$
was even and $n$ was odd, we would have to use the operators $S_{\dots}^{(p)U}$.
If $k$ was odd, the proof would be very similar but would require
to introduce new and heavy notations.

By Proposition \ref{prop:representation-avec-U}, Remark \ref{rem:sur-representation}
and Equation (\ref{eq:loi-de-U1}), we have, for all $\mu$ in $\mathcal{P}(\R)$
and all test function $\varphi$ in $\mathcal{C}_{b}^{+}(\R)$, 
\begin{multline*}
\left(\overline{R}_{n}^{\Delta}\overline{R}_{n-1}^{\Delta}\dots\overline{R}_{k+1}^{\Delta}(\mu)\right)(\varphi)\propto\int\varphi(u_{n+1}^{(1)})\left(\prod_{i=1}^{(n-k)/2}\psi_{k+2i-1}^{\Delta}(u_{k+2i-1})\right)\\
\times\left(\prod_{i=1}^{(n-k-2)/2}S_{n|k+2i+1}^{U}(u_{k+2i-1},du_{k+2i+1})\right)\\
\times S_{n|n+1}^{U}(u_{n-1},du_{n+1}^{(1)})\left(\int_{z'\in\R}S_{n|k}^{\Delta}((0,u_{k+1}^{(1)}),(du_{k+1}^{(2)},dz')\right)(\psi_{n|k}^{\Delta}(0,.)\bullet\mu)(du_{k+1}^{(1)})\,,
\end{multline*}
where we integrate over $u_{k+1}^{(1)}\in\R$, $u_{k+1},\,u_{k+3},\dots,\,u_{n-1}\in\R^{2}$,
$u_{n+1}^{(1)}\in\R$.%

By Proposition \ref{lem:U-mixing}, we know that $S_{n|k+2i+1}^{U}$
is $(\epsilon_{k+2i-1}^{2}(\epsilon'_{k+2i})^{2})$-mixing for all
$i$ in $\{1,2,\dots,1+(n-k)/2\}$. We now apply Proposition \ref{prop:representation-FK-sequence}
with the $S_{n|k+2i+1}^{U}$ playing the roles of the $\mathfrak{Q}_{\dots}$
and the $\psi_{k+2i-1}^{\Delta}$ playing the roles of the $\Psi_{\dots}$.
By Equations (\ref{eq:maj-erreur-locale-01}), (\ref{eq:maj-err-locale-02}),
we then have, for all $\mu$ and $\mu'$ in $\mathcal{P}(\R)$, 
\begin{multline*}
\left\Vert \overline{R}_{n}^{\Delta}\overline{R}_{n-1}^{\Delta}\dots\overline{R}_{k+1}^{\Delta}(\mu)-\overline{R}_{n}^{\Delta}\overline{R}_{n-1}^{\Delta}\dots\overline{R}_{k+1}^{\Delta}(\mu')\right\Vert \\
\leq\prod_{i=1}^{(n-k)/2}(1-\epsilon_{k+2i-1}^{2}(\epsilon'_{k+2i})^{2})\times2\inf\left(1,\frac{\Vert\psi_{n|k}^{\Delta}(0,.)\bullet\mu-\psi_{n|k}^{\Delta}(0,.)\bullet\mu'\Vert}{\epsilon_{k+1}^{2}(\epsilon'_{k+2})^{2}}\right)\,.
\end{multline*}
By Equations (\ref{eq:R-tilde-mixing}), (\ref{eq:maj-erreur-locale-01}),
(\ref{eq:maj-err-locale-02}), we have
\begin{eqnarray*}
\Vert\psi_{n|k}^{\Delta}(0,.)\bullet\mu-\psi_{n|k}^{\Delta}(0,.)\bullet\mu'\Vert & \leq & 2\inf\left(1,\frac{\Vert\psi_{n|k}^{\Delta}(0,.)\Vert_{\infty}}{\langle\mu,\psi_{n|k}^{\Delta}(0,.)\rangle}\Vert\mu-\mu'\Vert\right)\\
 & \leq & 2\inf\left(1,\frac{\Vert\mu-\mu'\Vert}{\epsilon_{k+1}^{2}}\right)\,.
\end{eqnarray*}
From which we get the result.\end{proof}

\subsection{Approximation of the optimal filter by the truncated filter}

We recall that ``$\preD$'' is defined in Definition \ref{def:preceq}.

\begin{proof}[Proof of Proposition \ref{prop:approx-par-filtre-robuste}]We
write the proof only for Equation (\ref{eq:approx-pi}), the proof
for Equation (\ref{eq:approx-pi-prime}) being very similar. We have
\begin{equation}
\Vert\pi_{n\tau}-\pi_{n\tau}^{\Delta}\Vert\leq\Vert\pi_{n\tau}-\overline{R}_{n}^{\Delta}(\pi_{(n-1)\tau})\Vert+\sum_{1\leq k\leq n-1}\Vert\overline{R}_{n:k+1}^{\Delta}(\pi_{k\tau})-\overline{R}_{n:k+1}^{\Delta}(\overline{R}_{k}^{\Delta}(\pi_{(k-1)\tau}))\Vert\,.\label{eq:somme-telescopique-01}
\end{equation}
Let us fix $k\in\{1,2,\dots,n-1\}$. %
{} From Lemma \ref{lem:synthese-des-resultats}, we get 
\begin{multline}
\E(\Vert\overline{R}_{n:k+1}^{\Delta}(\pi_{k\tau})-\overline{R}_{n:k+1}^{\Delta}(\overline{R}_{k}^{\Delta}(\pi_{(k-1)\tau}))\Vert)\\
\leq\E\left(\E\left(\left.\prod_{3\leq i\leq\lfloor\frac{n-k}{2}\rfloor}(1-(\epsilon_{k+2i-1}^{2}(\epsilon'_{k+2i})^{2}))\right|\mathcal{F}_{(k+2)\tau}\right)\times2\inf\left(1,\frac{\Vert\pi_{k\tau}-\overline{R}_{k}^{\Delta}(\pi_{(k-1)\tau})\Vert}{(\epsilon'_{k+2})^{2}\epsilon{}_{k+1}^{4}}\right)\right)\,,\label{eq:oublie-erreur-locale-01}
\end{multline}
with the convention that a product over indexes in the null set is
equal to one. From Lemma \ref{lem:contraction}, we get
\begin{multline}
\E(\Vert\overline{R}_{n:k+1}^{\Delta}(\pi_{k\tau})-\overline{R}_{n:k+1}^{\Delta}(\overline{R}_{k}^{\Delta}(\pi_{(k-1)\tau}))\Vert)\leq\left(1-\frac{(\epsilon(L,\Delta)^{2}\epsilon'(L,\Delta))^{2}}{2}\right)^{\left\lceil \frac{1}{2}\left(\lfloor\frac{n-k}{2}\rfloor-4\right)_{+}\right\rceil }\\
\times2\E\left(\inf\left(1,\frac{\Vert\pi_{k\tau}-\overline{R}_{k}^{\Delta}(\pi_{(k-1)\tau})\Vert}{(\epsilon'_{k+2})^{2}\epsilon^{4}{}_{k+1}}\right)\right)\,.\label{eq:oubli-erreur-locale-02}
\end{multline}
As in \cite{oudjane-rubenthaler-2005}, p. 434, we can bound
\begin{equation}
\inf\left(1,\frac{\Vert\pi_{n\tau}-\overline{R}_{k}^{\Delta}(\pi_{(k-1)\tau})\Vert}{(\epsilon'_{k+2})^{2}\epsilon_{k+1}^{4}}\right)\leq\inf\left(1,\frac{T(\Delta)}{(\epsilon'_{k+2})^{4}\epsilon{}_{k+1}^{8}}\right)+\inf\left(1,\frac{\Vert\pi_{n\tau}-\overline{R}_{k}^{\Delta}(\pi_{(k-1)\tau})\Vert^{2}}{T(\Delta)}\right)\,.\label{eq:inf-split}
\end{equation}

We have, if $\Delta$ satisfies the assumption of Proposition \ref{prop:erreur-locale},
\begin{multline}
\E\left(\inf\left(1,\frac{\Vert\pi_{k\tau}-\overline{R}_{k}^{\Delta}(\pi_{(k-1)\tau})\Vert^{2}}{T(\Delta)}\right)\right)\\
=\E\left(\frac{\Vert\pi_{k\tau}-\overline{R}_{k}^{\Delta}(\pi_{(k-1)\tau})\Vert^{2}}{T(\Delta)}\1_{[0,1]}\left(\frac{\Vert\pi_{k\tau}-\overline{R}_{k}^{\Delta}(\pi_{(k-1)\tau})\Vert^{2}}{T(\Delta)}\right)\right)\\
+\p\left(\frac{\Vert\pi_{k\tau}-\overline{R}_{k}^{\Delta}(\pi_{(k-1)\tau})\Vert^{2}}{T(\Delta)}>1\right)\\
\leq\frac{2}{\sqrt{T(\Delta)}}\E(\Vert\pi_{k\tau}-\overline{R}_{k}^{\Delta}(\pi_{(k-1)\tau})\Vert)\\
\mbox{(using Prop. \ref{prop:erreur-locale})}\preceq\sqrt{T(\Delta)}\,.\label{eq:inf-00}
\end{multline}

We look now at the term $\inf(1,T(\Delta)(\epsilon'_{k+2})^{-4}\epsilon_{k+1}^{-8})$.
Using Equations (\ref{eq:borne-theta}), (\ref{eq:def-d-Delta}),
(\ref{eq:def-T-Delta}), (\ref{eq:dl-sigma12}), (\ref{eq:dl-sigma22}),
(\ref{eq:lim-Cx}), (\ref{eq:lim-Cz}) and the remarks below Equation
(\ref{eq:borne-theta}), we have, for all $k$, 
\begin{multline}
T(\Delta)\underset{\Delta,c}{\preceq}\frac{h^{-1+\iota}\tau^{\iota-\frac{1}{2}}}{\Delta}\exp\left(-\frac{1}{2}\left(\frac{\tau^{\frac{1}{2}-\iota}h^{1-\iota}\Delta}{12C\sqrt{2}}\right)^{2}\right)\\
+\left(M+\sqrt{h}\right)\frac{he^{26M^{2}\tau+\frac{7\tau M}{2}+40M^{2}/h}}{\Delta\theta^{3/2}}\exp\left(-\frac{1}{4B_{2}}d(\Delta)^{2}\right)C_{1}'(h,\tau)+e^{-C_{0}\Delta^{2}}\,,\label{eq:maj-T-Delta}
\end{multline}
and using Equations (\ref{eq:def-xi_1}), (\ref{eq:def-xi_2}), (\ref{eq:def-epsilon}), 

\begin{multline}
\epsilon(D,\Delta)^{-1}=\exp\left(\frac{\left(D+\frac{\Delta}{B_{2}(1+p_{2,1})}\right)^{2}}{2\tau}-\frac{\left(\left(D-\frac{\Delta}{B_{2}(1+p_{2,1})}\right)_{+}\right)^{2}}{2\tau}\right)e^{2M\left(D+\frac{\Delta}{B_{2}(1+p_{2,1})}\right)+\left(\tau+\frac{\tau^{2}}{2}\right)M}\\
=\begin{cases}
\exp\left(\frac{2D\Delta}{\tau B_{2}(1+p_{2,1})}+2M\left(D+\frac{\Delta}{B_{2}(1+p_{2,1})}\right)+\left(\tau+\frac{\tau^{2}}{2}\right)M\right) & \mbox{ if }D\geq\frac{\Delta}{B_{2}(1+p_{2,1})}\,,\\
\exp\left(\frac{\left(D+\frac{\Delta}{B_{2}(1+p_{2,1})}\right)^{2}}{2\tau}+2M\left(D+\frac{\Delta}{B_{2}(1+p_{2,1})}\right)+\left(\tau+\frac{\tau^{2}}{2}\right)M\right) & \mbox{ otherwise .}
\end{cases}\label{eq:re-def-epsilon}
\end{multline}
and using Equation (\ref{eq:def-epsilon-prime}),
\begin{multline}
(\epsilon'(D,\Delta))^{-2}=\exp\left[\left(B_{2}p_{2,1}^{2}+\frac{1}{2\tau}\right)\left(\frac{\Delta}{B_{2}(1+p_{2,1})}+D\right)^{2}\right.\\
\left.+\left(2\Delta p_{2,1}+6M\right)\left(\frac{\Delta}{B_{2}(1+p_{2,1})}+D\right)+3\tau\left(M+\frac{M^{2}}{2}\right)\right]\,.\label{eq:re-def-epsilon-prime}
\end{multline}
We note that the above expressions are nondecreasing functions of
$D$. From Equation (\ref{eq:variation-m_k}), we get, for $j=k+1,\,k+2$.
\begin{equation}
D_{j}\leq C(\tau M+\mathcal{V}_{(k-1)\tau,(k+2)\tau}+2\mathcal{W}_{(k-1)\tau,(k+2)\tau})\,.\label{eq:maj-D_k}
\end{equation}
The variables $\mathcal{V}_{(k-1)\tau,(k+2)\tau}$ and $\mathcal{\mathcal{W}}_{(k-1)\tau,(k+2)\tau}$
are independent and can be controlled as in Equations (\ref{eq:maj-Vcal}),
(\ref{eq:maj-Wcal}). So we can bound
\[
\forall x\in\R\,,\,\p(\mathcal{V}_{(k-1)\tau,(k+2)\tau}+2\mathcal{W}_{(k-1)\tau,(k+2)\tau}\geq x)\leq2\p(2\mathcal{V}_{0,3\tau}\geq\frac{x}{2})\leq4\p(8|W_{3\tau}|\geq x)\,.
\]
So, by Lemma \ref{lem:comparaison-queues},
\begin{multline}
\E\left(\inf\left(1,\frac{T(\Delta)}{(\epsilon'_{k+2})^{4}\epsilon_{k+1}^{8}}\right)\right)\leq\int_{0}^{+\infty}\inf\left\{ 1,T(\Delta)\epsilon(C\tau M+Cz,\Delta)^{-8}\right.\\
\left.\times\epsilon'(C\tau M+Cz,\Delta)^{-4}\right\} \frac{8\exp\left(-\frac{z^{2}}{2(192\tau)}\right)}{\sqrt{2\pi\times192\tau}}dz\,.\label{eq:borne-T-01}
\end{multline}

We have %
{} 
\begin{multline}
\int_{0}^{\left(\frac{\Delta}{CB_{2}(1+p_{2,1})}-\tau M\right)_{+}}\epsilon(C\tau M+Cz,\Delta)^{-8}\epsilon'(C\tau M+Cz,\Delta)^{-4}\frac{8\exp\left(-\frac{z^{2}}{2(192\tau)}\right)}{\sqrt{2\pi\times192\tau}}dz\\
\leq8\epsilon\left(\frac{\Delta}{B_{2}(1+p_{2,1})},\Delta\right)^{-8}\epsilon'\left(\frac{\Delta}{B_{2}(1+p_{2,1})},\Delta\right)^{-4}\\
=8\exp\left[\frac{20}{\tau}\left(\frac{\Delta}{B_{2}(1+p_{2,1})}\right)^{2}+32\left(\frac{M\Delta}{B_{2}(1+p_{2,1})}\right)+8\left(\tau+\frac{\tau^{2}}{2}\right)M+8B_{2}p_{2,1}^{2}\left(\frac{\Delta}{B_{2}(1+p_{2,1})}\right)^{2}\right.\\
\left.+8\Delta p_{2,1}\left(\frac{\Delta}{B_{2}(1+p_{2,1})}\right)+24M\left(\frac{\Delta}{B_{2}(1+p_{2,1})}\right)+6\tau\left(M+\frac{M^{2}}{2}\right)\right]\,.\label{eq:borne-sur-inf-integrale}
\end{multline}
From Subsection \ref{subsec:Asymptotics} (remember also Equation
(\ref{eq:def-p})), we get 
\begin{equation}
p_{2,1}=O\left(\frac{1}{\theta}\right)\,,\,B_{2}\underset{\theta\rightarrow+\infty}{\longrightarrow}\frac{h}{2}\,.\label{eq:asymp-theta}
\end{equation}
So there exists $\tau_{0}$, such that, for $\tau\geq\tau_{0}$, 
\begin{multline}
\log\left(\int_{0}^{\left(\frac{\Delta}{CB_{2}(1+p_{2,1})}-\tau M\right)_{+}}\inf\{1,T(\Delta)\epsilon(C\tau M+Cz,\Delta)^{-8}\times\epsilon'(C\tau M+Cz,\Delta)^{-4}\}\right.\\
\left.\frac{8\exp\left(-\frac{z^{2}}{2(192\tau)}\right)}{\sqrt{2\pi\times192\tau}}dz\right)\preDc-\inf\left(\frac{1}{h},C_{0}\right)\Delta^{2}\,.\label{eq:maj-log-01}
\end{multline}

We now want to bound 
\[
\int_{\left(\frac{\Delta}{CB_{2}(1+p_{2,1})}-\tau M\right)_{+}}^{+\infty}\inf\{1,T(\Delta)\epsilon(C\tau M+Cz,\Delta)^{-8}\times\epsilon'(C\tau M+Cz,\Delta)^{-4}\}\frac{8\exp\left(-\frac{z^{2}}{2(192\tau)}\right)}{\sqrt{2\pi\times192\tau}}dz\,.
\]
Let us set, for $z\geq\left(\frac{\Delta}{CB_{2}(1+p_{2,1})}-\tau M\right)_{+}$,
\begin{multline*}
\Phi(\Delta,z)=T(\Delta)\epsilon(C\tau M+Cz,\Delta)^{-8}\times\epsilon'(C\tau M+Cz,\Delta)^{-4}\\
=T(\Delta)\times\exp\left(\frac{16(C\tau M+Cz)\Delta}{\tau B_{2}(1+p_{2,1})}+16M\left(C\tau M+Cz+\frac{\Delta}{B_{2}(1+p_{2,1})}\right)+8\left(\tau+\frac{\tau^{2}}{2}\right)M\right)\\
\times\exp\left[(2B_{2}p_{2,1}^{2}+\frac{1}{\tau})\left(\frac{\Delta}{B_{2}(1+p_{2,1})}+C\tau M+Cz\right)^{2}\right.\\
\left.+(4\Delta p_{2,1}+12M)\left(\frac{\Delta}{B_{2}(1+p_{2,1})}+C\tau M+Cz\right)+6\tau\left(M+\frac{M^{2}}{2}\right)\right]\,.
\end{multline*}
For $z\leq\left(\frac{\Delta}{CB_{2}(1+p_{2,1})}-\tau M\right)_{+}$,
we define $\Phi(\Delta,z)$ by 
\begin{multline*}
\Phi(\Delta,z)=\\
T(\Delta)\times\exp\left(\frac{16(C\tau M+Cz)\Delta}{\tau B_{2}(1+p_{2,1})}+16M\left(C\tau M+Cz+\frac{\Delta}{B_{2}(1+p_{2,1})}\right)+8\left(\tau+\frac{\tau^{2}}{2}\right)M\right)\\
\times\exp\left[(2B_{2}p_{2,1}^{2}+\frac{1}{\tau})\left(\frac{\Delta}{B_{2}(1+p_{2,1})}+C\tau M+Cz\right)^{2}\right.\\
\left.+(4\Delta p_{2,1}+12M)\left(\frac{\Delta}{B_{2}(1+p_{2,1})}+C\tau M+Cz\right)+6\tau\left(M+\frac{M^{2}}{2}\right)\right]\,.
\end{multline*}
For a constant $\overline{C}$ bigger than $C$, we define $\overline{\Phi}(\Delta,z)$
to have the same expression as $\Phi$, except that we replace $C$
by $\overline{C}$. We choose a 
\begin{equation}
\overline{C}=\sup\left(C,\frac{1}{\sqrt{2\times192}}\right)\,,\label{eq:dec-C-barre}
\end{equation}
 so that $z\mapsto\overline{\Phi}(\Delta,z)\times\exp(-z^{2}/(2\times192\tau))$
is nondecreasing in $z$. We have $\overline{\Phi}(\Delta,z)\rightarrow\infty$
when $z\rightarrow+\infty$ and $\Delta$ is fixed. Let us set, for
a fixed $\Delta$, 
\[
z_{0}=\inf\{z\,:\,\overline{\Phi}(\Delta,z)\geq1\}\,.
\]
There exists $\lambda_{1}>0$ and $\tau_{1}>0$ such that for all
$\tau\geq\tau_{1}$,
\[
T(\Delta)\epsilon(\overline{C}\tau M+\overline{C}\lambda_{1}\sqrt{\tau}\Delta,\Delta)^{-8}\epsilon'(\overline{C}\tau M+\overline{C}\lambda_{1}\sqrt{\tau}\Delta,\Delta)^{-4}\underset{\Delta\rightarrow+\infty}{\longrightarrow}0\,.
\]
Looking at the definition of $p_{2,1}$ (Equation (\ref{eq:def-p}))
and at Equations (\ref{eq:lim-Cxz}), (\ref{eq:lim-Cz}), we see that
$\lambda_{1}$ can be chosen as a function of $h$, which we denote
by $\lambda_{1}(h)$. And there exists $\lambda_{2}>0$ and $\tau_{2}>0$
such that for all $\tau\geq\tau_{2}$,
\[
T(\Delta)\epsilon(\overline{C}\tau M+\overline{C}\lambda_{2}\sqrt{\tau}\Delta,\Delta)^{-8}\epsilon'(\overline{C}\tau M+\overline{C}\lambda_{2}\sqrt{\tau}\Delta,\Delta)^{-4}\underset{\Delta\rightarrow+\infty}{\longrightarrow}+\infty\,.
\]
So there exists $\Delta_{\text{1 }}$ such that, for $\Delta$ bigger
than $\Delta_{1}$ and $\tau$ bigger than $\sup(\tau_{1},\tau_{2})$,
\[
\lambda_{1}\sqrt{\tau}\Delta\leq z_{0}\leq\lambda_{2}\sqrt{\tau}\Delta\,.
\]
 We can then bound, if $\tau\geq\sup(\tau_{1},\tau_{2})$ ,
\begin{eqnarray*}
\int_{0}^{z_{0}}\inf(1,\Phi(\Delta,z))\frac{8\exp\left(-\frac{z^{2}}{2(192\tau)}\right)}{\sqrt{2\pi(192\tau)}}dz & \leq & \int_{0}^{z_{0}}\inf(1,\overline{\Phi}(\Delta,z))\frac{8\exp\left(-\frac{z^{2}}{2(192\tau)}\right)}{\sqrt{2\pi(192\tau)}}dz\\
 & \leq & z_{0}\overline{\Phi}(\Delta,z_{0})\frac{8\exp\left(-\frac{z_{0}^{2}}{2(192\tau)}\right)}{\sqrt{2\pi(192\tau)}}\\
 & = & z_{0}\frac{8\exp\left(-\frac{z_{0}^{2}}{2(192\tau)}\right)}{\sqrt{2\pi(192\tau)}}\\
 & \leq & \lambda_{2}\sqrt{\tau}\Delta\frac{8\exp\left(-\frac{\lambda_{1}^{2}\Delta^{2}}{2\times192}\right)}{\sqrt{2\pi(192\tau)}}\,.
\end{eqnarray*}

{} So, if $\tau\geq\sup(\tau_{0},\tau_{1},\tau_{2}),$ we get, using
again Equation (\ref{eq:queue-gaussienne})

\begin{multline}
\int_{\left(\frac{\Delta}{CB_{2}(1+p_{2,1})}-\tau^{2}M\right)_{+}}^{+\infty}\inf\{1,T(\Delta)\epsilon(C\tau^{2}M+Cz,\Delta)^{-8}\times\epsilon'(C\tau^{2}M+Cz,\Delta)^{-4}\}\\
\times\frac{8\exp\left(-\frac{z^{2}}{2(192\tau)}\right)}{\sqrt{2\pi\times192\tau}}dz\\
\preceq\int_{0}^{z_{0}}\Phi(\Delta,z)\frac{e^{-\frac{z^{2}}{2(192\tau)}}}{\sqrt{2\pi\times192\tau}}dz+\int_{z_{0}}^{+\infty}\frac{e^{-\frac{z^{2}}{2(192\tau)}}}{\sqrt{2\pi\times192\tau}}dz\\
\leq\lambda_{2}\sqrt{\tau}\Delta\frac{8\exp\left(-\frac{\lambda_{1}^{2}\Delta^{2}}{384}\right)}{\sqrt{2\pi(192\tau)}}+\frac{e^{-\frac{z_{0}^{2}}{2(192\tau)}}}{z_{0}\sqrt{2\pi}}\times\sqrt{192\tau},\label{eq:borne-log-02}
\end{multline}
and so
\begin{multline}
\log\left(\int_{-\infty}^{+\infty}\inf(1,T(\Delta)\epsilon(C\tau M+Cz,\Delta)^{-8}\times\epsilon'(C\tau M+Cz,\Delta)^{-4})\frac{e^{-\frac{z^{2}}{2(192\tau)}}}{\sqrt{2\pi\times192\tau}}dz\right)\\
\preDc-\Delta^{2}\inf\left(\frac{1}{h},C_{0},\lambda_{1}^{2}\right)\,.\label{eq:borne-log-03}
\end{multline}

In the remaining of the proof, we will suppose $\tau\geq\sup(\tau_{0},\tau_{1},\tau_{2})$.
Looking at Equation (\ref{eq:somme-telescopique-01}), we see that
we can now bound all the terms on its right-hand side. We have
\[
\E(\Vert\pi_{n\tau}-\overline{R}_{n}^{\Delta}(\pi_{(n-1)\tau})\Vert)\preceq T(\Delta)\,,
\]
by Proposition \ref{prop:erreur-locale}. Recall, that, from Equations
(\ref{eq:maj-T-Delta}), (\ref{eq:def-d-Delta}), we get
\begin{equation}
\log(T(\Delta))\preDc-\Delta^{2}\inf\left(\frac{1}{h},C_{0}\right)\,.\label{eq:maj-log-T}
\end{equation}
 For $k$ in $\{1,\dots,n-1\}$, we have bounded
\begin{multline*}
\E(\Vert\overline{R}_{n:k+1}^{\Delta}(\pi_{k\Delta})-\overline{R}_{n:k+1}^{\Delta}(\overline{R}_{k}^{\Delta}(\pi_{(k-1)\Delta})\Vert)\\
\leq\left(1-\frac{\epsilon(L,\Delta)^{2}\epsilon'(L,\Delta)^{2}}{2}\right)^{\left\lceil \frac{1}{2}\left(\lfloor\frac{n-k}{2}\rfloor-4\right)_{+}\right\rceil }\times2\E\left(\inf\left(1,\frac{\Vert\pi_{k\tau}-\overline{R}_{k}^{\Delta}(\pi_{(k-1)\tau})\Vert}{(\epsilon'_{k+2})^{2}\epsilon_{k+1}^{4}}\right)\right)\,.
\end{multline*}
And the last expectation can be bounded by the sum of the following
expectations~: 
\[
\E\left(\inf\left(1,\frac{\Vert\pi_{k\tau}-\overline{R}_{k}^{\Delta}(\pi_{(k-1)\tau})\Vert^{2}}{T(\Delta)}\right)\right)\preceq\sqrt{T(\Delta)}\,,
\]
\[
\E\left(\inf\left(1,\frac{T(\Delta)}{(\epsilon'_{k+2})^{4}\epsilon_{k+1}^{8}}\right)\right)\leq\exp\left(-\widehat{B}_{1}\Delta^{2}\inf\left(\frac{1}{h},C_{0},\lambda_{1}^{2}\right)\right)\mbox{ for }\Delta\geq\Delta_{0}(\tau)\,,
\]
for some constant $\widehat{B}_{1}$ and some function $\Delta_{0}$,
where the bounds come from Equations (\ref{eq:inf-00}), (\ref{eq:borne-T-01}),%
{} (\ref{eq:borne-log-03}) (we also use Lemma \ref{lem:proprietes-preceq}).
The constant $\widehat{B}_{1}$ above is universal and $\Delta_{0}$
is continuous in $\tau$. So we get, for all $\Delta\geq\Delta_{0}(\tau)$,
using Equation (\ref{eq:maj-log-T}), 
\[
\E(\Vert\pi_{n\tau}-\pi_{n\tau}^{\Delta}\Vert)\leq\exp\left(-\widehat{C}_{1}\Delta^{2}\inf\left(\frac{1}{h},C_{0},\lambda_{1}^{2}\right)\right)\sum_{k\geq0}\left(1-\frac{\epsilon(L,\Delta)^{4}\epsilon'(L,\Delta)^{2}}{2}\right)^{\left\lceil \frac{1}{2}\left(\lfloor\frac{n-k}{2}\rfloor-4\right)_{+}\right\rceil }\,,
\]
(for some universal constant $\widehat{C}_{1}$) from which we get
\begin{eqnarray*}
\sup_{n\geq0}\log\E(\Vert\pi_{n\tau}-\pi_{n\tau}^{\Delta}\Vert) & \preDc & -\log(\epsilon(L,\Delta)\epsilon'(L,\Delta))-\Delta^{2}\inf\left(\frac{1}{h},C_{0},\lambda_{1}^{2}\right)\,.
\end{eqnarray*}
 Looking at Equations (\ref{eq:re-def-epsilon}), (\ref{eq:re-def-epsilon-prime}),
we see there exists $\tau_{3}$, such that, for $\tau>\tau_{3}$,
\[
\sup_{n\geq0}\log\E(\Vert\pi_{n\tau}-\pi_{n\tau}^{\Delta}\Vert)\preDc-\Delta^{2}\inf\left(\frac{1}{h},C_{0},\lambda_{1}^{2}\right)\,.
\]
\end{proof}

\subsection{Stability of the optimal filter}

\begin{proof}[Proof of Theorem \ref{thm:stabilite}] We decompose,
for all $n$,
\begin{multline*}
\Vert\pi_{n\tau}-\pi'_{n\tau}\Vert\leq\Vert\pi_{n\tau}-\overline{R}_{n}^{\Delta}\dots\overline{R}_{1}^{\Delta}(\pi_{0})\Vert+\Vert\overline{R}_{n}^{\Delta}\dots\overline{R}_{1}^{\Delta}(\pi_{0})-\overline{R}_{n}^{\Delta}\dots\overline{R}_{1}^{\Delta}(\pi_{0}')\Vert\\
+\Vert\overline{R}_{n}^{\Delta}\dots\overline{R}_{1}^{\Delta}(\pi_{0})-\pi'_{n\tau}\Vert\,.
\end{multline*}
Let $\tau_{\infty}$ be the parameter defined in Proposition \ref{prop:approx-par-filtre-robuste}.
Recall that the operators $(R_{n})_{n\geq0}$, $(R_{n}^{\Delta})_{n\geq0}$
depend on $\tau$. %
Suppose that $L$ is such that (as in Equation (\ref{eq:cond-L}))
\[
L>3|m_{0}|+3CM\tau_{\infty}\,,\,\widetilde{\alpha}(L)\leq\frac{1}{4}\,.
\]
Then, as in Equation (\ref{eq:oubli-erreur-locale-02}), we have,
for all $\tau\in[\tau_{\infty},2\tau_{\infty}]$, for all $n\geq0$,
\begin{multline*}
\E(\Vert\overline{R}_{n}^{\Delta}\dots\overline{R}_{1}^{\Delta}(\pi_{0})-\overline{R}_{n}^{\Delta}\dots\overline{R}_{1}^{\Delta}(\pi_{0}')\Vert)\\
\leq\left(1-\frac{\epsilon(L,\Delta)^{2}\epsilon'(L,\Delta)^{2}}{2}\right)^{\left\lceil \frac{1}{2}\left(\lfloor\frac{n}{2}\rfloor-4\right)_{+}\right\rceil }\times2\E\left(\inf\left(1,\frac{\Vert\pi_{0}-\pi_{0}'\Vert}{(\epsilon'_{2})^{2}\epsilon_{1}^{4}}\right)\right)\\
\leq2\left(1-\frac{\epsilon(L,\Delta)^{2}\epsilon'(L,\Delta)^{2}}{2}\right)^{\left\lceil \frac{1}{2}\left(\lfloor\frac{n}{2}\rfloor-4\right)_{+}\right\rceil }\,.
\end{multline*}
We have (using Equations (\ref{eq:re-def-epsilon}), (\ref{eq:re-def-epsilon-prime}))
\[
\log(\epsilon(L,\Delta)\epsilon'(L,\Delta))\underset{\Delta,c}{\succeq}-\left(\text{\ensuremath{\frac{\Delta^{2}}{\tau B_{2}^{2}(1+p_{2,1})^{2}}}+\ensuremath{\frac{\left(B_{2}p_{2,1}^{2}+\frac{1}{\tau}\right)\Delta^{2}}{B_{2}^{2}(1+p_{2,1})^{2}}}+\ensuremath{\frac{p_{2,1}\Delta^{2}}{B_{2}(1+p_{2,1})}}}\right)\,.
\]
We now take a sequence $\Delta_{n}=\sqrt{\nu\log(n)}$, for some $\nu>0$.
By Lemma \ref{lem:proprietes-preceq}, \ref{enu:lemme-technique-iii},
there exist a constants $b_{1}$ and an integer $n_{0}$ such that,
for all $\tau\in[\tau_{\infty},2\tau_{\infty}]$, for $n>n_{0}$,
\begin{multline}
\left(1-\frac{(\epsilon(L,\Delta_{n})\epsilon'(L,\Delta_{n}))^{2}}{2}\right)^{\left\lceil \frac{1}{2}\left(\lfloor\frac{n}{2}\rfloor-4\right)_{+}\right\rceil }\\
\leq\exp\left[-\frac{1}{2}\left\lceil \frac{1}{2}\left(\lfloor\frac{n}{2}\rfloor-4\right)_{+}\right\rceil \right.\\
\left.\times\exp\left(-b_{1}\Delta_{n}^{2}\left(\frac{1}{\tau B_{2}^{2}(1+p_{2,1})^{2}}+\frac{B_{2}p_{2,1}^{2}+\frac{1}{\tau}}{B_{2}^{2}(1+p_{2,1})^{2}}+\frac{p_{2,1}}{B_{2}(1+p_{2,1})}\right)\right)\right]\\
=\exp\left[-\frac{1}{2}\left\lceil \frac{1}{2}\left(\lfloor\frac{n}{2}\rfloor-4\right)_{+}\right\rceil n^{-\nu'}\right]\,,\label{eq:terme-01}
\end{multline}
with 
\[
\nu'=b_{1}\nu\left(\frac{1}{\tau B_{2}^{2}(1+p_{2,1})^{2}}+\frac{B_{2}p_{2,1}^{2}+\frac{1}{\tau}}{B_{2}^{2}(1+p_{2,1})^{2}}+\frac{p_{2,1}}{B_{2}(1+p_{2,1})}\right)\,.
\]
 By Proposition \ref{prop:approx-par-filtre-robuste}, we know there
exists a constants $b_{1}'$ and a integer $n_{0}'$ such that, for
all $\tau\in[\tau_{\infty},2\tau_{\infty}]$ and $n\geq n_{0}'$,
\begin{multline*}
\sup(\E(\Vert\pi_{n\tau}-\overline{R}_{n}^{\Delta_{n}}\dots\overline{R}_{1}^{\Delta_{n}}(\pi_{0})\Vert),\E(\Vert\pi_{n\tau}'-\overline{R}_{n}^{\Delta_{n}}\dots\overline{R}_{1}^{\Delta_{n}}(\pi_{0}')\Vert))\\
\leq\exp\left(-b_{1}'\Delta_{n}^{2}\lambda_{1}'(h)\right)\leq n^{-\nu''}\,,
\end{multline*}
with $\nu''=b'_{1}\nu\lambda_{1}'(h)$. Let us set $\epsilon\in(0,1)$.
We choose 
\[
\nu=\frac{(1-\epsilon)}{b_{1}}\left(\frac{1}{\tau B_{2}^{2}(1+p_{2,1})^{2}}+\frac{B_{2}p_{2,1}^{2}+\frac{1}{\tau}}{B_{2}^{2}(1+p_{2,1})^{2}}+\frac{p_{2,1}}{B_{2}(1+p_{2,1})}\right)^{-1}\,,
\]
which leads to $\nu'=1-\epsilon$. We set $\nu_{0}=\nu''$. For any
$t\geq\tau_{\infty}$, if we set $n=\lfloor t/\tau_{\infty}\rfloor$,
then $t=n\tau$ with $\tau\in[\tau_{\infty},2\tau_{\infty}]$, and
so~:
\[
\E(\Vert\pi_{t}-\pi'_{t}\Vert)\leq2n^{-\nu_{0}}+2\exp\left(-\frac{1}{2}\left\lceil \frac{1}{2}\left(\lfloor\frac{n-k}{2}\rfloor-4\right)_{+}\right\rceil n^{-\nu'}\right)\,,
\]
\[
\E(\Vert\pi_{t}-\pi'_{t}\Vert)=O(t^{-\nu_{0}})\,.
\]
\end{proof} 
\begin{rem}
One could seek to obtain a sharper bound in the above Theorem by choosing
another sequence $(\Delta_{n})_{n\geq0}$. Up to some logarithmic
terms, the bound would still be a power of $t$. 
\end{rem}

\section{Appendix}

\subsection{\label{subsec:Proofs-of-Section-1}Proofs of Section \ref{sec:Introduction}}

\begin{proof}[Proof of Lemma \ref{lem:encadrement-transition}] Following~\cite{bain-crisan-2009}
(Chapter 6, Section 6.1), we introduce the process 
\[
\widehat{V}_{t}=V_{t}+\int_{0}^{t}f(X_{s})ds\,,\,t\geq0\,.
\]
We introduce a new probability $\widetilde{\p}$ defined by
\begin{equation}
\left.\frac{d\p}{d\widetilde{\p}}\right|_{\mathcal{F}_{t}}=\exp\left(\int_{0}^{t}f(X_{s})d\widehat{V}_{s}-\frac{1}{2}\int_{0}^{t}f(X_{s})^{2}ds\right)\,.\label{eq:def-Ptilde}
\end{equation}
By Girsanov's theorem, $\widehat{V}$ is a standard Brownian motion
under $\widetilde{\p}$. We set $F$ to be a primitive of $f$. We
have, for all $t\geq0$, 
\begin{eqnarray*}
\int_{0}^{t}f(X_{s})d\widehat{V}_{s}-\frac{1}{2}\int_{0}^{t}f(X_{s})^{2}ds & = & \int_{0}^{t}f(X_{s})dX_{s}-\frac{1}{2}\int_{0}^{t}f(X_{s})^{2}ds\\
 & = & F(X_{t})-F(X_{0})-\frac{1}{2}\int_{0}^{t}f'(X_{s})ds-\frac{1}{2}\int_{0}^{t}f(X_{s})^{2}ds\\
 & \geq & -M|X_{t}-X_{0}|-\frac{Mt}{2}-\frac{M^{2}t}{2}\,.
\end{eqnarray*}
So, for any test function $\varphi$ in $\mathcal{C}_{b}^{+}(\R)$
(the set of bounded continuous functions on $\R$), $t\geq0$ 
\begin{eqnarray*}
\E(\varphi(X_{t})) & = & \E^{\p}(\varphi(X_{t}))\\
 & = & \E^{\widetilde{\p}}\left(\varphi(X_{t})\left.\frac{d\p}{d\widetilde{\p}}\right|_{\mathcal{F}_{t}}\right)\\
 & \geq & \E^{\widetilde{\p}}\left(\varphi(X_{t})\exp\left(-M|X_{t}-X_{0}|-\frac{Mt}{2}-\frac{M^{2}t}{2}\right)\right)\,.
\end{eqnarray*}
Similarly:
\[
\E(\varphi(X_{t}))\leq\E^{\widetilde{\p}}\left(\varphi(X_{t})\exp\left(M|X_{t}-X_{0}|+\frac{Mt}{2}\right)\right)
\]
So we have the result.\end{proof}

\begin{proof}[Proof of Lemma \ref{lem:encadrement-potentiel}]For
any test function $\varphi$ in $\mathcal{C}_{b}^{+}([0,t])$ and
any $t\geq0$, 
\begin{eqnarray*}
\E^{\p}(\varphi(Y_{0:t})|X_{0},X_{t}) & = & \frac{\E^{\widehat{\p}}\left(\left.\varphi(Y_{0:t})\left.\frac{d\p}{d\widehat{\p}}\right|_{\mathcal{F}_{t}}\right|X_{0},X_{t}\right)}{\E^{\widehat{\p}}\left(\left.\left.\frac{d\p}{d\widehat{\p}}\right|_{\mathcal{F}_{t}}\right|X_{0},X_{t}\right)}\\
 & = & \frac{\E^{\widehat{\p}}\left(\left.\varphi(Y_{0:t})\E^{\widehat{\p}}\left(\left.\left.\frac{d\p}{d\widehat{\p}}\right|_{\mathcal{F}_{t}}\right|X_{0},X_{t},Y_{0:t}\right)\right|X_{0},X_{t}\right)}{\E^{\widehat{\p}}\left(\left.\left.\frac{d\p}{d\widehat{\p}}\right|_{\mathcal{F}_{t}}\right|X_{0},X_{t}\right)}\,.
\end{eqnarray*}
By Girsanov's Theorem, $(\widehat{V},Y)$ is a standard two-dimensional
Brownian motion under $\widehat{\p}$. So, conditionally on $X_{0}$,
$X_{t}$, the law of $Y_{0:t}$ under $\p$ has the following density
with respect to the Wiener measure:
\[
y_{0:t}\mapsto\psi_{t}(Y_{0:t},X_{0},X_{t})=\frac{\E^{\widehat{\p}}\left(\left.\left.\frac{d\p}{d\widehat{\p}}\right|_{\mathcal{F}_{t}}\right|X_{0},X_{t},Y_{0:t}\right)}{\E^{\widehat{\p}}\left(\left.\left.\frac{d\p}{d\widehat{\p}}\right|_{\mathcal{F}_{t}}\right|X_{0},X_{t}\right)}
\]
 We have
\begin{multline}
\left.\frac{d\p}{d\widehat{\p}}\right|_{\mathcal{F}_{t}}=\exp\left(F(X_{1})-F(X_{0})-\frac{1}{2}\int_{0}^{t}f'(X_{s})ds-\frac{1}{2}\int_{0}^{1}f(X_{s})^{2}ds\right.\\
\left.+\int_{0}^{t}hX_{s}dY_{s}-\frac{1}{2}\int_{0}^{t}h^{2}X_{s}^{2}ds\right)\\
\E^{\widehat{\p}}\left(\left.\left.\frac{d\p}{d\widehat{\p}}\right|_{\mathcal{F}_{t}}\right|X_{0},X_{t}\right)=\E^{\widehat{\p}}\left(\left.\exp\left(F(X_{1})-F(X_{0})-\frac{1}{2}\int_{0}^{t}f'(X_{s})ds-\frac{1}{2}\int_{0}^{t}f(X_{s})^{2}ds\right)\right|X_{0},X_{1}\right)\,,\label{eq:dP/dPchapeau}
\end{multline}
so
\[
\exp\left(-M|X_{t}-X_{0}|-\frac{t(M+M^{2})}{2}\right)\leq\E^{\widehat{\p}}\left(\left.\left.\frac{d\p}{d\widehat{\p}}\right|_{\mathcal{F}_{t}}\right|X_{0},X_{t}\right)\leq\exp\left(M|X_{t}-X_{0}|+\frac{tM}{2}\right)\,.
\]
Using the above calculations (Equation (\ref{eq:dP/dPchapeau})),
we can write:
\begin{multline*}
\exp\left(-M|X_{t}-X_{0}|-\frac{t(M+M^{2})}{2}\right)\widehat{\psi}_{t}(Y_{0:t},x_{0},x_{1})\leq\\
\E^{\widehat{\p}}\left(\left.\left.\frac{d\p}{d\widehat{\p}}\right|_{\mathcal{F}_{t}}\right|X_{0},X_{t},Y_{0:t}\right)\leq\exp\left(M|X_{t}-X_{0}|+\frac{tM}{2}\right)\widehat{\psi}(Y_{0:t},x_{0},x_{1})\,.
\end{multline*}
So we have the result.\end{proof}

\begin{proof}[Proof of Lemma \ref{lem:kallianpur-striebel}]We define
a new probability $\check{\p}$ by 
\[
\left.\frac{d\p}{d\check{\p}}\right|_{\mathcal{F}_{t}}=\exp\left(\int_{0}^{t}hX_{s}dY_{s}-\frac{1}{2}\int_{0}^{t}h^{2}X_{s}^{2}ds\right)\,,\,\forall t\geq0\,.
\]
By Girsanov's Theorem, $(Y_{t})$ is a Brownian motion under $\check{\p}$.
For all bounded continuous $\varphi$ and all $t\geq0$, we have (Kallianpur-Striebel,
see \cite{bain-crisan-2009}, p.57)
\[
\E(\varphi(X_{t})|Y_{0:t})=\frac{\E^{\check{\p}}\left(\varphi(X_{t})\left.\frac{d\p}{d\check{\p}}\right|_{\mathcal{F}_{t}}|Y_{0:t}\right)}{\E^{\check{\p}}\left(\left.\frac{d\p}{d\check{\p}}\right||Y_{0:t}\right)}\,,
\]
and
\[
\E^{\check{\p}}\left(\varphi(X_{t})\left.\frac{d\p}{d\check{\p}}\right|_{\mathcal{F}_{t}}|Y_{0:t}\right)=\E^{\check{\p}}\left(\varphi(X_{t})\E^{\check{\p}}\left(\left.\frac{d\p}{d\check{\p}}\right|_{\mathcal{F}_{t}}|Y_{0:t},X_{0},X_{t}\right)|Y_{0:t}\right)\,,
\]
and
\begin{eqnarray*}
\E^{\check{\p}}\left(\left.\frac{d\p}{d\check{\p}}\right|_{\mathcal{F}_{t}}|Y_{0:t},X_{0},X_{t}\right) & = & \frac{\E^{\widehat{\p}}\left(\left.\frac{d\p}{d\check{\p}}\right|_{\mathcal{F}_{t}}\left.\frac{d\check{\p}}{d\widehat{\p}}\right|_{\mathcal{F}_{t}}|Y_{0:t},X_{0},X_{t}\right)}{\E^{\widehat{\p}}\left(\left.\frac{d\check{\p}}{d\widehat{\p}}\right|_{\mathcal{F}_{t}}|Y_{0:t},X_{0},X_{t}\right)}\\
 & = & \frac{\psi_{t}(Y_{0:t},X_{0},X_{t})\E^{\widehat{\p}}\left(\left.\frac{d\p}{d\widehat{\p}}\right|_{\mathcal{F}_{t}}|X_{0},X_{t}\right)}{\E^{\widehat{\p}}\left(\left.\frac{d\p}{d\widetilde{\p}}\right|_{\mathcal{F}_{t}}|Y_{0:t},X_{0},X_{t}\right)}\\
 & = & \psi_{t}(Y_{0:t},X_{0},X_{t})\times\frac{\E^{\widehat{\p}}\left(\left.\frac{d\p}{d\widetilde{\p}}\right|_{\mathcal{F}_{t}}|X_{0},X_{t}\right)}{\E^{\widehat{\p}}\left(\left.\frac{d\p}{d\widetilde{\p}}\right|_{\mathcal{F}_{t}}|X_{0},X_{t}\right)}\\
 & = & \psi_{t}(Y_{0:t},X_{0},X_{t})\,.
\end{eqnarray*}
As the law of $(X_{s})_{s\geq0}$ is the same under $\p$ or $\check{\p}$,
we get the desired result.\end{proof}

\subsection{Proofs of Section \ref{sec:Computation-of}\label{subsec:Proofs-of-Section-psi}}

We first prove two technical Lemmas.
\begin{lem}
\label{lem:representation-O-U}For any $t>0$, for any function $g:\R\mapsto\R$
which is measurable with respect to the Lebesgue measure and such
that $\int_{0}^{t}g(s)^{2}ds<\infty$, we have~:
\[
\int_{0}^{t}g(s)dB_{s}=\int_{0}^{t}\left(g(s)-\theta e^{\theta s}\int_{s}^{t}e^{-\theta u}g(u)du\right)d\beta_{s}\,.
\]
\end{lem}

\begin{proof}
Under $\q$, $B$ is an Ornstein-Uhlenbeck process (see Equation (\ref{eq:O-U})).
We can write $B$ as the strong solution of (\ref{eq:O-U}):
\begin{equation}
B_{t}=e^{-\theta t}\int_{0}^{t}e^{\theta s}d\beta_{s}\,,\,\forall t\geq0\,.\label{eq:representation-O-U}
\end{equation}
 We use the integration by parts  formula to compute:
\begin{multline*}
\int_{0}^{t}\left(g(s)-\theta e^{\theta s}\int_{s}^{t}e^{-\theta u}g(u)du\right)d\beta_{s}\\
=\int_{0}^{t}\left(g(s)-\theta e^{\theta s}\int_{0}^{t}e^{-\theta u}g(u)du\right)d\beta_{s}+\int_{0}^{t}\left(\theta e^{\theta s}\int_{0}^{s}e^{-\theta u}g(u)du\right)d\beta_{s}\\
=\int_{0}^{t}g(s)d\beta_{s}-\left(\int_{0}^{t}e^{-\theta u}g(u)du\right)\left(\int_{0}^{t}\theta e^{\theta s}d\beta_{s}\right)+\int_{0}^{t}\left(\theta e^{\theta s}\int_{0}^{s}e^{-\theta u}g(u)du\right)d\beta_{s}\\
=\int_{0}^{t}g(s)d\beta_{s}-\int_{0}^{t}e^{-\theta u}g(u)\left(\int_{0}^{u}\theta e^{\theta s}d\beta_{s}\right)du-\int_{0}^{t}\left(\int_{0}^{s}e^{-\theta u}g(u)du\right)\theta e^{\theta s}d\beta_{s}\\
+\int_{0}^{t}\left(\theta e^{\theta s}\int_{0}^{s}e^{-\theta u}g(u)du\right)d\beta_{s}\\
=\int_{0}^{t}g(s)d\beta_{s}-\int_{0}^{t}\theta g(u)B_{u}du=\int_{0}^{t}g(s)dB_{s}\,.
\end{multline*}
\end{proof}
\begin{lem}
\label{lem:integral-O-U}We have, for all~$s,t\geq0$,

{} 
\[
g(s)-\theta e^{\theta s}\int_{s}^{t}e^{-\theta u}g(u)du=\begin{cases}
e^{\theta(s-t)} & \mbox{ if }g(u)=1,\,\forall u,\\
\left(t+\frac{1}{\theta}\right)e^{\theta(s-t)}-\frac{1}{\theta} & \mbox{ if }g(u)=u,\,\forall u,\\
\left(t^{2}+\frac{2t}{\theta}+\frac{2}{\theta^{2}}\right)e^{\theta(s-t)}-\left(\frac{2s}{\theta}+\frac{2}{\theta^{2}}\right) & \mbox{ if }g(u)=u^{2},\,\forall u\,.
\end{cases}
\]
\end{lem}

\begin{proof}
The proof in the case $g(u)=1$ is straightforward. We compute, for
all $s,t\geq0$:
\begin{eqnarray*}
s-\theta e^{\theta s}\int_{s}^{t}ue^{-\theta u}du & = & s-\theta e^{\theta s}\left[\left(-\frac{u}{\theta}-\frac{1}{\theta^{2}}\right)e^{-\theta u}\right]_{s}^{t}\\
 & = & s-\theta\left(-\frac{t}{\theta}-\frac{1}{\theta^{2}}\right)e^{\theta(s-t)}+\theta\left(-\frac{s}{\theta}-\frac{1}{\theta^{2}}\right)\\
 & = & \left(t+\frac{1}{\theta}\right)e^{\theta(s-t)}-\frac{1}{\theta}\,,
\end{eqnarray*}
\begin{eqnarray*}
s^{2}-\theta e^{\theta s}\int_{s}^{t}u^{2}e^{-\theta u}du & = & s^{2}-\theta e^{\theta s}\left[\left(-\frac{u^{2}}{\theta}-\frac{2u}{\theta^{2}}-\frac{2}{\theta^{3}}\right)e^{-\theta u}\right]_{s}^{t}\\
 & = & s^{2}-\theta e^{\theta s}\left(-\frac{t^{2}}{\theta}-\frac{2t}{\theta^{2}}-\frac{2}{\theta^{3}}\right)e^{-\theta t}+\theta e^{\theta s}\left(-\frac{s^{2}}{\theta}-\frac{2s}{\theta^{2}}-\frac{2}{\theta^{3}}\right)e^{-\theta s}\\
 & = & \left(t^{2}+\frac{2t}{\theta}+\frac{2}{\theta^{2}}\right)e^{\theta(s-t)}-\left(\frac{2s}{\theta}+\frac{2}{\theta^{2}}\right)\,.
\end{eqnarray*}
\end{proof}
\begin{proof}[Proof of Lemma \ref{lem:calcul-var}] Lemma \ref{lem:representation-O-U}
tells us that the variables $G_{1}$, $G_{2}$, $G_{3}$, $G_{4}$
are centered Gaussians under $\mathbb{Q}$.~Using Lemma \ref{lem:integral-O-U},~we
compute the following expectations:
\begin{eqnarray*}
\E^{\q}(G_{1}^{2}) & = & \int_{0}^{1}e^{2\theta(s-1)}ds\\
 & = & \frac{1-e^{-2\theta}}{2\theta}\,,
\end{eqnarray*}
\begin{eqnarray*}
\E^{\q}(G_{2}^{2}) & = & \int_{0}^{1}\left(\left(1+\frac{1}{\theta}\right)e^{\theta(s-1)}-\frac{1}{\theta}\right)^{2}ds\\
 & = & \int_{0}^{1}\left(1+\frac{1}{\theta}\right)^{2}e^{2\theta(s-1)}+\frac{1}{\theta^{2}}-\frac{2}{\theta}\left(1+\frac{1}{\theta}\right)e^{\theta(s-1)}ds\\
 & = & \left[\left(1+\frac{1}{\theta}\right)^{2}\frac{e^{2\theta(s-1)}}{2\theta}+\frac{s}{\theta^{2}}-\left(1+\frac{1}{\theta}\right)\frac{2e^{\theta(s-1)}}{\theta^{2}}\right]_{0}^{1}\\
 & = & \left(1+\frac{1}{\theta}\right)^{2}\frac{(1-e^{-2\theta})}{2\theta}+\frac{1}{\theta^{2}}-\left(\frac{2}{\theta^{2}}+\frac{2}{\theta^{3}}\right)(1-e^{-\theta})\,,
\end{eqnarray*}
\begin{multline*}
\E^{\q}(G_{3}^{2})=\int_{0}^{1}\left(\left(1+\frac{2}{\theta}+\frac{2}{\theta^{2}}\right)e^{\theta(s-1)}-\left(\frac{2s}{\theta}+\frac{2}{\theta^{2}}\right)\right)^{2}ds=\\
\int_{0}^{1}\left(1+\frac{2}{\theta}+\frac{2}{\theta^{2}}\right)^{2}e^{2\theta(s-1)}+\left(\frac{2s}{\theta}+\frac{2}{\theta^{2}}\right)^{2}-2\left(1+\frac{2}{\theta}+\frac{2}{\theta^{2}}\right)e^{\theta(s-1)}\left(\frac{2s}{\theta}+\frac{2}{\theta^{2}}\right)ds=\\
\left[\left(1+\frac{2}{\theta}+\frac{2}{\theta^{2}}\right)^{2}\frac{e^{2\theta(s-1)}}{2\theta}+\left(\frac{2s}{\theta}+\frac{2}{\theta^{2}}\right)^{3}\frac{\theta}{6}-2\left(1+\frac{2}{\theta}+\frac{2}{\theta^{2}}\right)\left(\frac{2s}{\theta^{2}}\right)e^{\theta(s-1)}\right]_{0}^{1}=\\
\left(1+\frac{2}{\theta}+\frac{2}{\theta^{2}}\right)^{2}\frac{(1-e^{-2\theta})}{2\theta}+\left(\frac{2}{\theta}+\frac{2}{\theta^{2}}\right)^{3}\frac{\theta}{6}-\frac{8}{6\theta^{5}}-\frac{4}{\theta^{2}}\left(1+\frac{2}{\theta}+\frac{2}{\theta^{2}}\right)\,,
\end{multline*}
\begin{multline*}
\E^{\q}(G_{1}G_{2})=\int_{0}^{1}e^{\theta(s-1)}\times\left(\left(1+\frac{1}{\theta}\right)e^{\theta(s-1)}-\frac{1}{\theta}\right)ds\\
=\left[\left(1+\frac{1}{\theta}\right)\frac{e^{2\theta(s-1)}}{2\theta}-\frac{e^{\theta(s-1)}}{\theta^{2}}\right]_{0}^{1}=\left(\frac{1}{2\theta}+\frac{1}{2\theta^{2}}\right)(1-e^{-2\theta})+\frac{e^{-\theta}-1}{\theta^{2}}\,,
\end{multline*}
\begin{multline*}
\E^{\q}(G_{1}G_{3})=\int_{0}^{1}e^{\theta(s-1)}\left(\left(1+\frac{2}{\theta}+\frac{2}{\theta^{2}}\right)e^{\theta(s-1)}-\left(\frac{2s}{\theta}+\frac{2}{\theta^{2}}\right)\right)ds\\
=\left[\left(1+\frac{2}{\theta}+\frac{2}{\theta^{2}}\right)\frac{e^{2\theta(s-1)}}{2\theta}-\frac{2s}{\theta^{2}}e^{\theta(s-1)}\right]_{0}^{1}=\left(\frac{1}{2\theta}+\frac{1}{\theta^{2}}+\frac{1}{\theta^{3}}\right)(1-e^{-2\theta})-\frac{2}{\theta^{2}}\,,
\end{multline*}
\begin{multline*}
\E^{\q}(G_{2}G_{3})=\int_{0}^{1}\left(\left(1+\frac{1}{\theta}\right)e^{\theta(s-1)}-\frac{1}{\theta}\right)\times\left(\left(1+\frac{2}{\theta}+\frac{2}{\theta^{2}}\right)e^{\theta(s-1)}-\left(\frac{2s}{\theta}+\frac{2}{\theta^{2}}\right)\right)ds\\
=\int_{0}^{1}\left(1+\frac{1}{\theta}\right)\left(1+\frac{2}{\theta}+\frac{2}{\theta^{2}}\right)e^{2\theta(s-1)}-\frac{1}{\theta}\left(1+\frac{2}{\theta}+\frac{2}{\theta^{2}}\right)e^{\theta(s-1)}\\
-\left(1+\frac{1}{\theta}\right)e^{\theta(s-1)}\left(\frac{2s}{\theta}+\frac{2}{\theta^{2}}\right)+\frac{1}{\theta}\left(\frac{2s}{\theta}+\frac{2}{\theta^{2}}\right)ds\\
=\left[\left(1+\frac{1}{\theta}\right)\left(1+\frac{2}{\theta}+\frac{2}{\theta^{2}}\right)\frac{e^{2\theta(s-1)}}{2\theta}-\left(\frac{1}{\theta}+\frac{2}{\theta^{2}}+\frac{2}{\theta^{3}}\right)\frac{e^{\theta(s-1)}}{\theta}\right.\\
\left.-\left(1+\frac{1}{\theta}\right)\frac{2s}{\theta^{2}}e^{\theta(s-1)}+\left(\frac{s^{2}}{\theta^{2}}+\frac{2s}{\theta^{3}}\right)\right]_{0}^{1}
\end{multline*}
\end{proof}

\begin{proof}[Proof of Lemma \ref{lem:asymptotics-of-A-B-C}]The coefficient
of $x^{2}$ in $P$ is 
\begin{multline}
-A_{2}(\theta)=-h\frac{\theta}{6}+\frac{\sigma_{1}^{2}}{2}h\theta^{3}\left(-\frac{\alpha}{3}-\frac{\beta}{2}+a\right)^{2}\\
+\frac{\sigma_{2}^{2}}{2}h\theta^{3}\left(b-\frac{\gamma}{2}-\frac{\theta^{2}\alpha\gamma}{2}\sigma_{1}^{2}\left(-\frac{\alpha}{3}-\frac{\beta}{2}+a\right)\right)^{2}+\frac{1}{2}h\theta^{3}c^{2}\,.\label{eq:A2}
\end{multline}
 We compute (using \cite{mathematica} software):
\begin{equation}
\alpha=\frac{1}{\sqrt{2\theta}}+o\left(\frac{1}{\theta^{n}}\right)\,,\,\forall n\geq1\,,\label{eq:dl-alpha}
\end{equation}
\begin{equation}
\cov^{\q}(G_{1},G_{3})=\frac{1}{2\theta}-\frac{1}{\theta^{2}}+\frac{1}{\theta^{3}}+o\left(\frac{1}{\theta^{3}}\right)\,,\label{eq:dl-C13}
\end{equation}
\begin{equation}
\beta=\frac{1}{\sqrt{2\theta}}-\frac{\sqrt{2}}{\theta^{3/2}}+\frac{\sqrt{2}}{\theta^{5/2}}+o\left(\frac{1}{\theta^{5/2}}\right)\,,\label{eq:dl-beta}
\end{equation}
\begin{equation}
\beta^{2}=\frac{1}{2\theta}-\frac{2}{\theta^{2}}+\frac{4}{\theta^{3}}-\frac{4}{\theta^{4}}+o\left(\frac{1}{\theta^{4}}\right)\,,\label{eq:dl-beta2}
\end{equation}
\begin{equation}
\var^{\q}(G_{3})=\frac{1}{2\theta}-\frac{2}{3\theta^{2}}+o\left(\frac{1}{\theta^{4}}\right)\,,\label{eq:dl-V3}
\end{equation}
\begin{equation}
\gamma=\frac{2}{\theta\sqrt{3}}-\frac{\sqrt{3}}{\theta^{2}}+\frac{\sqrt{3}}{4\theta^{3}}+\frac{3\sqrt{3}}{8\theta^{4}}+o\left(\frac{1}{\theta^{4}}\right)\,,\label{eq:dl-gamma}
\end{equation}
\begin{equation}
\cov^{\q}(G_{1},G_{2})=\frac{1}{2\theta}-\frac{1}{2\theta^{2}}+o\left(\frac{1}{\theta^{3}}\right)\,,\label{eq:dl-C12}
\end{equation}
\begin{equation}
\sigma_{1}^{2}=\frac{6}{\theta}+\frac{18}{\theta^{2}}+\frac{18}{\theta^{3}}-\frac{54}{\theta^{4}}+o\left(\frac{1}{\theta^{4}}\right)\,,\label{eq:dl-sigma12}
\end{equation}
\begin{equation}
a=\frac{1}{\sqrt{2\theta}}-\frac{1}{\theta^{3/2}\sqrt{2}}+o\left(\frac{1}{\theta^{5/2}}\right)\,,\label{eq:dl-a}
\end{equation}
\begin{equation}
\var^{\q}(G_{2})=\frac{1}{2\theta}-\frac{3}{2\theta^{3}}+o\left(\frac{1}{\theta^{3}}\right)\,,\label{eq:dl-V2}
\end{equation}
\begin{equation}
b=\frac{\sqrt{3}}{2\theta}-\frac{\sqrt{3}}{4\theta^{2}}-\frac{9\sqrt{3}}{16\theta^{3}}+o\left(\frac{1}{\theta^{3}}\right)\,,\label{eq:dl-b}
\end{equation}

\begin{equation}
c^{2}=\frac{1}{4\theta^{2}}-\frac{5}{4\theta^{3}}+o\left(\frac{1}{\theta^{3}}\right)\,,\label{eq:dl-c2}
\end{equation}
\begin{equation}
\sigma_{2}^{2}=\frac{1}{3\theta^{2}}-\frac{1}{\theta}+2+o\left(\frac{1}{\theta^{3}}\right)\,.\label{eq:dl-sigma22}
\end{equation}
From which we deduce
\[
b-\frac{\gamma}{2}-\frac{\theta^{2}\alpha\gamma}{2}\sigma_{1}^{2}\left(-\frac{\alpha}{3}-\frac{\beta}{2}+a\right)=\frac{\sqrt{3}}{2\theta^{3}}+o\left(\frac{1}{\theta^{3}}\right)\,,
\]
\[
\frac{\sigma_{2}^{2}}{2}h\theta^{3}\left(b-\frac{\gamma}{2}-\frac{\theta^{2}\alpha\gamma}{2}\sigma_{1}^{2}\left(-\frac{\alpha}{3}-\frac{\beta}{2}+a\right)\right)^{2}\underset{\theta\rightarrow+\infty}{\longrightarrow}0\,,
\]
\[
\frac{1}{2}h\theta^{3}c^{2}=h\left(\frac{\theta}{8}-\frac{5}{8}+o(1)\right)\,,
\]
\[
\frac{\sigma_{1}^{2}}{2}h\theta^{3}\left(-\frac{\alpha}{3}-\frac{\beta}{2}+a\right)^{2}=h\left(\frac{\theta}{24}+\frac{1}{8}+o(1)\right)\,,
\]
\begin{equation}
A_{2}(\theta)\underset{\theta\rightarrow+\infty}{\longrightarrow}\frac{h}{2}\,.\label{eq:lim-Cx}
\end{equation}
The coefficient of $z^{2}$ in $P$ is
\begin{equation}
-B_{2}(\theta)=-h\frac{\theta}{6}+\frac{\sigma_{1}^{2}}{2}h\theta^{3}\left(-\frac{\alpha}{6}+\frac{\beta}{2}\right)^{2}+\frac{\sigma_{2}^{2}}{2}h\theta^{3}\left(\frac{\gamma}{2}-\frac{\theta^{2}\alpha\gamma}{2}\sigma_{1}^{2}\left(-\frac{\alpha}{6}+\frac{\beta}{2}\right)\right)^{2}\,.\label{eq:def-Cz2}
\end{equation}
We have:
\[
\frac{\sigma_{2}^{2}}{2}h\theta^{3}\left(\frac{\gamma}{2}-\frac{\theta^{2}\alpha\gamma}{2}\sigma_{1}^{2}\left(-\frac{\alpha}{6}+\frac{\beta}{2}\right)\right)^{2}\underset{\theta\rightarrow+\infty}{\longrightarrow}0\,,
\]
\[
\frac{\sigma_{1}^{2}}{2}h\theta^{3}\left(-\frac{\alpha}{6}+\frac{\beta}{2}\right)^{2}=h\left(\frac{\theta}{6}-\frac{1}{2}+o(1)\right)\,,
\]
\begin{equation}
B_{2}(\theta)\underset{\theta\rightarrow+\infty}{\longrightarrow}\frac{h}{2}\,.\label{eq:lim-Cz}
\end{equation}
The coefficient of $xz$ in $P$ is 
\begin{multline}
C_{1}(\theta)=-h\frac{\theta}{6}+\sigma_{1}^{2}h\theta^{3}\left(-\frac{\alpha}{3}-\frac{\beta}{2}+a\right)\left(-\frac{\alpha}{6}+\frac{\beta}{2}\right)\\
+\sigma_{2}^{2}h\theta^{3}\left(b-\frac{\gamma}{2}-\frac{\theta^{2}\alpha\gamma}{2}\sigma_{1}^{2}\left(-\frac{\alpha}{3}-\frac{\beta}{2}+a\right)\right)\left(\frac{\gamma}{2}-\frac{\theta^{2}\alpha\gamma}{2}\sigma_{1}^{2}\left(-\frac{\alpha}{6}+\frac{\beta}{2}\right)\right)\,\label{eq:def-C1}
\end{multline}
(it does not depend on $y_{0:\tau}$). We have:
\[
\sigma_{2}^{2}h\theta^{3}\left(b-\frac{\gamma}{2}-\frac{\theta^{2}\alpha\gamma}{2}\sigma_{1}^{2}\left(-\frac{\alpha}{3}-\frac{\beta}{2}+a\right)\right)\left(\frac{\gamma}{2}-\frac{\theta^{2}\alpha\gamma}{2}\sigma_{1}^{2}\left(-\frac{\alpha}{6}+\frac{\beta}{2}\right)\right)\underset{\theta\rightarrow+\infty}{\longrightarrow}0\,,
\]
\[
\sigma_{1}^{2}h\theta^{3}\left(-\frac{\alpha}{3}-\frac{\beta}{2}+a\right)\left(-\frac{\alpha}{6}+\frac{\beta}{2}\right)=h\left(\frac{\theta}{6}-\frac{3}{2\theta}+o\left(\frac{1}{\theta}\right)\right)\,,
\]
\begin{equation}
C_{1}(\theta)=\frac{3h}{2\theta}+o\left(\frac{1}{\theta}\right)\,.\label{eq:lim-Cxz}
\end{equation}
\end{proof}

We need the following Lemma before going into the proof of Lemma \ref{lem:variation-const}.
\begin{lem}
\label{lem:variation-Y}For all $k\in\N,$ $s\in[k,k+1]$ 
\[
\left|Y_{\tau s}-\int_{k}^{k+1}Y_{\tau u}du-Y_{\tau(s+1)}+\int_{k}^{k+1}Y_{\tau(u+1)}du\right|\preceq h\tau^{2}M+h\tau\mathcal{V}_{k\tau,(k+1)\tau}+\mathcal{W}_{k\tau,(k+1)\tau}\,,
\]
\[
\left|\int_{0}^{1}\frac{e^{-\theta}\sinh(\theta s)}{\theta}dY_{\tau k+\tau s}-\int_{0}^{1}\frac{e^{-\theta}\sinh(\theta s)}{\theta}dY_{\tau(k+1)+\tau s}\right|\preceq\frac{M\tau^ {}+\mathcal{V}_{k\tau,(k+1)\tau}+\mathcal{W}_{k\tau,(k+1)\tau}}{\theta}\,.
\]
And, for all $s\in[0,1]$,
\[
\left|Y_{\tau s}-\int_{0}^{1}Y_{\tau u}du\right|\preceq h\tau^{2}M+h\tau\mathcal{V}_{0,\tau}+\mathcal{W}_{0,\tau}\,,
\]
\[
\left|\int_{0}^{1}\frac{e^{-\theta}\sinh(\theta s)}{\theta}dY_{\tau s}\right|\preceq\frac{M\tau+\mathcal{V}_{0,\tau}+\mathcal{W}_{0,\tau}}{\theta}\,.
\]
\end{lem}

\begin{proof}
We write the proof only for the first two formulas. We have, for all
$k\in\N$, $s\in[k,k+1]$,
\begin{multline*}
\left|Y_{\tau s}-\int_{k}^{k+1}Y_{\tau u}du-Y_{\tau(s+1)}+\int_{k}^{k+1}Y_{\tau(u+1)}du\right|\\
=\left|-\int_{\tau s}^{\tau(s+1)}hX_{u}du+W_{\tau s}-W_{\tau(s+1)}-\int_{k}^{k+1}\left(-h\int_{\tau u}^{\tau(u+1)}X_{v}dv+W_{\tau u}-W_{\tau(u+1)}\right)du\right|\\
\preceq\left|\int_{k}^{k+1}h\left(\int_{\tau s}^{\tau(s+1)}X_{v}dv-\int_{\tau u}^{\tau(u+1)}X_{v}dv\right)du\right|+\mathcal{\mathcal{W}}_{k\tau,(k+2)\tau}\\
=\left|h\int_{k}^{k+1}\left(\int_{\tau s}^{\tau(s+1)}X_{v}-X_{v+\tau(u-s)}dv\right)du\right|+\mathcal{\mathcal{W}}_{k\tau,(k+2)\tau}\\
=\left|h\int_{k}^{k+1}\left(\int_{\tau s}^{\tau(s+1)}\int_{v+\tau(u-s)}^{v}f(X_{t})dt+V_{v}-V_{v+\tau(u-s)}dvdu\right)\right|+\mathcal{\mathcal{W}}_{k\tau,(k+2)\tau}\\
\leq h\tau^{2}M+h\tau\mathcal{V}_{k\tau,(k+2)\tau}+\mathcal{\mathcal{W}}_{k\tau,(k+2)\tau}\,,
\end{multline*}
and (using integration by parts)
\begin{multline*}
\left|\int_{0}^{1}\frac{e^{-\theta}\sinh(\theta s)}{\theta}dY_{\tau k+\tau s}-\int_{0}^{1}\frac{e^{-\theta}\sinh(\theta s)}{\theta}dY_{\tau(k+1)+\tau s}\right|\\
=\left|\int_{0}^{\tau}\frac{e^{-\theta}\sinh(hs)}{\theta}(h(X_{\tau k+s}-X_{\tau(k+1)+s})ds+dW_{\tau k+s}-dW_{\tau(k+1)+s})\right|\\
\leq\left|\int_{0}^{\tau}\frac{e^{-\theta}\sinh(hs)}{\theta}h(X_{\tau k+s}-X_{\tau(k+1)+s})ds\right|\\
+\left|\frac{e^{-\theta}\sinh(\theta)}{\theta}(W_{\tau(k+1)}-W_{\tau(k+2)})\right|\\
+\left|\int_{0}^{\tau}(W_{\tau k+s}-W_{\tau(k+1)+s})\frac{e^{-\theta}\cosh(hs)}{\tau}ds\right|\\
\preceq\int_{0}^{\tau}\frac{e^{-\theta}\sinh(hs)}{\theta}h(\tau M+\mathcal{V}_{k\tau,(k+2)\tau})ds+\frac{e^{-\theta}\sinh(\theta)}{\theta}\mathcal{W}_{k\tau,(k+2)\tau}+\int_{0}^{\tau}\mathcal{W}_{k\tau,(k+2)\tau}\frac{e^{-\theta}\cosh(hs)}{\tau}ds\\
\preceq\frac{h(M\tau+\mathcal{V}_{k\tau,(k+2)\tau})}{h\theta}+\frac{\mathcal{W}_{k\tau,(k+2)\tau}}{\theta}\,.
\end{multline*}
\end{proof}
\begin{proof}[Proof of Lemma \ref{lem:variation-const}]We write the
proof in the case $k=0$. From (\ref{eq:psi-chapeau-01}), (\ref{eq:psi-chapeau-05}),
we deduce
\begin{multline}
B_{1}(Y_{0:\tau},\theta)=-\sigma_{1}^{2}h\theta^{2}\left(-\frac{\alpha}{6}+\frac{\beta}{2}\right)\lambda_{1}(Y_{0:\tau})+h\int_{0}^{1}sdY_{\tau s}\\
+\sigma_{2}^{2}h\theta^{2}\left(\frac{\gamma}{2}-\frac{\theta^{2}\alpha\gamma\sigma_{1}^{2}}{2}\left(-\frac{\alpha}{6}+\frac{\beta}{2}\right)\right)\left(-\lambda_{2}(Y_{0:\tau})+\frac{\theta^{2}\alpha\gamma\sigma_{1}^{2}}{2}\lambda_{1}(Y_{0:\tau})\right)\,.\label{eq:def-Cz}
\end{multline}
For further use, we also write the formula for $A_{1}(Y_{0:\tau},\theta)$:
\begin{multline}
A_{1}(Y_{0:\tau},\theta)=h\int_{0}^{1}(1-s)dY_{s\tau}+\sigma_{1}^{2}h\theta^{2}\left(\frac{\alpha}{3}+\frac{\beta}{2}-a\right)\lambda_{1}(Y_{0:\tau})\\
+\sigma_{2}^{2}h\theta^{2}\left(b-\frac{\gamma}{2}-\frac{\theta^{2}\alpha\gamma\sigma_{1}^{2}}{2}\left(-\frac{\alpha}{3}-\frac{\beta}{2}+a\right)\right)\left(-\lambda_{2}(Y_{0:\tau})+\frac{\theta^{2}\alpha\gamma\sigma_{1}^{2}}{2}\lambda_{1}(Y_{0:\tau})\right)-h\theta^{2}\lambda_{3}(Y_{0:\tau})c.\label{eq:def-Cx}
\end{multline}
 We have to remember here that $\lambda_{1}$, $\lambda_{2}$, $\lambda_{3}$
are functions of $y_{0:\tau}$. So we might write $\lambda_{1}(y_{0:\tau})$,
\ldots{} to stress this dependency (and the same goes for other quantities).
From Lemmas \ref{lem:representation-O-U}, \ref{lem:integral-O-U},
we get ($g_{1}$, $g_{2}$ defined below)
\begin{eqnarray}
\cov^{\q}(G_{1},G_{4})(Y_{0:\tau}) & = & \int_{0}^{1}e^{\theta(s-1)}\times\left(g_{1}(s)-\theta e^{\theta s}\int_{s}^{1}e^{-\theta u}g_{1}(u)du\right)ds\nonumber \\
 & = & \int_{0}^{1}g_{1}(s)e^{\theta(s-1)}ds-\int_{0}^{1}e^{-\theta u}g_{1}(u)\int_{0}^{u}\theta e^{2\theta s-\theta}dsdu\nonumber \\
 & = & \int_{0}^{1}g_{1}(s)e^{\theta(s-1)}ds-\int_{0}^{1}e^{-\theta u}g_{1}(u)\left(\frac{e^{2\theta u}-1}{2}\right)e^{-\theta}du\nonumber \\
 & = & \int_{0}^{1}g_{1}(s)e^{-\theta}\cosh(\theta s)ds\,,\label{eq:cov14-a}
\end{eqnarray}
\begin{equation}
\cov^{\q}(G_{1},G_{4})(Y_{\tau:2\tau})=\int_{0}^{1}g_{2}(s)e^{-\theta}\cosh(\theta s)ds\,,\label{eq:cov14-b}
\end{equation}
with 
\[
g_{1}(s)=Y_{\tau s}-\int_{0}^{1}Y_{\tau u}du\,,\,g_{2}(s)=Y_{\tau(s+1)}-\int_{0}^{1}Y_{\tau(u+1)}du\,,
\]
and
\begin{multline}
\cov(G_{3},G_{4})(Y_{0:\tau})=\int_{0}^{1}\left(\left(1+\frac{2}{\theta}+\frac{2}{\theta^{2}}\right)e^{\theta(s-1)}-\left(\frac{2s}{\theta}+\frac{2}{\theta^{2}}\right)\right)\\
\times\left(g_{1}(s)-\theta e^{\theta s}\int_{s}^{1}e^{-\theta u}g_{1}(u)du\right)ds\\
=\left(1+\frac{2}{\theta}+\frac{2}{\theta^{2}}\right)\cov(G_{1},G_{4})(Y_{0:\tau})-\int_{0}^{1}\frac{2s}{\theta}g_{1}(s)ds\\
+\int_{0}^{1}e^{-\theta u}g_{1}(u)\int_{0}^{u}\left(2s+\frac{2}{\theta}\right)e^{\theta s}dsdu\\
=\left(1+\frac{2}{\theta}+\frac{2}{\theta^{2}}\right)\cov(G_{1},G_{4})(Y_{0:\tau})-\int_{0}^{1}\frac{2s}{\theta}g_{1}(s)ds\\
+\int_{0}^{1}e^{-\theta u}g_{1}(u)\frac{2u}{\theta}e^{\theta u}du\\
=\left(1+\frac{2}{\theta}+\frac{2}{\theta^{2}}\right)\cov(G_{1},G_{4})(Y_{0:\tau})\,,\label{eq:cov34-a}
\end{multline}
\begin{equation}
\cov(G_{3},G_{4})(Y_{\tau:2\tau})=\left(1+\frac{2}{\theta}+\frac{2}{\theta^{2}}\right)\cov(G_{1},G_{4})(Y_{\tau:2\tau})\,.\label{eq:cov34-b}
\end{equation}

From (\ref{eq:def-lambdas}), (\ref{eq:dl-alpha})-(\ref{eq:dl-sigma22}),
we deduce (using again \cite{mathematica})
\[
-\sigma_{1}^{2}h\theta^{2}\left(-\frac{\alpha}{6}+\frac{\beta}{2}\right)\lambda_{1}(Y_{0:\tau})=-h(2\theta+O(1))\cov^{\q}(G_{1},G_{4})(Y_{0:\tau})\,,
\]
\[
\sigma_{2}^{2}h\theta^{2}\left(\frac{\gamma}{2}-\frac{\theta^{2}\alpha\gamma\sigma_{1}^{2}}{2}\left(-\frac{\alpha}{6}+\frac{\beta}{2}\right)\right)=hO\left(\theta\right)\,,
\]
\begin{eqnarray}
-\lambda_{2}(Y_{0:\tau})+\frac{\theta^{2}\alpha\gamma\sigma_{1}^{2}}{2}\lambda_{1}(Y_{0:\tau}) & = & \cov^{\q}(G_{1},G_{4})(Y_{0:\tau})\left(-\frac{1}{\gamma}\left(1+\frac{2}{\theta}+\frac{2}{\theta^{2}}-\frac{\beta}{\alpha}\right)+\frac{\beta}{\alpha\gamma}+\frac{\theta^{2}\gamma\sigma_{1}^{2}}{2}\right)\nonumber \\
 & = & \cov^{\q}(G_{1},G_{4})(Y_{0:\tau})\times O\left(\frac{1}{\theta}\right)\,.\label{eq:Cx-part-02}
\end{eqnarray}
So we get
\begin{multline}
-\sigma_{1}^{2}h\theta^{2}\left(-\frac{\alpha}{6}+\frac{\beta}{2}\right)\lambda_{1}(Y_{0:\tau})+\sigma_{2}^{2}h\theta^{2}\left(\frac{\gamma}{2}-\frac{\theta^{2}\alpha\gamma\sigma_{1}^{2}}{2}\left(-\frac{\alpha}{6}+\frac{\beta}{2}\right)\right)\left(-\lambda_{2}(Y_{0:\tau})+\frac{\theta^{2}\alpha\gamma\sigma_{1}^{2}}{2}\lambda_{1}(Y_{0:\tau})\right)\\
=-2h(\theta+O(1))\cov^{\q}(G_{1},G_{4})(Y_{0:\tau})\,.\label{eq:borne-B-01}
\end{multline}
We have
\begin{eqnarray}
\cov^{\q}(G_{1},G_{4})(Y_{0:\tau}) & = & \int_{0}^{1}\left(Y_{\tau s}-\int_{0}^{1}Y_{\tau u}du\right)e^{-\theta}\cosh(\theta s)ds\nonumber \\
 & = & \left(Y_{\tau}-\int_{0}^{1}Y_{\tau u}du\right)\frac{e^{-\theta}\sinh(\theta)}{\theta}-\int_{0}^{1}\frac{e^{-\theta}\sinh(\theta s)}{\theta}dY_{\tau s}\nonumber \\
 & = & \int_{0}^{1}sdY_{\tau s}\times\frac{e^{-\theta}\sinh(\theta)}{\theta}-\int_{0}^{1}\frac{e^{-\theta}\sinh(\theta s)}{\theta}dY_{\tau s}\,,\label{eq:borne-B-02}
\end{eqnarray}
and so
\begin{multline}
-2h(\theta+O(1))\cov^{\q}(G_{1},G_{4})(Y_{0:\tau})+h\int_{0}^{1}sdY_{\tau s}=-2h(\theta+O(1))\left(\int_{0}^{1}sdY_{\tau s}\right)\left(\frac{1}{2\theta}-\frac{e^{-2\theta}}{2\theta}\right)\\
+h\int_{0}^{1}sdY_{\tau s}+h(\theta+O(1))\times2\int_{0}^{1}\frac{e^{-\theta}\sinh(\theta s)}{\theta}dY_{\tau s}\\
=h\left(\int_{0}^{1}sdY_{\tau s}\right)\times O\left(\frac{1}{\theta}\right)-h(\theta+O(1))\times2\int_{0}^{1}\frac{e^{-\theta}\sinh(\theta s)}{\theta}dY_{\tau s}\label{eq:borne-B-03}
\end{multline}
And so, using Lemma \ref{lem:variation-Y}, Equations (\ref{eq:borne-B-01}),
(\ref{eq:borne-B-02}), (\ref{eq:borne-B-03}) (as similar formulas
of the ones above are valid if we replace $Y_{0:\tau}$ by $Y_{0:2\tau}$),
we get
\begin{eqnarray*}
|B_{1}(Y_{0:\tau},\theta)-B_{1}(Y_{\tau:2\tau},\theta)| & \preceq & \frac{1}{\tau}(h\tau^{2}M+h\tau\mathcal{V}_{0,2\tau}+\mathcal{W}_{0,2\tau})+h\theta\left(\frac{M\tau+\mathcal{V}_{0,2\tau}+\mathcal{W}_{0,2\tau}}{\theta}\right)\\
 & \preceq & Mh\tau+h\mathcal{V}_{0,2\tau}+(h+\frac{1}{\tau})\mathcal{W}_{0,2\tau}\,.
\end{eqnarray*}
\end{proof}

\subsection{Technical Lemmas used in Section \ref{sec:Robust-approximation-of}}
\begin{lem}
\label{lem:lambda-integral}We have $\lambda_{1}=\int_{0}^{\tau}f_{1}(s)dy_{\tau s}$
for some deterministic function $f_{1}$. And the same is true for
$\lambda_{2}$, $\lambda_{3}$. 
\end{lem}

\begin{proof}
We write the proof only for $\lambda_{1}$. Using Equation (\ref{eq:def-lambdas})
and Lemmas \ref{lem:representation-O-U}, \ref{lem:integral-O-U}
and integrations by parts, we get
\begin{eqnarray*}
\alpha\lambda_{1} & = & \cov^{\q}(G_{1},G_{4})\\
 & = & \cov^{\q}\left(B_{1},\int_{0}^{1}y_{\tau s}-\left(\int_{0}^{1}y_{\tau u}du\right)dB_{s}\right)\\
 & = & \cov^{\q}\left(\int_{0}^{1}e^{\theta(s-1)}d\beta_{s},\int_{0}^{1}y_{\tau s}-\left(\int_{0}^{1}y_{\tau u}du\right)-\theta e^{\theta s}\int_{s}^{1}e^{-\theta u}\left(y_{\tau u}-\left(\int_{0}^{1}y_{\tau v}dv\right)\right)dud\beta_{s}\right)\\
 & = & \int_{0}^{1}e^{\theta(s-1)}\left(y_{\tau s}-\left(\int_{0}^{1}y_{\tau u}du\right)-\theta e^{\theta s}\int_{s}^{1}e^{-\theta u}\left(y_{\tau u}-\left(\int_{0}^{1}y_{\tau v}dv\right)\right)du\right)ds\\
 & = & \int_{0}^{1}e^{\theta(s-1)}y_{\tau s}ds-\int_{0}^{1}e^{\theta(s-1)}ds\times\int_{0}^{1}y_{\tau u}du-\int_{0}^{1}e^{-\theta u}y_{\tau u}\int_{u}^{1}\theta e^{\theta(2s-1)}dsdu\\
 &  & \,\,\,\,\,\,\,\,-\int_{0}^{1}\theta e^{\theta(2s-1)}\frac{(e^{-\theta s}-e^{-\theta})}{\theta}ds\times\int_{0}^{1}y_{\tau u}du\\
\text{} & = & \frac{y_{\tau}}{\theta}-\int_{0}^{1}\frac{e^{\theta(s-1)}}{\theta}dy_{\tau s}-\frac{(1-e^{-\theta})}{\theta}\left(y_{\tau}-\int_{0}^{1}udy_{\tau u}\right)-\int_{0}^{1}y_{\tau u}\sinh(\theta(1-u))du\\
 &  & \,\,\,\,\,\,\,-\left(\frac{1-e^{-\theta}}{\theta}-\frac{(1-e^{-2\theta})}{2\theta}\right)\left(y_{\tau}-\int_{0}^{1}udy_{\tau u}\right)\\
 & = & \frac{y_{\tau}}{\theta}-\int_{0}^{1}\frac{e^{\theta(s-1)}}{\theta}dy_{\tau s}-\frac{(1-e^{-\theta})}{\theta}\left(y_{\tau}-\int_{0}^{1}udy_{\tau u}\right)+\frac{y_{\tau}}{\theta}-\int_{0}^{1}\frac{\cosh(\theta(1-u))}{\theta}dy_{\tau u}\\
 &  & \,\,\,\,\,\,\,-\left(\frac{1-e^{-\theta}}{\theta}-\frac{(1-e^{-2\theta})}{2\theta}\right)\left(y_{\tau}-\int_{0}^{1}udy_{\tau u}\right)\,.
\end{eqnarray*}
\end{proof}
\begin{lem}
\label{lem:A-as-integral}We can write 
\[
A_{1}^{Y_{0:\tau}}=\int_{0}^{\tau}(f_{1}(s)dW_{s}+f_{2}(s))dX_{s}
\]
for some deterministic functions $f_{1}$, $f_{2}$.
\end{lem}

\begin{proof}
From Lemma \ref{lem:lambda-integral} and Equation (\ref{eq:def-Cx}),
we know there exists a deterministic function $g_{1}$ such that 
\[
A_{1}^{Y_{0:\tau}}=\int_{0}^{1}g_{1}(s)dY_{\tau s}\,.
\]
So, integrating by parts, we get (where $G$ is the antiderivative
of $g_{1}$ such that $G(0)=0$)
\begin{eqnarray*}
A_{1}^{Y_{0:\tau}} & = & \int_{0}^{1}g_{1}(s)(hX_{\tau s}\tau ds+dW_{\tau s})\\
 & = & G(1)hX_{\tau}\tau-\int_{0}^{1}G(s)h\tau dX_{\tau s}+\int_{0}^{1}g_{1}(s)dW_{\tau s}\\
 & = & \int_{0}^{\tau}\left(G(1)h-G\left(\frac{s}{\tau}\right)h\right))dX_{s}+\int_{0}^{\tau}g_{1}\left(\frac{s}{\tau}\right)dW_{s}\,.
\end{eqnarray*}
\end{proof}
\begin{lem}
\label{lem:tech-03}For any measurable $\mathcal{B}$, subset of $\R^{2}$
and $\Phi$ a Gaussian density,
\begin{multline*}
\int_{(t_{1},t_{2})\in\mathcal{B}}\Phi(t_{1},t_{2})\Psi_{1}(t_{1},t_{2})dt_{1}dt_{2}\\
\leq\int_{\mathcal{B}}\Phi(t_{1},t_{2})C_{1}'(h,\tau)\exp\left(2M|(1,-1)\kappa^{-1}(t_{1},t_{2})^{T}|+\tau(2M+M^{2})\right)dt_{1}dt_{2}
\end{multline*}
(the constant $C_{1}'(h,\tau)$ coming from Lemma \ref{lem:borne-Psi_1}).
\end{lem}

\begin{proof}
For any $\epsilon'<\epsilon$ ($\epsilon$ comes from Lemma \ref{lem:borne-Psi_1}),
we can write $\mathcal{B}=\sqcup_{i\in I}\mathcal{B}_{i}$, where
$I\subset\N,$ $\sqcup$ means that this is a partition, and for all
$i$, the set $\mathcal{B}_{i}$ is a subset of $B(2(x_{i},z_{i})\kappa,\epsilon')$
(for some $x_{i}$, $z_{i}$). For all $i$, for almost all $(t_{1},t_{2})$
in $B(2(x_{i},z_{i})\kappa,\epsilon')$, we have, by Lemma \ref{lem:borne-Psi_1},
\begin{eqnarray*}
\Psi_{1}(t_{1},t_{2}) & \leq & C_{1}'(h,\tau)\exp(4M|x_{i}-z_{i}|+\tau(2M+M^{2}))\\
 & = & C_{1}'(h,\tau)\exp(4M|(x_{i},z_{i})(1,-1)^{T}|+\tau(2M+M^{2}))\\
 & = & C_{1}'(h,\tau)\exp(4M|2(x_{i},z_{i})\kappa\frac{1}{2}\kappa^{-1}(1,-1)^{T}|+\tau(2M+M^{2}))\\
 & \leq & C_{1}'(h,\tau)\exp(4M|(t_{1},t_{2})\frac{1}{2}\kappa^{-1}(1,-1)^{T}|+C'\epsilon'+\tau(2M+M^{2}))\,,
\end{eqnarray*}
for some constant $C'$. So we have
\begin{multline*}
\int_{(t_{1},t_{2})\in\mathcal{B}}\Phi(t_{1},t_{2})\Psi_{1}(t_{1},t_{2})dt_{1}dt_{2}=\sum_{i\in I}\int_{(t_{1},t_{2})\in\mathcal{B}_{i}}\Phi(t_{1},t_{2})\Psi_{1}(t_{1},t_{2})dt_{1}dt_{2}\\
\leq\sum_{i\in I}\int_{(t_{1},t_{2})\in\mathcal{B}_{i}}\Phi(t_{1},t_{2})C_{1}'(h,\tau)\exp(4M|(t_{1},t_{2})\frac{1}{2}\kappa^{-1}(1,-1)^{T}|+C'\epsilon'+\tau(2M+M^{2}))dt_{1}dt_{2}\\
=\int_{(t_{1},t_{2})\in\mathcal{B}}\Phi(t_{1},t_{2})C_{1}'(h,\tau)\exp(4M|(t_{1},t_{2})\frac{1}{2}\kappa^{-1}(1,-1)^{T}|+C'\epsilon'+\tau(2M+M^{2}))dt_{1}dt_{2}\,.
\end{multline*}
And by taking $\epsilon'\rightarrow0$, we get the desired result.
\end{proof}

\subsection{Technical Lemma used in Section \ref{sec:Stability-results}}
\begin{lem}
\label{lem:comparaison-queues}Suppose that $f_{1}$ and $f_{2}$
are two probability densities on $\R_{+}$ and that $\varphi$ is
a nondecreasing function from $\R_{+}$ to $\R_{+}$. Suppose the
distribution functions $F_{1}$ an $F_{2}$ (associated respectively
to $f_{1}$, $f_{2}$) are such that 
\[
1-F_{1}(x)\leq K(1-F_{2}(x))\text{ , for all }x\geq0\,,
\]
for some constant $K\geq1$. Then, for $X_{1}$ a random variable
of density $f_{1}$ and $X_{2}$ a random variable of density $f_{2}$,
we have
\[
\E(\varphi(X_{1}))\leq K\E(\varphi(X_{2}))\,.
\]
\end{lem}

\begin{proof}
Let $U$ be a random variable of uniform law on $[0,1]$. Let $m$
be such that $F_{2}(m)=1-1/K$. Let $(K(1-F_{2}))^{-1}$ be the inverse
of $\widehat{F}_{2}\,:\,x\in[m,+\infty[\mapsto K(1-F_{2}(x))\in[0,1]$.
We have:
\begin{eqnarray*}
\E(\varphi(X_{1})) & = & \E(\varphi((1-F_{1})^{-1}(U))\\
 & \leq & \E(\varphi((K(1-F_{2}))^{-1}(U)))\\
 & = & \E(\varphi(X_{2})|X_{2}\geq m)\\
 & = & \E(\varphi(X_{2})\1_{[m,+\infty[}(X_{2}))\times K\\
 & \leq & K\E(\varphi(X_{2}))\,.
\end{eqnarray*}
\end{proof}
\textbf{THANKS}~: The authors would like to thank the following colleagues,
whose help was greatly appreciated: Dan Crisan, François Delarue.\bibliographystyle{alpha}
\bibliography{bib-article-bui}

\end{document}